\documentclass[twoside,11pt]{article}

\usepackage{blindtext}

\usepackage{jmlr2e}
\usepackage{color,xcolor}
\usepackage{amsfonts}
\usepackage{amsmath}
\usepackage{amssymb}
\usepackage{dsfont}
\usepackage{mathrsfs}
\usepackage{ulem}

\newcommand{\Hk}{\mathcal{H}}
\newcommand{\HSk}{HS(\mathcal{H})}

\newcommand{\nb}{n}
\newcommand{\ind}[1]{\mathds{1}_{#1}}%Indicatrice
\newcommand{\N}{\mathbb{N}}%Entiers naturels
\newcommand{\R}{\mathbb{R}}%Réels
\newcommand{\PP}{\mathbb{P}}%Probabilité
\newcommand{\PPO}{\mathbb{P} }
\newcommand{\EE}{\mathbb{E}}%Espérance
\definecolor{vert_question}{rgb}{0.0, 0.5, 0.0}

\newcommand{\SW}{\Sigma}
\newcommand{\ST}{\Sigma_{T}}
\newcommand{\hS}{\widehat{\Sigma}}
\newcommand{\hST}{\widehat{\Sigma}_{T}}

\newcommand{\A}[1]{$\mathtt{A_#1}$}

\newcommand{\ti}[1]{\textit{#1}} %pour mettre en italique

\newcommand{\trueq}[1]{q_{1-\alpha}(#1) }
\newcommand{\qmaj}{\mathcal{Q}}
\newcommand{\qmajbis}{\mathscr{Q}}
\newcommand{\qmajfin}{Q}
\newcommand{\hatqmajfin}{\widehat{Q}} 
\renewcommand{\P}{\mathbb{P}}
\DeclareMathOperator{\sign}{sign}

\usepackage{lastpage}
\ShortHeadings{Spectrally truncated normalized MMD}{Lacroix, Michel, Picard and Rivoirard}
\firstpageno{1}

\begin{document}

\title{Non-asymptotic two-sample kernel testing with the spectrally truncated normalized MMD}

\author{\name Perrine Lacroix \email perrine.lacroix@univ-nantes.fr \\
       \addr Nantes Université, CNRS, Laboratoire de Mathématiques Jean Leray, LMJL, \\
        UMR 6629, F-44000 Nantes, France \\
       \addr Laboratoire de Biologie et Modélisation de la Cellule \\
       Ecole Normale Supérieure de Lyon, CNRS, UMR5239, Université Claude Bernard Lyon 1, \\
       Lyon, France 
       \AND
       \name Bertrand Michel \email bertrand.michel@ec-nantes.fr \\
       \addr  Nantes Université, Ecole Centrale Nantes, CNRS, Laboratoire de Mathématiques Jean Leray, LMJL, \\
	UMR 6629, F-44000 Nantes, France
       \AND
       \name Franck Picard \email franck.picard@ens-lyon.fr \\
       \addr Laboratoire de Biologie et Modélisation de la Cellule \\
       Ecole Normale Supérieure de Lyon, CNRS, UMR5239, Université Claude Bernard Lyon 1, \\
       Lyon, France
       \AND
       \name Vincent Rivoirard  \email Vincent.Rivoirard@dauphine.fr \\
       \addr CEREMADE, CNRS, Universit\'e Paris-Dauphine,\\
        Universit\'e PSL, 75016 Paris, FRANCE\\
       \addr Universit\'e Paris-Saclay, CNRS, Inria, LMO\\ 
        91405, Orsay, FRANCE
       }
       
\editor{My editor}

\maketitle

\begin{abstract}%   <- trailing '%' for backward compatibility of .sty file
Kernel methods provide a flexible and powerful framework for nonparametric statistical testing by embedding probability distributions into a reproducing kernel Hilbert space (RKHS). In this work, we study the kernel two-sample testing problem and focus on a normalized version of the Maximum Mean Discrepancy (MMD) as a test statistic, which scales the discrepancy by the within-group covariance operator to account for data variability. This normalization has been shown to improve test power in both theoretical and empirical settings. Because this normalization requires regularization, we study the non-asymptotic properties of the spectrally truncated normalized MMD (st-nMMD) and derive an exponential upper bound under the null hypothesis. Thanks to this result we propose a sharp and explicit upper bound for the corresponding non-asymptotic quantile, along with a data-adaptive estimator. We further propose an algorithm to tune the hyperparameters involved in the quantile estimation, including the truncation level, without requiring data splitting. We demonstrate the performance of the st-nMMD through numerical experiments under both the null and alternative hypotheses.
\end{abstract}


%%%%%%%%%%%%%%%%%%%%%%%%
%%%%%%%%%%%%%%%%%%%%%%%%
\section{Introduction}\label{sec:intro}
The increasing availability of high-throughput technologies has led to the widespread generation of complex, high-dimensional data, intensifying the need for flexible non-parametric methods with strong statistical guarantees. In this context, kernel-based hypothesis testing has emerged as a powerful framework, combining nonlinear distributional embeddings with linear methods and principled inference procedures \citep{muandet2017kernel}. These methods have shown strong theoretical and empirical performance in classical problems such as two-sample and independence testing, and have found successful applications across diverse scientific domains \citep{fromont2012kernels,zhang2011kernel, ozier2024kernel}.

In this article, we study kernel-based two-sample testing for equality of distributions, where independent samples are available from each of the two unknown distributions. A seminal contribution is due to \cite{gretton2006kernel,gretton2012kernel}, who proposed the Maximum Mean Discrepancy (MMD) as a test statistic. The approach embeds probability measures into a Reproducing Kernel Hilbert Space (RKHS) via their mean embeddings, and tests the null hypothesis by evaluating the norm of the difference between these embeddings. The main statistical challenge lies in controlling the fluctuations of the test statistic under the null hypothesis, which requires characterizing its null distribution or, at least, accurately estimating its quantiles. Early approaches relied on asymptotic approximations, under which the MMD statistic converges to an infinite weighted sum of chi-squared random variables. Although the resulting test is consistent, the associated asymptotic distribution involves weights that depend on unknown population quantities, directly related to the variability of the embedded observations, rendering this approach impractical for direct use. Consequently, data-splitting procedures are commonly used in practice to estimate the quantiles of the MMD test statistic \citep{fromont2012kernels,Chwialkowski2014wild,schrab2023mmd}.

Beyond statistical calibration issues, recent theoretical work has demonstrated that relying solely on mean embeddings is insufficient to capture all information relevant for two-sample testing, resulting in minimax suboptimality of MMD-based tests \citep{hagrass2024spectral,li2024optimality,schrab2023mmd}. Revisiting and bringing up to date early developments in kernel-based testing \citep{eric2007testing}, recent work has introduced normalized variants of the MMD that account for the residual variability of the embedded data directly in the definition of the test statistic. It corresponds to the functional and kernelized version of the $T^2$-Hotelling test~\citep{hotelling1931generalization}, which is itself the multivariate version of the Student test. Explicitly accounting for the variability of the embedded data within the test statistic facilitates the analysis of its null distribution and has been shown to yield powerful testing procedures. Moreover, from a methodological perspective, this strategy aligns with the kernelized version of Hotelling’s $T^2$ statistic \citep{lehmann1986testing}, with the resulting test statistic corresponding to a kernelized Mahalanobis distance, as used in non-linear discriminant analysis.

Normalizing the MMD requires the inversion of the within-group covariance operator in the RKHS, which is inherently singular and therefore necessitates regularization. Current strategies are based on the ridge-regularized inverse, leading to tests for which the asymptotic null distribution, consistency, and empirical power against fixed alternatives were early established \citep{eric2007testing}. In particular, they get empirically a more powerful test with respect to the MMD. 
Recent non-asymptotic developments have shown that combining mean embeddings with the empirical within-group covariance operator is sufficient to capture the information needed to discriminate the null from the alternative hypothesis, yielding non-asymptotic guarantees and minimax optimality over a class of alternatives \citep{hagrass2024spectral}. However, the null distribution of the ridge-regularized, normalized MMD statistic still involves a complex form, characterized by infinite weighted sums with unknown weights. As a result, data-splitting procedures are again required in practice, leading to an even higher computational cost than that of the standard MMD, due to the need for  spectral decomposition of the within-group covariance operator \citep{hagrass2024spectral}. 

However, an alternative regularization strategy based on spectral truncation is possible. Although previously discussed, it has not been fully explored in existing work \citep{eric2007testing,hagrass2024spectral}, despite several appealing advantages. From a practical perspective, the eigenvectors of the  within-group covariance operator define a basis that naturally induces discriminant directions, enabling a tight integration of statistical testing with rich nonlinear data representations \citep{ozier2024kernel}. These representations capture and visualize what potentially differentiates the two distributions. To give an example in single-cell data analysis, \cite{ozier2024kernel} reveal some cell-types with specific roles during a reversion process in a differentiation mechanism. Such a representational viewpoint is largely absent from existing kernel testing frameworks, which typically provide only binary decisions and fail to exploit the full representation potential of kernel methods. From a theoretical standpoint, spectral truncation also offers notable advantages, as the asymptotic null distribution of the test statistic reduces to a chi-squared variable with parameter equal to the number of retained principal directions and then does not depend on unknown hyperparemeters. This approach is therefore particularly appealing due to its simplicity and computational efficiency. However, as we will show, this approach does not yield a properly calibrated test in the non-asymptotic regime. 

In this work, we propose a non-asymptotic kernel-based testing procedure that is based on a truncated, spectrally regularized normalized MMD. For convenience, we refer to our procedure as \textit{st-nMMD} (spectrally truncated normalized MMD). Our motivation is double: to propose a test that is properly calibrated from the theoretical point of view, along with data-adaptive quantile that performs well in practice without data-splitting \citep{hagrass2024spectral}.

%%%%%%%%%%%%%
\subsection{Detailed contributions}
Our first contribution is to derive a quantile of the st-nMMD statistic under the null hypothesis for the kernel-based two-sample testing of equality of distributions. More specifically, under mild assumptions ensuring that the eigenvalues and spectral gaps are bounded away from zero, we derive an explicit and sharp expression for the non-asymptotic quantile, which allows us to obtain a calibrated test. We further propose a simplified version of this quantile, at the cost of additional assumptions, to obtain a more computationally tractable and practically usable procedure. These results rely on precise non-asymptotic concentration inequalities, in particular for auto-renormalized processes \citep{bertail2008exponential}. Then, in the asymptotic regime, we establish the optimality of our procedure and identify the dominant terms of the quantile. In particular, we highlight the central role played by the spectral elements of the empirical within-group covariance operator. We further elucidate the connection between our non-asymptotic quantile and its asymptotic counterpart. Finally, we propose an estimator of the non-asymptotic quantile based on empirical estimates of the eigenelements of the within-group covariance operator, together with a fully data-driven procedure for tuning the hyperparameters involved in its definition. In particular, we introduce a method for selecting the number of spectral components that retain sufficient information for reliable testing. The quantile we propose is then data-adaptive, in contrast to many tests based on deterministic quantiles. We demonstrate the performance of our method on simulated data in various scenarios. In particular, our procedure is always calibrated regardless the distributions, sizes, or dimensions of the two samples. Moreover, even if the proposed test is slightly conservative as compared with the chi-squared asymptotic quantile (that does not necessarily yields calibrated tests for small and moderate finite sample sizes), its empirical power remains competitive. This nice property of our method is due to the data-adaptive nature of the proposed quantile, that achieves high performance under both the null and the alternative hypotheses.

%%%%%%%%%%%%%%%%
\subsection{Outline} 
In Section~\ref{sec:regul_norm_MMD} we define the st-nMMD statistic. The study of its quantiles is conducted in Section~\ref{sec:theory}. After proposing simplified versions of these quantiles, we design in Section~\ref{sec:datadriven} an algorithm convenient for two-sample test in a non-asymptotic setting. This algorithm is studied in Section~\ref{sec:simu} on synthetic datasets and on the MNIST dataset. Proofs of all theoretical results are gathered in Appendix, as well as supplementary figures.

\paragraph{Notation}
We denote by $\| \cdot \|_{\Hk}$ the norm associated to $\Hk$. We also denote by ${\mathbb N}$ the set of nonnegative integers and ${\mathbb N}^*={\mathbb N}\setminus \{0\}$.
We denote by $\trueq{S}$  the  quantile of a statistic $S$ at the level $(1-\alpha) \in (0,1)$. Finally, the notation $C_{1:T}$ refers to $\left(C_t\right)_{t \in \{1,\ldots,T\}}$.
%%%%%%%%%%%%%%%%%%%%%%%
%%%%%%%%%%%%%%%%%%%%%%%
\section{The spectrally truncated normalized MMD}\label{sec:regul_norm_MMD}
%%%%%%%%%%%%%%%%%%%%%%%
\subsection{Kernel embedding of distributions}
Let $\left(\mathcal{Z},\|\cdot\|_{\mathcal{Z}}\right)$ be a separable metric space of (possibly large) dimension $d$ and let $X$ and $Y$ be independent random variables taking values in $\mathcal{Z}$, with respective distributions $\mathbb{P}_X$ and $\mathbb{P}_Y$. We consider the two-sample testing problem with null hypothesis $H_0$ and alternative hypothesis $H_1$, defined as
\begin{equation}
H_0 : {\mathbb{P}_X = \mathbb{P}_Y}
\quad \text{and} \quad
H_1 : {\mathbb{P}_X \neq \mathbb{P}_Y}.
\label{non_parametric_test}
\end{equation}
We consider a reproducing kernel $k(\cdot,\cdot)$ with associated reproducing kernel Hilbert space (RKHS) $\Hk$. For an introduction to kernel embeddings, we refer the reader to \cite{muandet2017kernel}. 
We assume that $k$ is characteristic so that the testing problem can then be restated in terms of kernel mean embeddings (Riesz representation theorem) as
\begin{equation}
H_0 : {\mu_X = \mu_Y}
\quad \text{and} \quad
H_1 : {\mu_X \neq \mu_Y}.
\label{kernel_test}
\end{equation}
We also assume that $k$ is continuous and that the mapping $z \in \mathcal{Z} \mapsto k(z,z) \in \mathbb{R}$ is integrable with respect to both $\mathbb{P}_X$ and $\mathbb{P}_Y$. The mean embeddings and the covariance operators of $\mathbb{P}_X$ and $\mathbb{P}_Y$ are then well-defined and are expressed as follows:
\begin{align*}
	\mu_X = \EE_{\PP_X}[k(X,\cdot)] \quad &; \quad \mu_Y = \EE_{\PP_Y}[k(Y,\cdot)]. \\
    \Sigma_{X} = \int_{x \in \mathcal{Z}} \Big( k(x,\cdot) - \mu_X \Big)^{\otimes^2_{\Hk}}  d\P_X(x) \quad &; \quad     \Sigma_{Y} = \int_{y \in \mathcal{Z}} \Big( k(y,\cdot) - \mu_Y \Big)^{\otimes^2_{\Hk}} d\P_Y(y),
\end{align*}
where $\otimes_{\Hk}$ denotes the usual tensor operator (see Appendix \ref{Appendice_Operators} and \cite{muandet2017kernel}). We define the homogeneous within-group covariance operator as 
\begin{equation}
\Sigma = (1-p)\Sigma_X + p\Sigma_Y,
\label{within_covariance_operator}
\end{equation}
where $p \in (0,1)$ is a weighting coefficient used to balance the two populations. In the sequel, we focus on the homoscedastic setting in which $\Sigma = \Sigma_X = \Sigma_Y$.

Since $\Hk$ is a separable Hilbert space (as $\mathcal{Z}$ is separable and $k(\cdot,\cdot)$ is continuous), the operator $\Sigma$ is self-adjoint, nonnegative, and trace class. Hence we introduce its eigenelements $(\lambda_t, f_t)_{t \in \mathbb{N}^*}$, where $(f_t)_{t \in \mathbb{N}^*}$ forms an orthonormal basis of $\Hk$, and $(\lambda_t)_{t \in \mathbb{N}^*}$ is a non-increasing sequence of nonnegative real numbers converging to $0$ (possibly vanishing beyond a finite index). The operator $\Sigma$ admits the spectral decomposition 
$$\Sigma = \sum_{t=1}^{\infty} \lambda_t f_t \otimes_{\Hk} f_t.$$ 
For a fixed $T \in \mathbb{N}^*$, we define the truncated spectral decomposition of $\Sigma$ by
$$\Sigma_{T} = \sum_{t=1}^{T} \lambda_t f_t \otimes_{\Hk} f_t.$$

%%%%%%%%%%%%%%%%%%%%%%%
\subsection{Estimation of Distribution Embeddings}
Suppose we observe $n_X$ independent observations $(X_1,\ldots,X_{n_X})$ from $\mathbb{P}_X$ and $n_Y$ independent observations $(Y_1,\ldots,Y_{n_Y})$ from $\mathbb{P}_Y$, where each observation is described by $d$ dependent variables, with $d$ large with respect to $n_X$ and $n_Y$. 
The empirical counterpart of previously defined quantities is obtained using  empirical estimates defined as
\begin{equation*}
    \widehat{\mu}_X = \frac{1}{n_X} \underset{i=1}{\overset{n_X}{\sum}} k(X_i,\cdot) \quad ; \quad \widehat{\mu}_Y = \frac{1}{n_Y} \underset{j=1}{\overset{n_Y}{\sum}} k(Y_j,\cdot),
\end{equation*}
and 
\begin{equation}
     \widehat{\Sigma} = \frac{n_X}{n_X+n_Y} \widehat{\Sigma}_{X} + \frac{n_Y}{n_X+n_Y} \widehat{\Sigma}_{Y},
    \label{emprical_SW}
\end{equation}
with 
$$  \widehat{\Sigma}_{X} = \frac{1}{n_X} \underset{i=1}{\overset{n_X}{\sum}} \Big(k(X_i,\cdot)-\widehat{\mu}_X\Big)^{\otimes^2_{\Hk}} \quad ; \quad \widehat{\Sigma}_{Y} = \frac{1}{n_Y} \underset{j=1}{\overset{n_Y}{\sum}} \Big(k(Y_j,\cdot)-\widehat{\mu}_Y\Big)^{\otimes^2_{\Hk}}.$$

We define the spectral decomposition and, for a fixed $T \in \mathbb{N^*}$,  the truncated spectral decomposition of $\widehat{\Sigma} $ by
\begin{equation*}
    \hS = \underset{t=1}{\overset{+ \infty}{\sum}} \widehat{\lambda}_t \widehat{f}_t \otimes_{\Hk} \widehat{f}_t \quad ; \quad \hST = \underset{t=1}{\overset{T}{\sum}} \widehat{\lambda}_t \widehat{f}_t \otimes_{\Hk} \widehat{f}_t,
\end{equation*}
where  $(\widehat{f}_t)_{t \in \N^{*}}$ is an orthonormal basis and $(\widehat{\lambda_t})_{t \in \N^{*}}$ is a decreasing sequence of real numbers with $\widehat{\lambda_t}=0$ for $t\geq t^*$ for some $t^* \leq n_X + n_Y$. For any $r \in \R^{*}$, we  set
$$ \hS^{r} = \underset{t=1}{\overset{+ \infty}{\sum}} \widehat{\lambda}_t^{r} \widehat{f}_t \otimes_{\Hk} \widehat{f}_t \quad ; \quad \hST^{r} = \underset{t=1}{\overset{T}{\sum}} \widehat{\lambda}_t^{r} \widehat{f}_t \otimes_{\Hk} \widehat{f}_t,$$
with the convention $\widehat{\lambda}_t^{r}=0$ if $\widehat{\lambda}_t=0$. Observe that the case $r=-1$ corresponds to taking the pseudo-inverse of $\hS$ and $\hST$. Finally we consider the following test statistics, hereafter referred to as the st-nMMD statistic:
\begin{equation}
    \widehat{D}^2_T = \frac{n_X n_Y}{n_X + n_Y}  \left \| \hST^{-\frac{1}{2}} \left(\widehat{\mu}_X - \widehat{\mu}_Y \right)   \right \|_{\Hk}^2 = \frac{n_X n_Y}{n_X + n_Y} \underset{t=1}{\overset{T}{\sum}} \frac{\langle \widehat{f}_t, \widehat{\mu}_X - \widehat{\mu}_Y \rangle_{\Hk}^2}{\widehat{\lambda}_t},
    \label{Hotelling_test_stat}
\end{equation}
where $\langle \cdot , \cdot  \rangle_{\Hk}$ stands for the RKHS scalar product 
and $T$ is chosen so that all eigenvalues $\big(\widehat{\lambda}_t\big)_{t=1,\ldots,T}$ are positive.

The st-nMMD statistic is a finite sum of ratios between the projections of the difference in mean embeddings onto the directions governed by the variability of the data, and the eigenvalues of the empirical within-group covariance operator, which control the fluctuations of the numerator. This formulation allows us to exploit concentration inequalities for self-normalized processes \citep{bertail2008exponential}. Finally, we emphasize that the st-nMMD statistic depends on $T$, the number of spectral components considered.

Early procedures based on the st-nMMD relied on the asymptotic $\chi^2$ distribution under the null hypothesis. In particular, \cite{eric2007testing} showed that $$\widehat{D}^2_T \underset{n_Y \rightarrow + \infty}{\underset{n_X \rightarrow + \infty}{\overset{\mathcal{L}}{\longrightarrow}}}\chi^2(T),$$ where $\overset{\mathcal{L}}{\rightarrow}$ stands for the convergence in distribution. As expected, this approximation appears to be poorly calibrated for moderate sample sizes and for large values of the truncation parameter. As shown in Section~\ref{sec:simu}, which presents the numerical results, this miscalibration is already apparent for $n_X = n_Y = 50$ (Figure~\ref{Fig:level_simulations_T_005}). This highlights the practical need for a new non-asymptotic test that does not rely on permutation methods, in order to avoid heavy computational costs.
%%%%%%%%%%%%%%%%%%%%%%%
%%%%%%%%%%%%%%%%%%%%%%%
\section{Non-asymptotic exponential upper-bound for the st-nMMD}
\label{sec:theory}
To establish a non-asymptotic control of the st-nMMD test we study the random fluctuations of the statistic $\widehat{D}^2_T$ under the null hypothesis. Since this statistic corresponds to a renormalized squared distance, the testing procedure naturally rejects the null hypothesis $H_0$ for large values of $\widehat{D}^2_T$. Consequently, the rejection area is of the form $[x, +\infty)$, for some threshold $x > 0$ to be determined.  Let $\trueq{\widehat{D}^2_T}$ denote the exact $(1-\alpha)$-quantile of the null distribution of $\widehat{D}^2_T$, as defined in~(\ref{Hotelling_test_stat}), for any $\alpha \in (0,1)$. In the following, we derive  explicit upper bound on $\trueq{\widehat{D}^2_T}$. 
This result yields a properly calibrated test, ensuring control of the type-I error. 	

In the sequel, we assume that $H_0$ holds, so that $\PP_X = \PP_Y$, that we denote simply by $\P$, and the mean embeddings of the two populations coincide and are equal to $\mu = \EE_{\PP}[k(Z,\cdot)]$. The within-group covariance operator is then given by
\[\Sigma = \int_{z \in \mathcal{Z}} \Big( k(z,\cdot) - \mu \Big)^{\otimes^2_{\Hk}} d\PP(z).\]
	We then define the left–right spectral gaps of $\Sigma$ as
	\begin{align*}
		\Delta_1 &= \frac{1}{2} \left( \lambda_1 - \lambda_{2} \right), \\
		\Delta_t &= \frac{1}{2} \min \left( \lambda_t - \lambda_{t+1} , \lambda_{t-1} - \lambda_t \right), \quad \text{for } t \in \{2, \cdots, T-1\}, \\
		\Delta_T &= \frac{1}{2} \min \left( \lambda_T , \lambda_{T-1} - \lambda_T \right).
	\end{align*}
	Following assumptions will be considered in the sequel:
	\begin{itemize}
		\item[\A{1}] the two populations are well balanced: $n_X = n_Y$, which we denote by $\nb$ for simplicity.
		\item[\A{2}] the kernel is bounded: $M_k < +\infty$, with $M_k = \underset{z \in \mathcal{Z}}{\sup} \, k(z, z)$.
		\item[\A{3}] the eigenvalues of $\Sigma$ are all simple.
	\end{itemize}
Assumption~\A{1} is mainly technical, as our strategy (see Theorem~\ref{Hoeffding_inequality} in Section~\ref{existing_results}) relies on Hoeffding's inequality for random variables with symmetric distributions \citep{bertail2008exponential}, which requires balanced sample sizes. Assumption~\A{2} is very mild and natural in the context of kernel methods. In particular, it is satisfied by commonly used kernels such as the Gaussian and Laplacian kernels.  Assumption~\A{3}, while more restrictive, is essential:  our proof relies on the direction-wise decomposition of $\widehat{D}^2_T$ given in~(\ref{Hotelling_test_stat}), which requires strictly positive spectral gaps around each eigenvalue \citep{blanchard2007statistical,zwald2005convergence,ozier2024extending}.

%%%%%%%%%%%%
\subsection{Concentration inequalities on the st-nMMD statistic}
\label{sec:main} 
To derive a non-asymptotic quantile for the distribution of the st-nMMD statistic $\widehat{D}_T^2$, we first establish a non-asymptotic concentration inequality for this statistic. Our main result is stated under Assumptions~\A{1}$\sim$\A{3}.
 Observe that in this case, $  \widehat{D}^2_T$ writes
$$
    \widehat{D}^2_T = \frac{n}{2} \underset{t=1}{\overset{T}{\sum}} \frac{\langle \widehat{f}_t, \widehat{\mu}_X - \widehat{\mu}_Y \rangle_{\Hk}^2}{\widehat{\lambda}_t}.
$$
\begin{theorem}
    For all $\delta>0$, let us define $\qmaj\left(\nb,\delta,M_k,\lambda_{1:T},\Delta_{1:T}\right)$ by
    \begin{align}
        \qmaj \bigg(\nb,&\delta,M_k,\lambda_{1:T},\Delta_{1:T}\bigg) \notag \\
        &\hspace{-0.7cm}:= 2 T \underset{t=1,\ldots,T}{\max}  \left(\frac{\lambda_t - 4 M_k \sqrt{\frac{\delta}{\nb}}}{\lambda_t - K_{1,t}\left(\nb,M_k,\delta,\Delta_t \right) }  \right) \left(\sqrt{\delta} + \frac{K_2 \left( \nb,M_k,\delta \right)}{\underset{t=\{1,\ldots,T\}}{\min} \left\{\Delta_t \sqrt{\lambda_t - 4 M_k \sqrt{\frac{\delta}{\nb}}} \right\}} \right)^2,
        \label{x_bound}
    \end{align}
    where for all $t \in \{1,\ldots,T\}$ 
    \begin{align*}
        K_{1,t}\left(\nb,M_k,\delta,\Delta_t \right) 
         &= 4 M_k \left(1+ 2 \frac{\nb-1}{\nb} \right) \sqrt{\frac{\delta}{\nb}} \ \notag \\
        & \hspace{-0.5 cm} + \frac{24 M_k^2}{\Delta_t} \left(\frac{\sqrt{2} \left(1 +\sqrt{2\delta} \right)}{\nb } +  \frac{4}{\sqrt{\nb}} \left(\frac{\nb-1}{\nb}\right)^2 \right) \left( 1+ \sqrt{\frac{\delta}{2}} \right)  + \frac{2 M_k}{\nb} \left(2 + \sqrt{\delta} \right)^2     \end{align*}
        and
         \begin{align*}
         K_2 \left( \nb,M_k,\delta \right)  &= \frac{12 M_k^{\frac{3}{2}}}{\sqrt{2 \nb}} \left(1 + \sqrt{2\delta}\right) \left(1 + \sqrt{\frac{\delta}{2}} \right).
    \end{align*}
    Assume that \A{1}$\sim$\A{3} are satisfied and let us suppose that 
     \begin{equation}\label{A4bis}
 \frac{12 M_k}{\sqrt{n}} \left( 1 + \sqrt{\frac{\delta}{2}}\right)  < \underset{t \in \{1,\ldots,T\}}{\min} \Delta_t \tag{SP1}
 \end{equation}
 and
    \begin{equation}
        \lambda_t > K_{1,t}\left(\nb,M_k,\delta,\Delta_t \right), \quad t \in \{1,\ldots,T\}. \tag{SP2}
        \label{lambda_t_condition}
    \end{equation}
Then, we have:
    \begin{equation*}
    \PPO \left(\widehat{D}^2_T > \qmaj\left(\nb,\delta,M_k,\lambda_{1:T},\Delta_{1:T}\right) \right) \leq  9 T e^{-\delta}.
    \end{equation*}
\label{the_theorem}
\end{theorem}
At the price of an additional assumption, the second term on the right-hand side of (\ref{x_bound}) can be upper bounded, yielding a simpler expression for the upper bound, denoted by $\qmajbis\left(\nb,\delta,M_k,\lambda_{1:T},\Delta_{1:T}\right)$.
\begin{corollary}
Assume that \A{1}$\sim$\A{3} hold, together with the spectral conditions (\ref{A4bis}) and  (\ref{lambda_t_condition}) for all $t \in \{1,\ldots,T\}$. Let $\delta>0$. If in addition
\begin{equation}
    \underset{t=\{1,\ldots,T\}}{\min} \left\{\Delta_t \sqrt{\lambda_t - 4 M_k \sqrt{\frac{\delta}{\nb}}} \right\} \geq c \frac{K_2 \left( \nb,M_k,\delta \right)}{\sqrt{\delta}},
    \label{K_3_simplified}  \tag{SP3}
\end{equation}
for $c$ a positive constant, then we obtain a similar result:
    \begin{equation*}
        \PPO \left(\widehat{D}^2_T > \qmajbis \left(\nb,\delta,M_k,\lambda_{1:T},\Delta_{1:T}\right) \right) \leq  9 T e^{-\delta},
    \end{equation*}
    with
    \begin{equation}
        \qmajbis\left(\nb,\delta,M_k,\lambda_{1:T},\Delta_{1:T}\right) := 2\left(1+c^{-1}\right)^2 T \delta \underset{t=1,\ldots,T}{\max}  \left(\frac{\lambda_t - 4 M_k \sqrt{\frac{\delta}{\nb}}}{\lambda_t - K_{1,t}\left(\nb,M_k,\delta,\Delta_t \right) }  \right). 
        \label{x_simplified}
    \end{equation}
\label{bound_simple}
\end{corollary}
In the sequel, for notational simplicity, we denote the quantities $\qmaj \left(\nb,\delta,M_k,\lambda_{1:T},\Delta_{1:T}\right) $ and $\qmajbis\left(\nb,\delta,M_k,\lambda_{1:T},\Delta_{1:T}\right)$, introduced in Theorem~\ref{the_theorem} and Corollary~\ref{bound_simple}, simply by $\qmaj$ and $\qmajbis$ respectively.

Theorem~\ref{the_theorem} and Corollary~\ref{bound_simple} provide two exponential deviation inequalities for the statistic $\widehat{D}^2_T$ under the null hypothesis~$H_0$. They yield two upper bounds, $\qmaj$ and $\qmajbis$,  for $\trueq{\widehat{D}^2_T}$, thereby leading to non-asymptotically calibrated tests, as described in the next section. Both quantities $K_{1,t}$ and $K_2$ involved in  $\qmaj$ and $\qmajbis$ are intricate, but they are derived from sharp non-asymptotic concentration inequalities and can therefore be applied for any $\nb$. 

Let us now analyse the upper bounds $\qmaj$ and $\qmajbis$. First, observe that 
$$  
\widehat{D}^2_T = \frac{1}{2} \underset{t=1}{\overset{T}{\sum}} \left(\frac{\langle \widehat{f}_t, \sqrt{n} \left(\widehat{\mu}_X - \widehat{\mu}_Y \right) \rangle_{\Hk}}{\sqrt{\widehat{\lambda}_t}} \right)^2,
$$
and under the null hypothesis $H_0$, the random variable $\left(\widehat{\mu}_X - \widehat{\mu}_Y\right)$ is of the order $1/\sqrt{n}$ \citep{eric2007testing}. 
Hence, we expect that the order of magnitudes in $\qmaj$ and $\qmajbis$ remains asymptotically constant with respect to $n$, which indeed is the case. This arises from the fact that concentration inequalities used to get $\qmaj$ and $\qmajbis$ are sharp \citep{zwald2005convergence,reiss2020nonasymptotic,bertail2008exponential}. 
Lastly, $\qmaj$ and $\qmajbis$ depend on the eigenelements. In particular, the eigenvalues appear in both the numerator and the denominator in a similar way,  ensuring that the ratio remains well-controlled even when the eigenvalues are very small. Additionally, the denominator depends on the spectral gaps, as in \cite{reiss2020nonasymptotic}.

Assumptions~(\ref{A4bis}), (\ref{lambda_t_condition}), and~(\ref{K_3_simplified}) impose lower-bound conditions on the eigenvalues and spectral gaps, requiring them not to be too small. In particular, they imply that $\lambda_{1:T}$ and $\Delta_{1:T}$ should be at least of the order of the parametric rate $1/\sqrt{n}$. This is not surprising, as excessively small values of $\lambda_t$ lead to estimation difficulties, while small values of $\Delta_t$ make it hard to distinguish $\lambda_t$ from $\lambda_{t+1}$ or $\lambda_{t-1}$, and consequently to separate the corresponding eigenfunctions $f_t$ from $f_{t+1}$ or $f_{t-1}$. Since eigenvalues and the total inertia in the embedded space are closely related, these lower bounds imply that the variance of the embedded data along the first $T$ directions of the within-group covariance operator $\Sigma$ in $\Hk$ is not too small.  These conditions implicitly constrain the underlying distributions $\PP_X$ and $\PP_Y$.  Note that for sufficiently large $n$, condition~(\ref{lambda_t_condition}) automatically implies the gap condition~(\ref{A4bis}).

%%%%%%%%%%%%%%%%%
\subsection{Main ideas of the proof.}
The proofs of Theorem~\ref{the_theorem} and Corollary~\ref{bound_simple} are given in Appendix~\ref{la_section_proof}. Here we outline the main ideas of our strategy. The main difficulty in obtaining a non-asymptotic upper bound for the st-nMMD test arises from the use of spectral truncation as a regularization strategy for $\widehat{\Sigma}$. A first natural approach would be to control the fluctuations of $\widehat{D}^2_T$ globally through the expression (\ref{Hotelling_test_stat}) by dealing with the terms $\widehat{\Sigma}_T$ and $\widehat{\mu}_X-\widehat{\mu}_Y$ separately. 
In particular, the difference in mean embeddings can be controlled using results from \cite{gretton2012kernel}. Operator perturbation theory could then be applied to control the renormalizing term $\widehat{\Sigma}_T$ \citep{koltchinskii2000random,blanchard2007statistical,zwald2005convergence}. However, this global approach leads to an upper bound involving the term $\lambda_T^{-1}$ as a multiplicative factor (similarly to Proposition~3 in \cite{ozier2024extending}), which can be prohibitively large.

A second idea, which we adopt, is to control simultaneously the terms $\widehat{\mu}_X-\widehat{\mu}_Y$ and $\widehat{f}_t / \widehat{\lambda_t}$ using a local approach (\textit{e.g.} direction-by-direction). To this end, we  leverage the geometry induced by projecting $\widehat{\mu}_X - \widehat{\mu}_Y$ onto each direction $\widehat{f}_t$. 
The statistic $\widehat{D}^2_T$ can then be expressed as a sum of ratios, naturally suggesting an interpretation as a sum of self-normalized processes, where each denominator corresponds to the standard deviation of its numerator.
To achieve this appropriate normalization, we derive a sequence of sharp concentration inequalities, primarily based on McDiarmid's inequality and tools from operator perturbation theory (see Appendix \ref{existing_results} for details). We then adapt the work of \cite{bertail2008exponential}, who derived sharp Hoeffding-type inequalities for multivariate symmetric self-normalized sums, to the two-sample testing framework in $\Hk$ accounting for both the kernel structure and the specific $\Sigma$-renormalization. 
The symmetrization step is made possible by the balanced-sample Assumption~\A{1}. Following \cite{bertail2008exponential}, we obtain a local direction-wise control by analyzing each term in the sum defining $\widehat{D}_T^2$. As a consequence, the resulting upper bound scales linearly with $T$, while the terms involving $\lambda_t$ appear in both the numerator and the denominator of the quantile. This prevents the bound from exploding when the $\lambda_t$'s take small values. We note that our proof remains true for non-characteristic kernels but the equivalence between tests (\ref{non_parametric_test}) and (\ref{kernel_test}) is no longer satisfied.

%%%%%%%%%%%%%%%%%%%%%%%%%%%%%%%%%%%%%%
\subsection{From concentration inequalities to non-asymptotic calibrated tests}
Theorem \ref{the_theorem} and Corollary \ref{bound_simple} provide two calibrated tests. Let $\alpha \in (0,1)$ be a significant level  and fix $\delta > 0$ such that $\alpha = 9T e^{-\delta}$. Then, the tests
\begin{equation}
		\ind{\big\{\widehat{D}^2_T > \qmaj \big\}}
		\quad \text{and} \quad
		\ind{\big\{\widehat{D}^2_T > \qmajbis \big\}},
		\label{les_deux_tests}
	\end{equation}
have type-I error controlled at level $\alpha$. 
\begin{remark}
Under the alternative hypothesis $H_1$, since $\left(f_t\right)_{t \in \N^{*}}$ is an orthonormal basis of $\Hk$, there exists $s \in \N^{*}$ such that $\langle f_s, \mu_X - \mu_Y \rangle_{\Hk} \neq 0$. So, if $s \leq T$, then 
$$
\widehat{D}^2_T \geq \frac{n}{2} \frac{\langle \widehat{f}_s, \widehat{\mu}_X - \widehat{\mu}_Y \rangle_{\Hk}^2}{\widehat{\lambda}_s} \underset{n \rightarrow + \infty}{\sim} \frac{n}{2} \frac{\langle f_s, \mu_X - \mu_Y \rangle_{\Hk}^2}{\lambda_s}  \underset{n \rightarrow + \infty}{\longrightarrow} + \infty,
$$
where $\sim$ denotes asymptotic equivalence as $n \to \infty.$ Hence, for sufficiently large $T$, the rejection region of the form $[x, + \infty]$ is justified. 
\end{remark}

%%%%%%%%%%%%%%%%%%%%%%
\subsection{Study in the asymptotic regime}\label{sec:Discussion:asym} 
We discuss our main results (Section~\ref{sec:main}) from the perspective of the asymptotic regime in which $n$ tends to infinity. In this setting, the parameter $\delta$ is allowed to depend on $n$, with $\delta \equiv \delta(n) \to +\infty$. Our goal is to identify the leading terms in the quantiles $\qmaj$ and $\qmajbis$ and to discuss their optimality, based on the results derived from Theorem~\ref{the_theorem} and Corollary~\ref{bound_simple}. Beforehand, we recall that condition~(\ref{A4bis}) requires that, asymptotically, $\underset{t \in \{1,\ldots,T\}}{\min}\Delta_t$ be larger than $M_k \sqrt{\delta / n}$, up to a constant independent of $M_k$ and $n$.
\begin{proposition}\label{prop: asymp1}
Assume that \A{1}$\sim$\A{3} and the gap condition (\ref{A4bis}) are satisfied. We assume that $\delta$ depends on $n$ with $\delta\equiv \delta(n)\to+\infty$ and $\delta(n)/n\to 0$ when $n\to+\infty$. Then, Condition (\ref{lambda_t_condition}) of Theorem \ref{the_theorem} and (\ref{K_3_simplified}) of Corollary \ref{bound_simple} can be written as:  for any $t \in \{1,\ldots,T\}$, 
\begin{equation}
\label{condDeltaLambda}
\lambda_t \Delta_t \geq c_1 M_k^2 \sqrt{\frac{\delta}{\nb}},
\end{equation}
which implies 
$$   \lambda_t \geq c_2  \left( \frac{\delta}{\nb} \right)^{\frac{1}{4}} 
$$
where $c_1$ and $c_2$ do not depend on $M_k$, $\nb$, $\lambda_{1:T}$ and $\Delta_{1:T}$. \\
Furthermore, if (\ref{condDeltaLambda}) is satisfied, then we have
\begin{equation}
    \PPO \left(\widehat{D}^2_T > 2\left(1+c^{-1}\right)^2 T \delta \underset{t\in\{1,\ldots,T\}}{\max}  \left(\frac{\lambda_t - 8 M_k \sqrt{\frac{\delta}{\nb}}}{\lambda_t - \left(8 M_k  + \frac{\kappa M_k^2}{\Delta_t} \right) \sqrt{\frac{\delta}{\nb}}}  \right) \right) \leq  9 T e^{-\delta},
    \label{simplied_version_asymptotic_bound}
\end{equation}
where $\kappa >0$ is a constant independent of $M_k$, $\nb$, $\lambda_{1:T}$ and $\Delta_{1:T}$. 
\label{asymptotic_regime_with_delta_condition}
\end{proposition}
Proposition~\ref{asymptotic_regime_with_delta_condition} is proved in Appendix~\ref{sec:proofProp1}. This result states that if Conditions (\ref{A4bis}) and (\ref{lambda_t_condition}) are satisfied, then condition (\ref{K_3_simplified}) of Corollary \ref{bound_simple} is automatically satisfied for  a positive constant $c$ not depending on $M_k$, $\nb$, $\lambda_{1:T}$ and $\Delta_{1:T}$, along with Inequality \eqref{condDeltaLambda}. 
A careful examination of the proof shows that if \eqref{condDeltaLambda} is satisfied with $c_1$ sufficiently large, then Condition (\ref{lambda_t_condition}) is also true.
We recall that under Assumption~\eqref{A4bis}, the spectral gaps $\Delta_t$ cannot be too small. The first part of Proposition~\ref{asymptotic_regime_with_delta_condition} further strengthens this requirement. At the same time, the eigenvalues are allowed to converge to zero as $n \to \infty$, but at a rate slower than the parametric rate. Lower bounds on both eigenvalues and spectral gaps are essential to obtain our result, and neither condition can be deduced from the other. 

To discuss the second part of Proposition~\ref{prop: asymp1}, note that since $\sqrt{\delta / \nb}$ is negligible compared to $\lambda_t$ for all $t \in \{1,\ldots,T\}$, Inequality~(\ref{simplied_version_asymptotic_bound}) implies that, in the asymptotic regime, the quantile is close to $2(1 + c^{-1})^2 T \delta$. In particular, if Condition~(\ref{K_3_simplified}) of Corollary~\ref{bound_simple} holds for a sufficiently large constant $c$, the quantile is close to $2T\delta$. Recalling that a chi-squared distribution with $T$ degrees of freedom has expectation equal to $T$, we obtain an asymptotic quantile that matches, up to the multiplicative factor $2\delta$, the quantile of a chi-squared distribution with $T$ degrees of freedom. This aligns with the known asymptotic convergence of the test statistic $\widehat{D}^2_T$ to a $\chi^2(T)$ distribution \citep{eric2007testing}. Moreover, if $D^2_T\sim\chi^2(T)$, then we have the following sharp concentration inequality
	$$\P(D^2_T\geq T+2\sqrt{Tz}+2z)\leq \exp(-z),\quad z>0,$$
	see Lemma~1 of \cite{LaurentMassart2000}. 
	Overall, the results obtained in the asymptotic regime indicate that our proposed quantile is optimal up to the multiplicative factor $2\delta$, which arises from the use of concentration inequalities to control the random terms appearing in the definition of $\widehat{D}^2_T$ and to estimate the eigenelements $\lambda_{1:T}$ and $f_{1:T}$. Finally, observe that if assumptions of Proposition~\ref{prop: asymp1} hold with $\delta=\gamma\log(n)$, for some constant $\gamma>0$, then the probability in \eqref{simplied_version_asymptotic_bound} decreases at a polynomial rate
	$$ \PPO \left(\widehat{D}^2_T > 2\left(1+c^{-1}\right)^2 T \gamma \log(\nb) \underset{t=1,\ldots,T}{\max}  \left(\frac{\lambda_t - 8 M_k \sqrt{\frac{\gamma \log(\nb) }{\nb}}}{\lambda_t - \left(8 M_k  + \frac{\kappa M_k^2}{\Delta_t} \right) \sqrt{\frac{\gamma \log(\nb) }{\nb}}}  \right) \right) \leq \frac{9 T}{\nb^{\gamma}}.$$

In the next section, we provide a fully data-driven testing procedure with a reasonable computational cost which can be used by practitioners. 
%%%%%%%%%%%%%%%%%%%%%%
%%%%%%%%%%%%%%%%%%%%%%
\section{A data-driven calibration}\label{sec:datadriven}
%%%%%%%%%%%%%%%%%%%%%%
\subsection{Bringing the gap between asymptotic and non asymptotic quantiles}
\label{sec:simplified:asym}
 Theorem~\ref{the_theorem} and Corollary~\ref{bound_simple} provide non-asymptotic upper bounds for $\trueq{\widehat{D}^2_T}$, for any $\alpha \in (0,1)$. Since $\qmaj$ and $\qmajbis$ are explicit and implementable, provided that the spectral elements and the truncation level $T$ are known, it yields a calibrated test under the null hypothesis. In practice, although the eigenelements and spectral gaps can be replaced by their empirical estimates when $\nb$ is sufficiently large, our bounds involve absolute constants that may be large, potentially leading to overly conservative tests. Hence, we propose treating these constants as hyperparameters to be tuned using the available data. To facilitate practical tuning, we further simplify the quantiles from (\ref{x_bound}) and~(\ref{x_simplified}), by retaining only the leading-order terms in $\nb$, which capture the dominant behavior of the bounds in the non-asymptotic regime and reduce to a reasonable number of hyperparameters to tune. The following proposition provides a calibrated test under the null hypothesis where the parameter $\delta$ is now treated as a constant independent of $\nb$.
\begin{proposition}
Assume \A{1}$\sim$\A{3} and the gap condition (\ref{A4bis}) hold and let $\alpha \in (0,1)$. The spectral conditions (\ref{lambda_t_condition}) of Theorem \ref{the_theorem} and (\ref{K_3_simplified}) of Corollary \ref{bound_simple} can be written as: for all $t \in \{1,\ldots,T\}$,
\begin{equation}
    \lambda_t \Delta_t \geq c_4 M_k{^2} \frac{1}{\sqrt{\nb}},
    \label{lambda_t_condition_simplified}
\end{equation}
for $c_4$ depending on $\alpha$ but not on $M_k$, $\nb$, $\lambda_{1:T}$ and $\Delta_{1:T}$. \\
Furthermore, if (\ref{lambda_t_condition_simplified}) is satisfied, then we have, 
 \begin{equation}
         \PPO \left(\widehat{D}^2_T >  \trueq{\chi^2(T)} \underset{t=1,\cdots,T}{\max} \left( \frac{1}{1 - \frac{\rho_t}{\sqrt{\nb} \lambda_t \Delta_t}} \right)
 \right) \leq  \alpha,
         \label{our_used_quantile}
     \end{equation}
     where
     \begin{equation}
         \trueq{\chi^2(T)} \underset{t=1,\cdots,T}{\max} \left( \frac{1}{1 - \frac{\rho_t}{\sqrt{\nb} \lambda_t \Delta_t}} \right)
        \label{simplied_quantile}
    \end{equation}
represents the asymptotic version of $\qmajbis$ defined in (\ref{x_simplified}) and where $\rho_{1:T} >0$ are some constants such that the denominators above remain positive.  
\label{Condition_simplification}
\end{proposition}
Proposition \ref{Condition_simplification}, which holds for any $n$, provides upper bounds for $\trueq{\widehat{D}^2_T}$, that are simpler than $\qmaj$ and $\qmajbis$, but at the price of unknown constants $\rho_{1:T} >0$. As in Proposition \ref{asymptotic_regime_with_delta_condition}, both eigenvalues and spectral gaps must be sufficiently large, particularly exceeding the noise level $\frac{1}{\sqrt{\nb}}$. Inequality (\ref{lambda_t_condition_simplified}) can be interpreted as a signal detection condition, reflecting the trade-off between the $\alpha$-level and the complexity of estimating the eigenelements. The quantity in (\ref{simplied_quantile}) is obtained by multiplying the chi-squared quantile  $\trueq{\chi^2(T)}$ by $A_{T,n}$, with
$$A_{T,n}=\underset{t=1,\cdots,T}{\max} \left( \frac{1}{1 - \frac{\rho_t}{\sqrt{\nb} \lambda_t \Delta_t}} \right).$$
Clearly, $A_{T,n} \to 1$ as $n \to +\infty$, and asymptotically we recover the chi-squared quantile when $T$ is fixed.
Moreover, observe that $A_{T,n}$ does not depend on $t$. The maximization over $t$ is a technical trick that comes into play in the last lines of the proof of Theorem \ref{the_theorem}. We therefore propose averaging the contributions over all $t$-directions, leading to the following non-asymptotic upper bound of the quantile:
\begin{equation}
\label{eq:quantile-moy}
\qmajfin_{1-\alpha}\left(\nb,\alpha, \rho_{1:T},\lambda_{1:T},\Delta_{1:T}\right) =\trueq{\chi^2(T)} \times \frac{1}{T}\sum_{t=1}^T\left( \frac{1}{1 - \frac{\rho_t}{\sqrt{\nb} \lambda_t \Delta_t}} \right),
\end{equation}
also denoted $\qmajfin_{1-\alpha}$  in the sequel. Empirical studies support taking the average, providing calibrated tests for each $T$ that are less conservative compared to using the maximum (see Figure \ref{Fig:level_simulations_T_max_005} in Appendix \ref{sec:additional_figures}).

%%%%%%%%%%%%%%%%%%%%%%%%%
\subsection{Final quantile approximation and calibration} 
\label{subs:finalquantilecalib}
We still assume that the conditions of Proposition \ref{Condition_simplification} hold. Then, we can use the expression $\qmajfin_{1-\alpha} =\qmajfin_{1-\alpha}\left(\nb,\alpha, \rho_{1:T},\lambda_{1:T},\Delta_{1:T}\right)$ given in (\ref{eq:quantile-moy}) to obtain a theoretically calibrated test. For practical implementation, three challenges must be addressed. The first one is the estimation of the eigenelements; the second one is the calibration of $\rho_{1:T} >0$ for a given value of $T>0$; and the third one is the choice of $T$. Regarding the first issue, replacing $\lambda_t$ and $\Delta_t$ with their estimators $\widehat{\lambda}_t$ and $\widehat{\Delta}_t$ is theoretically justified thanks to Assumption~(\ref{lambda_t_condition_simplified}). Then, $\qmajfin_{1-\alpha}$ can be estimated by
\begin{equation}
\label{eq:quantile-estimated}
\hatqmajfin_{1-\alpha}(T)=\trueq{\chi^2(T)} \times \frac{1}{T}\sum_{t=1}^T\left( \frac{1}{1 - \frac{\rho_t}{\sqrt{\nb} \widehat{\lambda_t} \widehat{\Delta_t}}} \right).
\end{equation}

In (\ref{eq:quantile-estimated}), the hyperparameters $\rho_{1:T}$ have to be calibrated. We propose to calibrate them directly on the available dataset, depending on $\nb$ and the complex structure between variables, and without any data splitting. 

For a given value of $T$, a computable and operational version of the quantile~\eqref{eq:quantile-moy} requires calibrating the parameters $\rho_t$, which is performed using the data-driven algorithm detailed below.
The key heuristic we propose is to choose the $\rho_t$'s of the same order as their respective denominators in~\eqref{eq:quantile-moy}, so as to ensure the convergence of $\hatqmajfin_{1-\alpha}(T)$ to the correct asymptotic quantile $\trueq{\chi^2(T)}$. To simplify the calibration, we set all the $\rho_t$’s equal to a common value $\rho$. To ensure that $\hatqmajfin_{1-\alpha}(T)$ remains well defined, we take
\begin{equation}
\label{eq:rhofinal}
\rho_t = \rho = \frac 12  \sqrt{\nb} \underset{t \in \{1,\ldots,T\}}{\min} \left\{ \widehat{\lambda}_t \widehat{\Delta}_t \right\} .
\end{equation}
With this definition, $\hatqmajfin_{1-\alpha}(T)$ is thus an upward adjustment of the chi-squared quantile $\trueq{\chi^2(T)}$, which is desirable, as $\trueq{\chi^2(T)}$ underestimates the appropriate threshold (see Figure \ref{Fig:level_simulations_T_005} in Section \ref{sec:simu}).
The test decision then directly follows from this quantile approximation and the computation of $\widehat{D}^2_T$: if $\widehat{D}^2_T \leq \hatqmajfin_{1-\alpha}(T)$, the null hypothesis $H_0$ is retained; otherwise, $H_0$ is rejected.
\begin{remark} 
In the expression \eqref{eq:rhofinal}, we take half of the minimum to ensure that $\rho$ is smaller than  all the quantities $\sqrt{n}  \widehat{\lambda}_t \widehat{\Delta}_t$'s. Indeed, setting $ \rho = \eta \sqrt{\nb} \min_{t \in \{1,\ldots,T\}} \left\{ \widehat{\lambda}_t \widehat{\Delta}_t \right\}$ 
for some  $0 <\eta < 1$
ensures that
$$
\trueq{\chi^2(T)} < \hatqmajfin_{1-\alpha}(T)
< \frac{1}{1- \eta}\, \trueq{\chi^2(T)}.$$
This tuning parameter controls the closeness of the data-driven quantile to the asymptotic quantile and the conservative level of the test.
When $\eta$ tends to $0$, the quantile $\trueq{\chi^2(T)}$ is recovered, whereas, when $\eta$ tends to $1$, $\hatqmajfin_{1-\alpha}(T)$ tends to infinity. 
Empirical studies support setting $\eta = 1/2$;  however, this parameter can be further adjusted depending on the context (see Figure \ref{Fig:level_simulations_T_eta_005} in Appendix \ref{sec:additional_figures}).
\end{remark}
In practice, we also need to choose the number of eigendirections $T$. We propose the following rule of thumb
\begin{equation}
\label{eq:choixT}
\widehat T = \max \left\{ t : \ \forall s \leq t, \ 
\widehat{\lambda}_s \geq \frac{\sqrt{\widehat{\lambda}_1}}{\sqrt{2\nb}}, 
\ \text{and} \ 
2 \widehat{\Delta}_s \geq \frac{\sqrt{\widehat{\Delta}_1}}{\sqrt{\nb}} \right\}.
\end{equation}
If $\widehat T = 0$, we set $\hatqmajfin_{1-\alpha}= +\infty$, in which case the test does not reject the null hypothesis. \\
This choice for $\widehat T$ validates the estimator performances of the eigenelements (larger than the signal-to-noise ratio). The quantities $\widehat{\lambda}_1$ and $\widehat{\Delta}_1$ in (\ref{eq:choixT}) stand for approximation of the variance of all the eigenvalues and spectral gaps respectively.

 In summary, our testing procedure is fully determined by equations~\eqref{eq:quantile-estimated}, \eqref{eq:rhofinal} and \eqref{eq:choixT}. This procedure is entirely data-driven and does not require any data splitting, which is commonly used in the literature.

%%%%%%%%%%%%%%%%%%%%%%
%%%%%%%%%%%%%%%%%%%%%%
\section{Numerical experiments}\label{sec:simu}
%%%%%%%%%%%%%%%%%%%%%%
\subsection{Simulation design}\label{sec:simulation protocol}

We assess the empirical performance of our procedure through simulations conducted under the null hypothesis $H_0:\mathbb{P}_X=\mathbb{P}_Y$. We focus on the balanced designs with $n_X=n_Y=n/2$. Two independent samples $X_1,\dots,X_{n/2}$ and $Y_1,\dots,Y_{n/2}$  are generated from the same isotropic distributions of \citep{hagrass2023spectral}: (i) Gaussian $\mathcal{N}_d(0,I_d)$; (ii)  uniform $\mathcal{U}_d([0,1]^d)$; (iii) Cauchy with location parameter $m=0$ and scale parameter $s=1$ (independent coordinates); and (iv) von Mises--Fisher on the unit sphere with concentration parameter $\kappa=4$ and mean direction $\mu = d^{-1/2}(1,\ldots,1)_d^\top$. We consider sample sizes $n \in \{100,1000,5000\}$ and dimensions $d \in \{2,10,100\}$. For each configuration $(n,d)$ and each distribution, we perform $R=10000$ independent repetitions.  

In addition to purely simulated data, we considered the MNIST dataset \citep{lecun2010mnist}, which consist of images of written digits (0 to 9), downsampled to $d=7 \times 7 = 49$ as in \cite{schrab2023mmd,hagrass2023spectral}. Based on these data, we denote by $\mathbb{P}:\{0,1,2,3,4,5,6,7,8,9\}$, the images containing all digits. We use this distribution to generate data under the null hypothesis. Then we also consider: $\mathbb{Q}_1:\{1,3,5,7,9\}$ (strongest separation with respect to $\mathbb{P}$), $\mathbb{Q}_2:\{0,1,3,5,7,9\}$, $\mathbb{Q}_3:\{0,1,2,3,5,7,9\}$, $\mathbb{Q}_4:\{0,1,2,3,5,7,9\}$, and $\mathbb{Q}_5:\{0,1,2,3,4,5,7,9\}$ (weakest separation with respect to $\mathbb{P}$). The distributions are constructed so that the discrepancy with $\mathbb{P}$ decreases progressively as additional digit classes are included, allowing us to assess how the procedure adapts to alternatives of varying difficulty. We sample $R=10,000$ of such distributions.

The procedures are evaluated at nominal levels $\alpha \in \{0.05,0.01\}$, and we compare the proposed non-asymptotic data-driven calibrated quantile with the asymptotic $\chi^2$ approximation. The st-nMMD test is performed thanks to the \texttt{ktest} Python package \citep{ozier2024kernel}, with the Gaussian kernel with bandwidth tuned thanks to the median heuristic \citep{garreau2017large}. Let $\widehat{D}^{2,(r)}_T$ denote the test statistic computed at repetition $r \in \{1, \ldots,10000\} $, and let   $\hatqmajfin_{1-\alpha}^{(r)}(T)$ be our operational version, fully data-driven, for the $(1-\alpha)$-quantile as defined in Eq.~(\ref{eq:quantile-estimated}). The empirical level is estimated by
\[
\widehat{\mathcal{E}}(\alpha,T)
=
\frac{1}{R}
\sum_{r=1}^{R}
\mathbf{1}_{
\left\{
\widehat{D}^{2,(r)}_T
>
\hatqmajfin_{1-\alpha}^{(r)}(T) 
\right\}},
\]
which provides a direct estimate of the Type-I error of the test. Note that this error rate depends on $T$, the number of eigenelements of the within-group covariance operator, which can either be fixed or selected using the procedures described in Section~\ref{sec:selection T}.
%%%%%%%%%%%%%%%%%%%%%%
\subsection{Calibration performance}\label{sec: type I error}
Figures \ref{Fig:level_simulations_T_005} and  \ref{Fig:level_simulations_T_001} show that the empirical level of both the asymptotic and non-asymptotic procedures does not depend on the data-generating distribution, but rather on the number of observations and dimensions. As expected, the asymptotic $\chi^2$-based procedure fails to achieve calibration in the non-asymptotic regime (e.g., $n=100$), particularly in low-dimensional settings (e.g., $d=2$). In contrast, our non-asymptotic procedure remains calibrated (all the $95\%$ confidence interval below $\alpha$) regardless of the sample size and dimension (for small $T$). These observations validate our proposed data-driven quantile (\ref{eq:quantile-estimated}) and hyperparameters calibration (\ref{eq:rhofinal}). For $d = 2,10$, the empirical level  increases with the truncation parameter, as non-informative directions can inflate the false positive rejection rate. Notably, higher dimensionality ($d=100$) tends to moderate the tests level,  resulting in more conservative procedures, regardless of the quantile used. This effect is more pronounced when the sample size is small. A similar pattern is observed in the MNIST-based simulations (Figure \ref{Fig:level_mnist_T}, $d = 49$), where the asymptotic procedure becomes more conservative for small $n$.
%%%%%%%%%%%%%%%%%%%%%%
\subsection{Power performance}
Since the proposed procedure is calibrated and slightly conservative under the null hypothesis, we also investigate its empirical power under alternatives. In particular, the non-asymptotic quantile we propose is data-adaptive, as it depends on the eigenvalues of the within-group covariance operator. These eigenvalues generally differ under the null and under the alternative, leading to different calibration values across scenarios. This contrasts with the asymptotic $\chi^2$-approximation, whose quantile is deterministic and does not depend on the data. Consequently, the conservative behavior observed under $H_0$ does not necessarily lead to a loss of power under the alternative, since the quantile adapts to the underlying covariance structure.

To investigate the adaptivity properties of our procedure we consider the MNIST dataset and challenge distribution $\mathbb{P}:\{0,1,2,3,4,5,6,7,8,9\}$, against: $\mathbb{Q}_1, \hdots, \mathbb{Q}_5$. The distributions are constructed so that the discrepancy with $\mathbb{P}$ decreases progressively as additional digit classes are included, allowing us to assess how the procedure adapts to alternatives of varying difficulty.

The results of Figure \ref{Fig:power_mnist} should be interpreted with caution since the simulations under the null assumption $H_0$ (see Section \ref{sec: type I error} above) showed that the asymptotic $\chi^2$-based test is not properly calibrated when $n=100$, exhibiting an inflated Type-I error. Consequently, and as expected, its empirical power is artificially higher than that of the proposed non-asymptotic test in this regime. Figure \ref{Fig:power_mnist} shows that empirical power depends on the choice of the distribution $\mathbb{Q}_1,\ldots,\mathbb{Q}_5$. As expected, the more this set differs from $\mathbb{P}$, the higher the power. In most cases, empirical power also increases with the truncation parameter, as more directions characterizing differences between the two distributions are included in the st-nMMD statistic. As the sample-size increases, the empirical power of the non-asymptotic test is indeed lower than the  asymptotic $\chi^2$-based procedure, but without any substantial loss. When $n=5000$, both procedures exhibit essentially identical power performances, confirming that the proposed method remains competitive while ensuring valid finite-sample calibration. 

%%%%%%%%%%%%%%%%%%%%%%
\subsection{Selection of the truncation parameter}\label{sec:selection T}
To assess the performance of the proposed selection algorithm for choosing the truncation parameter $T$, we first examine the empirical distribution of the selected values under the null hypothesis (Figures~\ref{Fig:freq_simulations_T_values} and \ref{Fig:freq_simulations_T_values_mnist_level}) and under the MNIST alternatives (Figure~\ref{Fig:freq_mnist_T_values}). A clear trend emerges from the simulations: under both the null and alternative hypotheses, the selected truncation remains typically small, regardless of the underlying distribution and the sample size. Consistent with Section~\ref{sec: type I error}, the selected parameter $T$ does not depend on the data-generating distribution under the null, but rather on the number of observations and dimensions. When $n$ increases, the selected parameter $T$ increases too. The ambient dimension $d$ also has a pronounced effect on the selection of the parameter $T$. In particular, the selected truncation level decreases as the dimension increases, and remains especially low when $d=100$. This behavior is consistent with the fact that, under the null hypothesis (equality of distributions), the informative signal in higher-order components is negligible. Therefore, selecting a small number of dimensions is theoretically coherent. The empirical level obtained on simulated data remains well controlled when the truncation parameter $T$ is selected by the proposed procedure (Figure~\ref{Fig:level_simulations_T_select}, and \ref{Fig:level_mnist_T_values}), uniformly across dimensions and sample sizes. This confirms that the data-driven selection of $T$ does not compromise the finite-sample validity of the test. Regarding empirical power (Figure~\ref{Fig:power_mnist_T_select}), the same qualitative trends observed when varying $T$ are recovered when $T$ is selected automatically. As expected, since the procedure is primarily designed to guarantee non-asymptotic control of the Type-I error, it may exhibit a slight loss of power in practice compared to asymptotic calibration. Nevertheless, this trade-off appears moderate and remains consistent with the objective of ensuring rigorous finite-sample level control. In particular, empirical power is close to $1$ when $n=5000$.

%%%%%%%%%%%%%%%%%%%%%%
%%%%%%%%%%%%%%%%%%%%%%
\section{Conclusion}
In this work, we investigate kernel-based two-sample testing of equality of distributions. In particular, we propose a new non-asymptotic statistical testing procedure based on the normalized Maximum Mean Discrepancy. While regularization is usually chosen as the ridge penalty, we propose studying the truncated spectral decomposition instead. The test based on the truncated spectrally regularized normalized MMD (st-nMMD) has already been studied theoretically in an asymptotic framework involving the chi-squared quantile but non-asymptotic theoretical approaches were lacking. Still, we show that chi-squared quantiles fail to yield a calibrated test in finite samples. 
We first derive an exponential deviation inequality for the st-nMMD statistic under the null hypothesis, providing a sharp and explicit upper bound on the non-asymptotic quantile and thereby a calibrated test. 
We establish the optimality of the quantile's upper bound secondly in the asymptotic regime, and thirdly propose an estimator of the non-asymptotic quantile involving the dominant terms of the upper-bound. To bridge the gap between asymptotic and non-asymptotic procedures, we emphasize that our estimator is an upward adjustment of the chi-squared one. 
Our estimator is data-adaptive, as it depends on the eigenvalues of the within-group covariance operator and on hyperparameters that we tune directly on the dataset through a practical algorithm that we implement. In contrast to previous works, this algorithm does not require data-splitting. It includes calibration of the truncation parameter, thereby fixing the number of spectral components that retain sufficient information for reliable testing.
Lastly, numerical experiments on both simulated data and the MNIST dataset demonstrate the performance of the st-nMMD procedure. In particular, tests are always calibrated regardless of the distributions, sample sizes or dimensions of the two samples. 
Even though our test is a bit conservative,  
we show, across different configurations under the alternative hypothesis created from the MNIST dataset, that it is competitive in terms of empirical power. 
The data-adaptive nature of our quantile is key to performing well under both the null and the alternative assumptions. 
Finally, our proposed kernel-based two-sample testing of equality of distribution involves spectral truncation as the regularization, avoiding using any methods that separate data; is based on a non-asymptotic approach adapted to the high-dimensional context and to the small or moderate samples; and uses a data-adaptive quantile which adapts according to both null and alternative configurations allowing to be provide a calibrated statistical test with a competitive power. 

As future work, we will naturally study the st-nMMD statistic  under an alternative hypothesis to assess the test's power in a minimax perspective. A perspective of our work can be to relax some assumptions as the well-balanced sample sizes. We note that similar results to this work can be obtained without much effort by relaxing Assumption \A{3} (except for the gap around the truncation, which must remain strictly positive) at the cost of having to account for the multiplicity of eigenvalues, which would make the calculations more complex. 

\acks{
	The authors would like to thank Patrice Bertail, Gilles Blanchard, Martin Wahl for fruitful discussions and input. The research was supported by the project AI4scMed, France 2030
	ANR-22-PESN-0002.
}

\begin{figure}
	\begin{center}
		\includegraphics[scale=0.7]{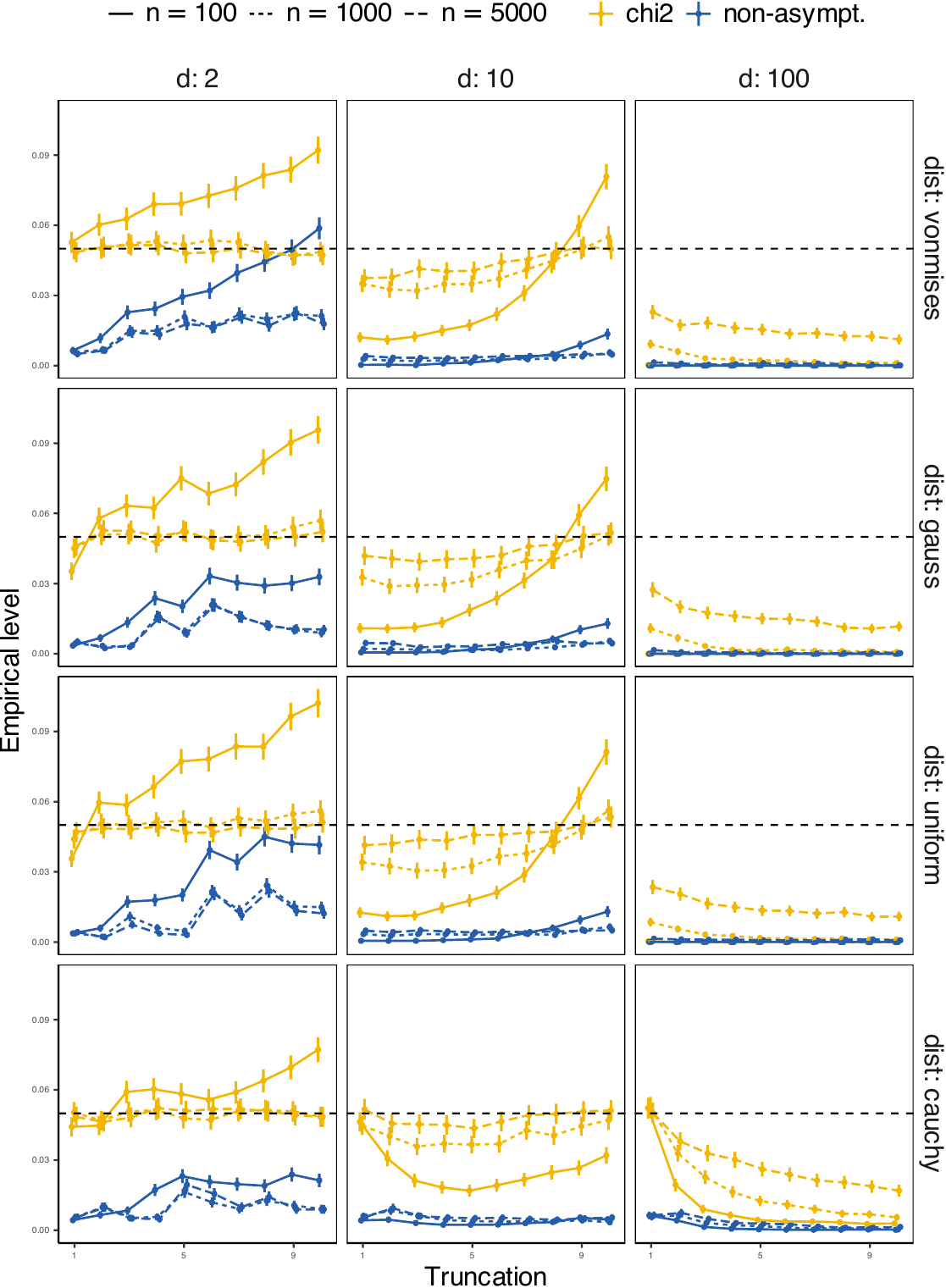}
		\caption{Average empirical level (and 95\% confidence interval) of the st-nMMD test, for varying truncations $T$, for the test based on the asymptotic $\chi^2$ approximation and our non-asymptotic bound. The test is performed at a nominal level $\alpha=0.05$ (black dashed horizontal line), for 4 different distributions and 3 different dimensions $d \in \{2, 10,100\}$. \label{Fig:level_simulations_T_005}}
	\end{center}
\end{figure}

\begin{figure}
	\begin{center}
		\includegraphics[scale=0.7]{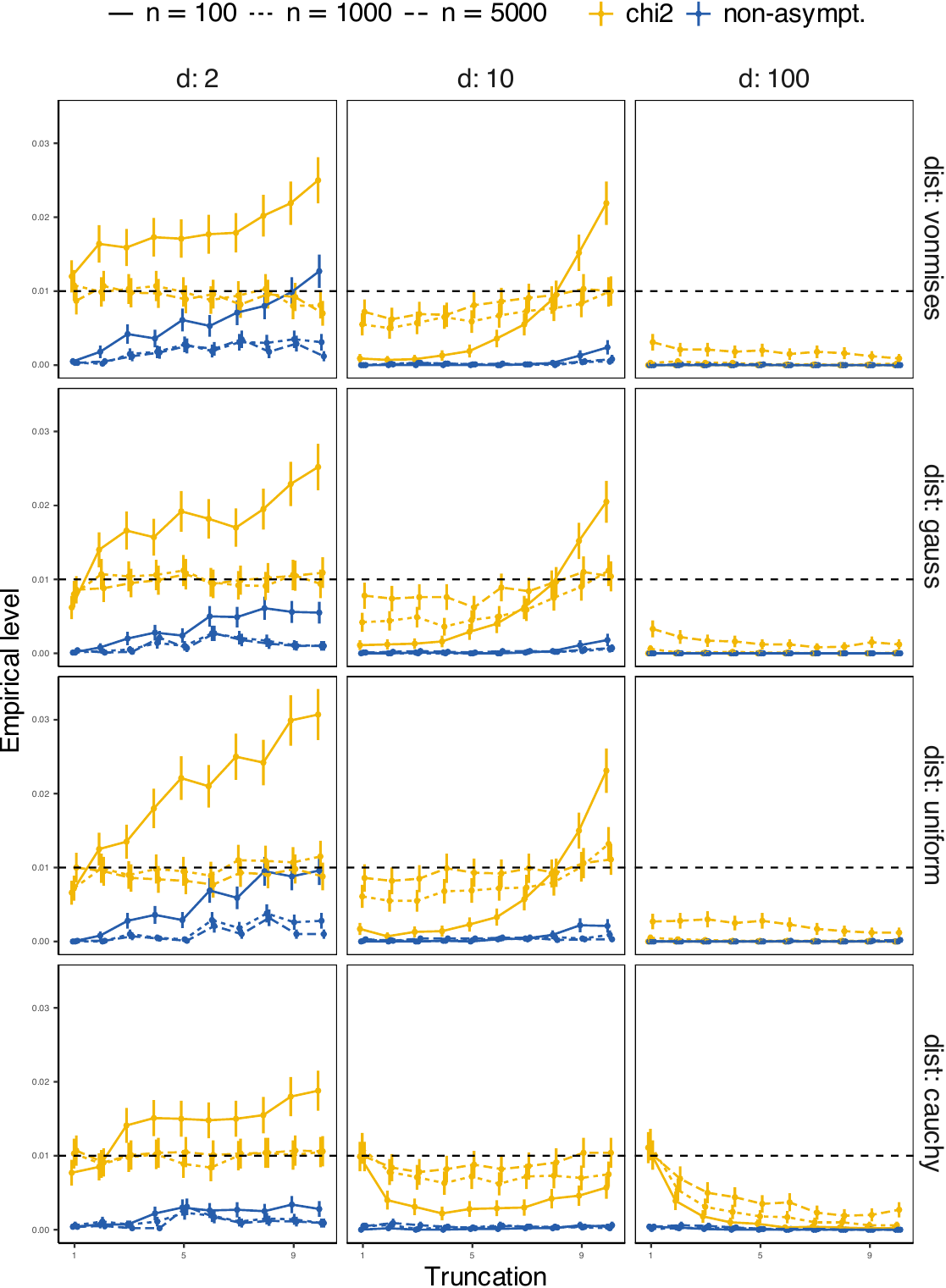}
		\caption{Average empirical level (and 95\% confidence interval) of the st-nMMD test, for varying truncations $T$, for the test based on the asymptotic $\chi^2$ approximation and our non-asymptotic bound. The test is performed at a nominal level $\alpha=0.01$ (black dashed horizontal line), for 4 different distributions and 3 different dimensions $d \in \{2, 10,100\}$. \label{Fig:level_simulations_T_001}}
	\end{center}
\end{figure}

\begin{figure}
	\begin{center}		
		\includegraphics[scale=0.50]{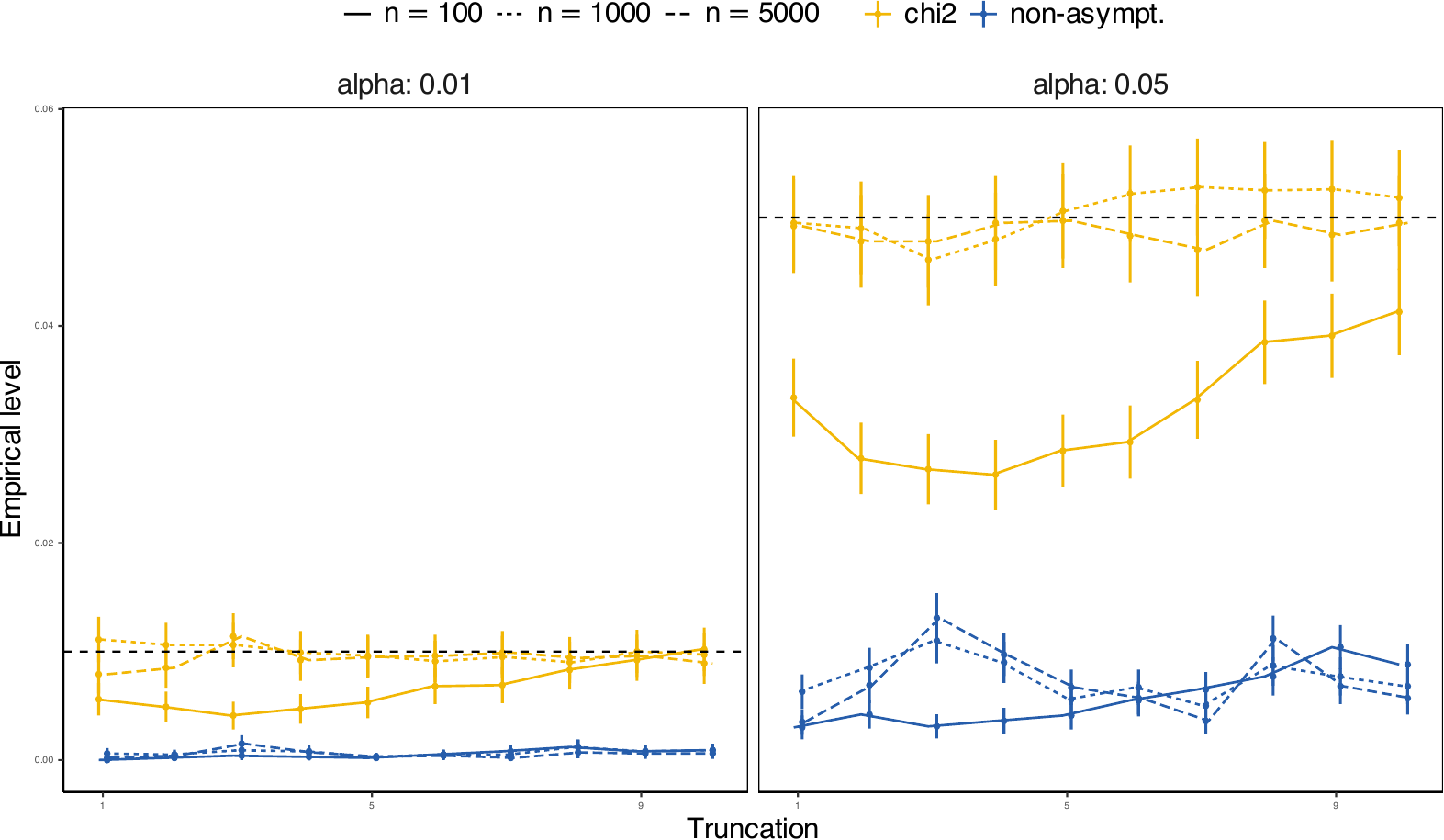}
		\caption{Average empirical level (and 95\% confidence interval) of the st-nMMD test, for varying truncations $T$, for the test based on the asymptotic $\chi^2$ approximation and our non-asymptotic bound. The test is performed at a nominal level $\alpha=0.01$ (left) and $\alpha=0.05$ (right) (black dashed horizontal line), for the MNIST datasets with dimension $d=49$. \label{Fig:level_mnist_T}}
	\end{center}
\end{figure}

\begin{figure}
	\begin{center}
		\includegraphics[scale=0.7]{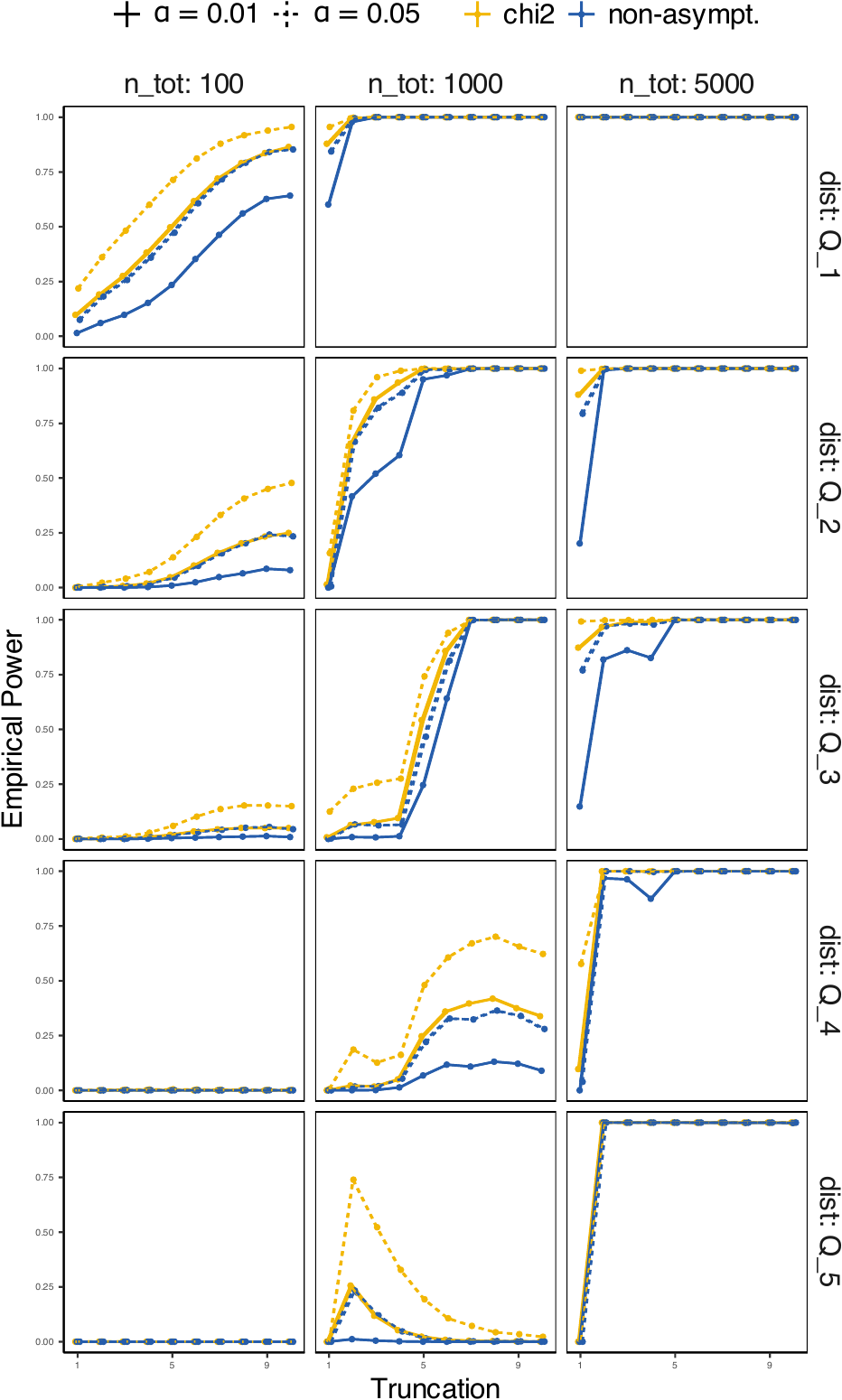}
	\end{center}
	\caption{Average empirical power (and 95\% confidence interval) of the st-nMMD test, for varying truncations $T$, for the test based on the asymptotic $\chi^2$ approximation and our non-asymptotic bound. The test is performed at a nominal level $\alpha=0.05$, for the MNIST datasets with dimension $d=49$. \label{Fig:power_mnist}} 
\end{figure}

\begin{figure}
	\begin{center}
		\includegraphics[scale=0.7]{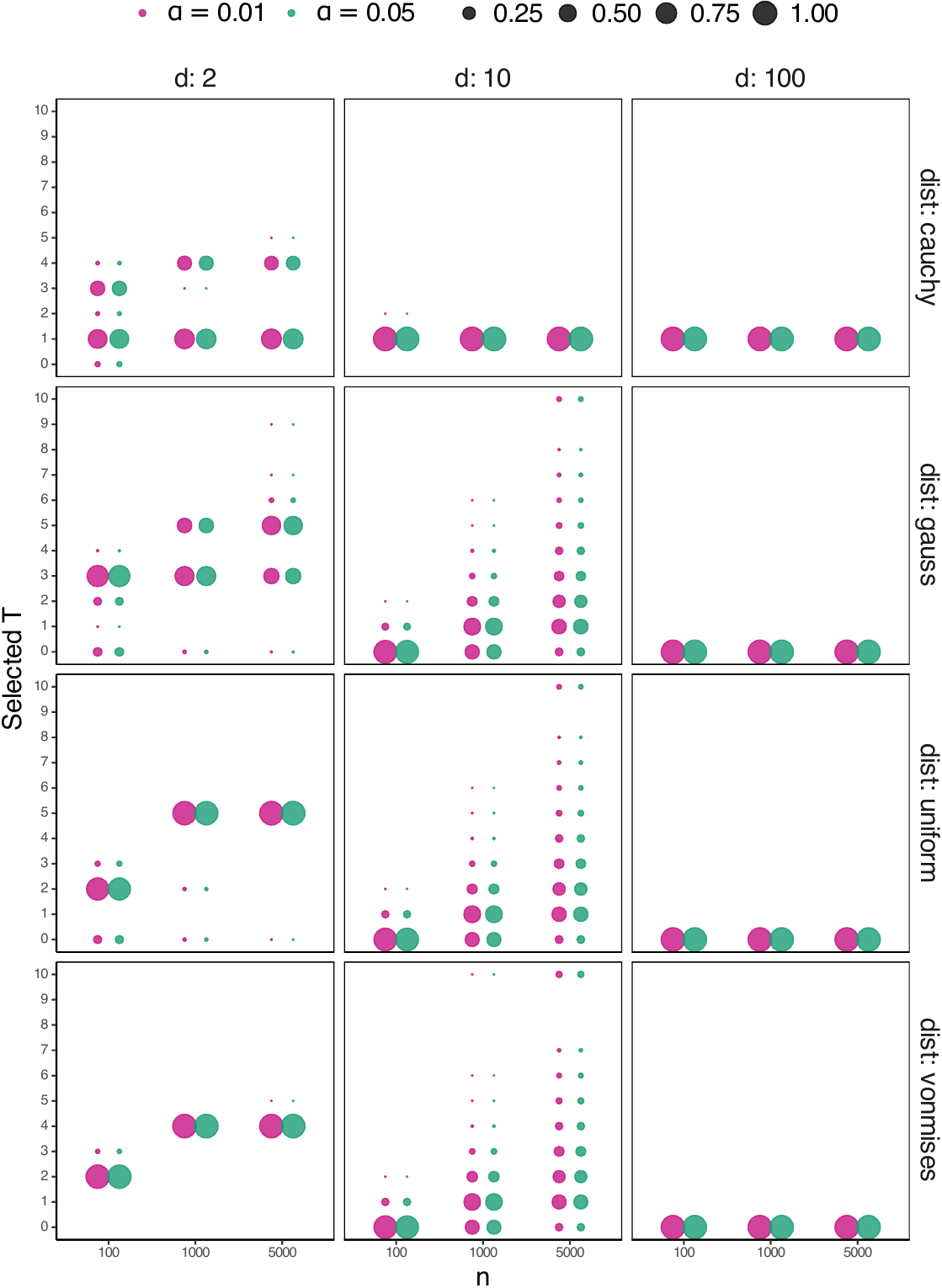}
	\end{center}
		\caption{Frequencies of the truncation parameter $T$ as proposed by our selection method on simulated distributions. The st-nMMD test is performed at a nominal level $\alpha=0.01$ (purple) or $\alpha=0.05$ (green). The sizes of the dots indicate the frequency of each value of $T$ among 10000 simulations. $T=0$ indicates that the null hypothesis is not rejected. \label{Fig:freq_simulations_T_values}}
\end{figure}

\begin{figure}
	\begin{center}
		\includegraphics[scale=0.4]{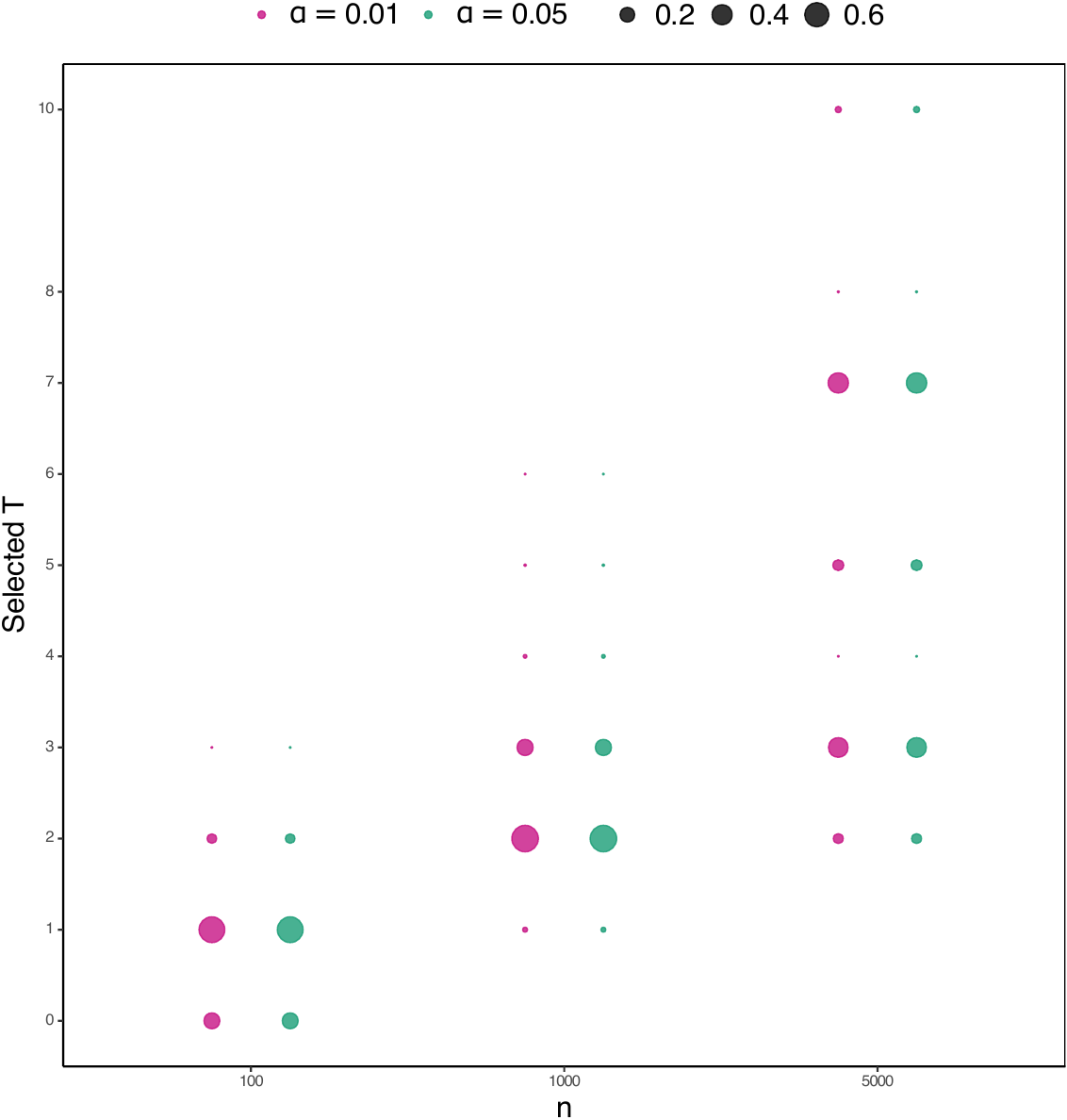}
	\end{center}
	\caption{Frequencies of the truncation parameter $T$ as proposed by our selection method on the MNIST dataset under the null. The st-nMMD test is performed at a nominal level $\alpha=0.01$ (purple) or $\alpha=0.05$ (green). The sizes of the dots indicate the frequency of each value of $T$ among 10000 simulations. $T=0$ indicates that the null hypothesis is not rejected. \label{Fig:freq_simulations_T_values_mnist_level}}
\end{figure}

\begin{figure}
	\begin{center}
		\includegraphics[scale=0.5]{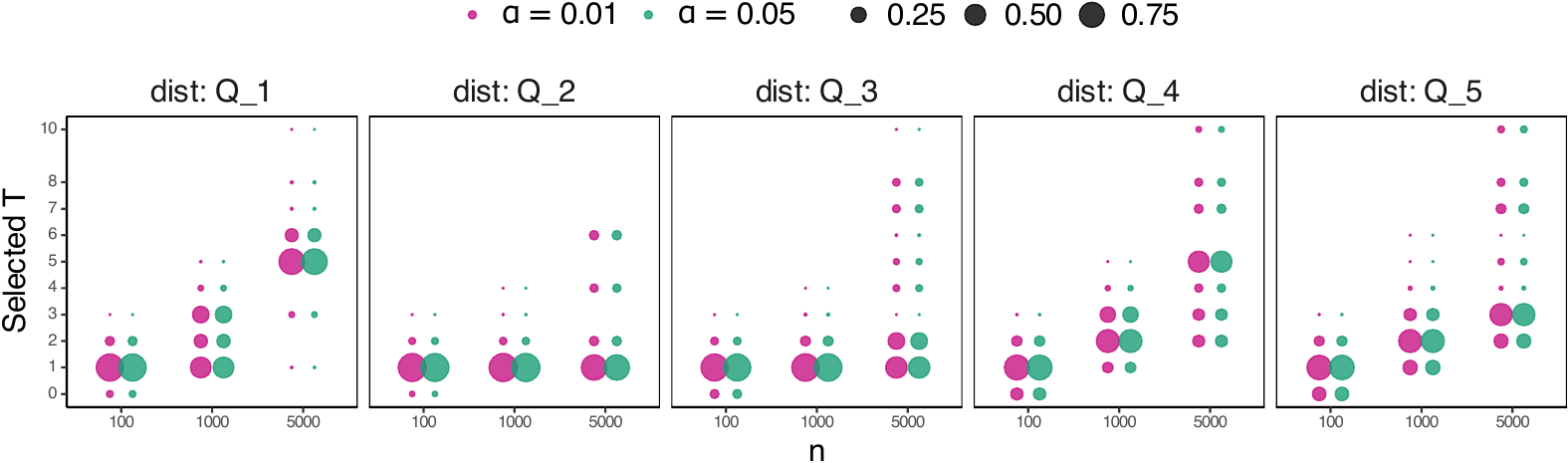}
	\end{center}
	\caption{Frequencies of the truncation parameter $T$ as proposed by our selection method on the MNIST dataset under the alternative. The st-nMMD test is performed at a nominal level $\alpha=0.01$ (purple) or $\alpha=0.05$ (green), for the MNIST datasets with dimension $d = 49$. The sizes of the dots indicate the frequency of each value of $T$ among 10000 simulations. $T=0$ indicates that the null hypothesis is not rejected. \label{Fig:freq_mnist_T_values}}
\end{figure}

\begin{figure}
	\begin{center}
		\includegraphics[scale=0.6]{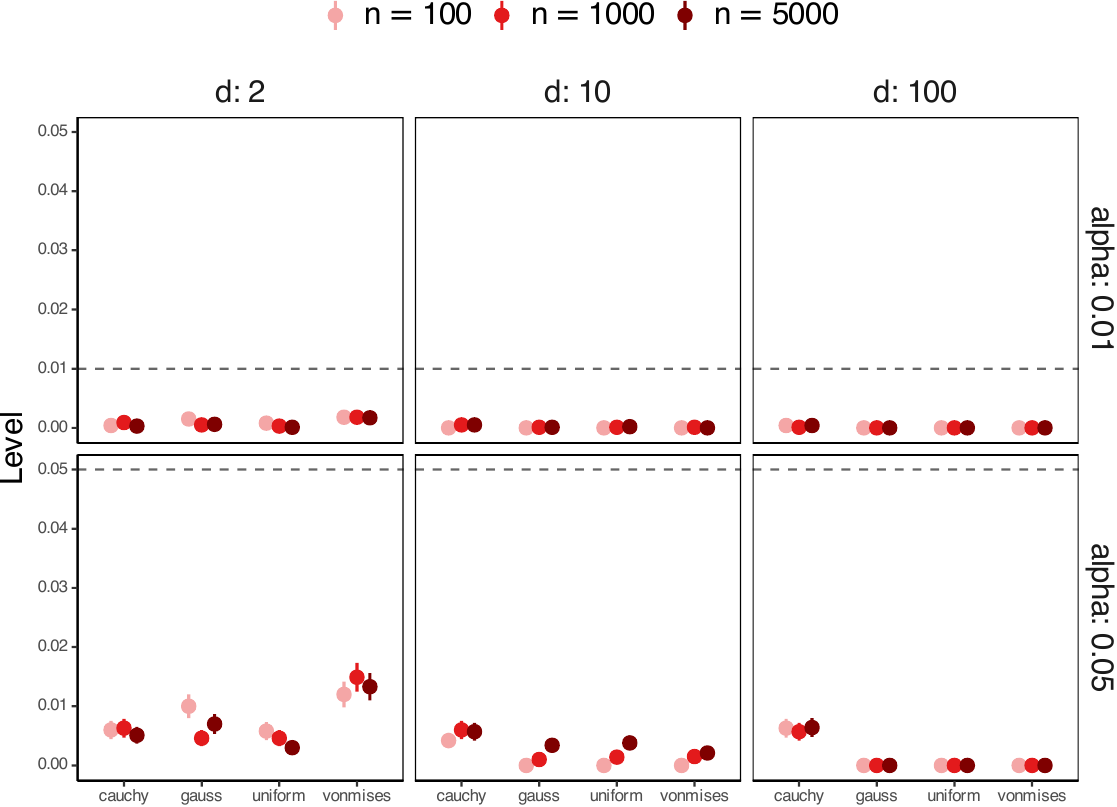}
	\end{center}
	\caption{Average empirical level (and 95\% confidence interval) of the st-nMMD test, for truncation parameter $T$ selected by our algorithm and for the st-nMMD test based on our non-asymptotic bound. The test is performed at a nominal level $\alpha=0.01$ (top) and $\alpha=0.05$ (bottom) (black dashed horizontal line), for 4 different distributions and 3 different dimensions $d \in \{2, 10,100\}$. \label{Fig:level_simulations_T_select}}
\end{figure}

\begin{figure}
	\begin{center}
		\includegraphics[scale=0.6]{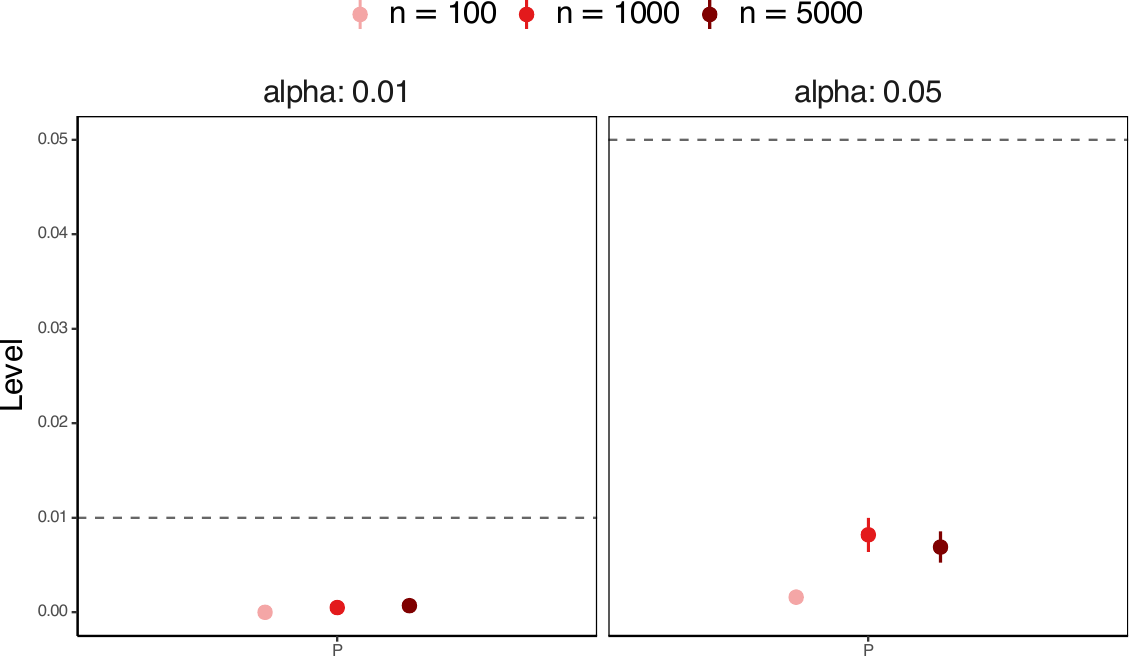}
	\end{center}
	\caption{Average empirical level (and 95\% confidence interval) of the st-nMMD test, for truncation parameter $T$ selected by our algorithm and for the st-nMMD test based on our non-asymptotic bound. The test is performed at a nominal level $\alpha=0.01$ (left) and $\alpha=0.05$ (right) (black dashed horizontal line), for the MNIST datasets under the null ($\mathbb{P}$ distribution) with dimension $d=49$. \label{Fig:level_mnist_T_values}}
\end{figure}

\begin{figure}
	\begin{center}
		\includegraphics[scale=0.6]{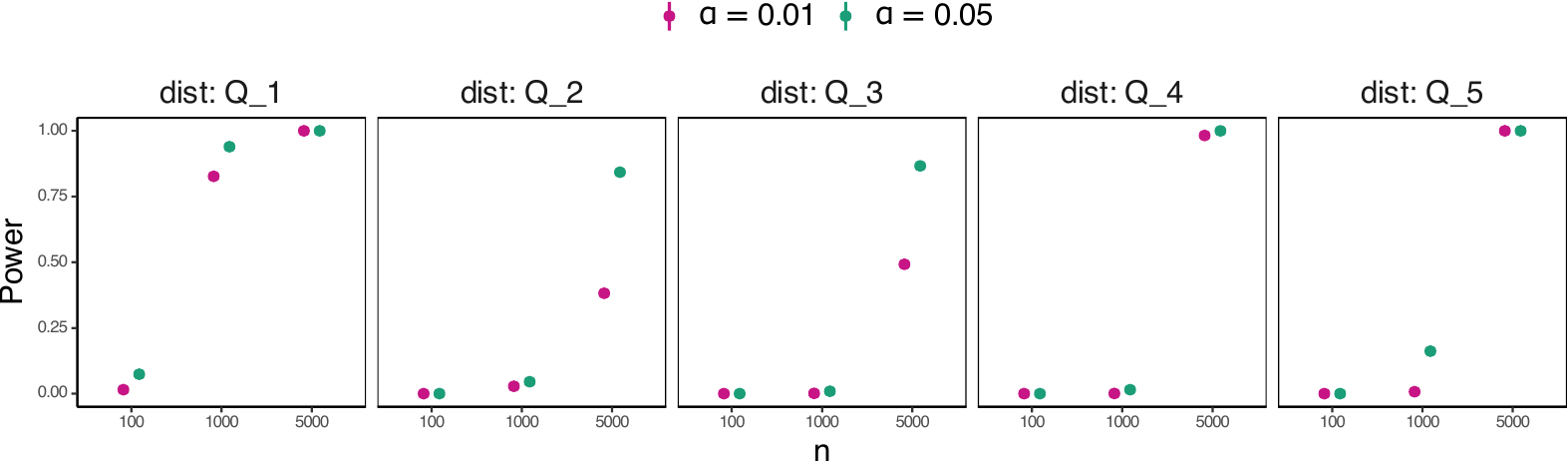}
	\end{center}
	\caption{Average empirical power (and 95\% confidence interval) of the st-nMMD test, for selected truncation parameter $T$. The test is performed at a nominal level $\alpha=0.01$ (purple) and $\alpha=0.05$ (green), for the MNIST datasets with dimension d = 49. \label{Fig:power_mnist_T_select}}
\end{figure}


\section*{\Large Appendix}

\appendix

This appendix is first devoted to all the theoretical proofs (Sections \ref{Appendice_Operators}$\sim$\ref{la_section_proof_asympt}), secondly to additional figures  (Section \ref{sec:additional_figures}) that complete Section \ref{sec:simu}, and thirdly to the existing results we rely on extensively (Section \ref{existing_results}).
%and thirdly to some details about the computation of $\widehat{D}_T^2$ in practice (Section \ref{section_details_computation}). 
%\perrine{à compléter si jamais on rajoute des sections, typiquement les variations de $\beta$ dans l'algo ou d'autres analyses complémentaires sur les données simulées ou MNIST}

\medskip

\medskip

For the following, we denote by $\Pi_{V}$ the orthogonal projector on $V$ a closed subspace of $\Hk$. In particular, for each $t \in \{1,\ldots,T\}$, $\Pi_{f_t}$ and $\Pi_{\widehat{f}_t}$ are the orthogonal projectors onto the subspaces spanned by the eigenfunctions $f_t$ and $\widehat{f}_t$ respectively. 
We recall that we assume $H_0$, so $\PP$ stands for $\PP_{H_0}$.
We recall the homoscedastic assumption : $\Sigma_{X} = \Sigma_{Y} = \Sigma_W$, also equals $\Sigma$ under $H_0$. \\
%Operators $\SW$, $\hS$, $\ST$ and $\hST$ are simply denoted by $\Sigma$, $\widehat{\Sigma}$, $\Sigma_T$ and $\widehat{\Sigma}_T$ respectively. \\
We introduce the feature map $\Phi : \mathcal{Z} \rightarrow \Hk$ such that $\forall (z,z') \in \mathcal{Z}^2$, $\phi(z) = k(z,\cdot)$ and $k(z, z') = \langle \phi(z), \phi(z')\rangle_{\Hk}$. For the following, we work with $\phi$ for simplification task but we recall that given $k(\cdot,\cdot)$ is sufficient to run the test statistic avoiding the explicit determination of $\phi$. 
We note that since $\Hk$ is separable, 
$\Sigma$ is a trace-class operator, and from \A{3}, operators $\Sigma_T^{-\frac{1}{2}}$ and $\widehat{\Sigma}_T^{-\frac{1}{2}}$ are self-adjoint nonnegative trace-class operators (with high probability for $\widehat{\Sigma}_T^{-\frac{1}{2}}$). 

\section{Background on operators}
\label{Appendice_Operators}

Let $(\Hk,\|\cdot\|_{\Hk})$ be a separable Hilbert space endowed with the norm $\|\cdot\|_{\Hk}$. 
An Hilbert-Schmidt operator $C : \Hk \rightarrow \Hk$ is a linear operator if 
$$\|C\|^2_{\HSk} \quad (= \underset{i=1}{\overset{+ \infty}{\sum}} \|Cf_i\|^2_{\Hk} \quad \text{for any orthonormal basis} \ (f_i)_{i \geq 1} \ \text{of} \ \Hk) \quad < + \infty.$$ 
Let $(\HSk,\|\cdot\|_{\HSk})$ be the separable Hilbert space of all Hilbert-Schmidt operators on $\Hk$ endowed with the inner product 
 $$\langle C, T \rangle_{\HSk} = \underset{i=1}{\overset{+ \infty}{\sum}} \langle Cf_i, T f_i\rangle_{\Hk},$$ 
for any orthonormal basis $(f_i)_{i \geq 1}$ of $\Hk$. 
A linear operator $C : \Hk \rightarrow \Hk$ is self-adjoint if $\langle C g, h \rangle_{\Hk} = \langle g, Ch \rangle_{\Hk}$ for any $(g,h) \in \Hk$. 
We recall that $C \in \HSk$ is compact and if $C$ is additionally a positive self-adjoint operator, its eigenvalues are all nonnegative and nonzero and the space of the eigenfunctions is an orthonormal basis of $\Hk$.
An operator $C \in \HSk$ is trace-class if $\underset{i=1}{\overset{+ \infty}{\sum}} |\langle C f_i, f_i \rangle_{\Hk} | < + \infty$ for any orthonormal basis $(f_i)_{i \geq 1}$. In this case, the trace of $C$ is 
$$ \text{Tr}(C) = \underset{i=1}{\overset{+ \infty}{\sum}} \langle C f_i, f_i \rangle_{\Hk},$$ 
and is independent of $(f_i)_{i \geq 1}$. 
A finite rank operator is trace-class, and a trace-class operator is Hilbert-Schmidt. \\
For $(f,g) \in \Hk \backslash \{0\}$, the rank one operator $f \otimes_{\Hk} g \in \HSk$ is defined by 
\begin{equation}
     f \otimes_{\Hk} g : h \in \Hk \mapsto  \langle g , h \rangle_{\Hk}  f \in \Hk,
    \label{otimes_equality}
\end{equation}
and satisfies
\begin{equation}
    \left \| f \otimes_{\Hk} f \right \|_{\HSk} =  \left \| f \right \|^2_{\Hk} \quad \text{and} \quad \left \| f \otimes_{\Hk} g \right \|_{\HSk} \leq  \left \| f \right \|_{\Hk} \left \| g \right \|_{\Hk}.
    \label{norm_inequality}
\end{equation}
We refer the reader to \citep{dunford1963spectral} for an Hilbert space theory reference.
%%%%%%%%%%%%%%%%%%%%%%
%%%%%%%%%%%%%%%%%%%%%%
\section{Proofs of the non-asymptotic results}\label{la_section_proof}
To improve the readability of the proofs, we simplify the notation $k(Z,\cdot)$ (or $k(z,\cdot)$) by $\Phi(Z)$ (or $\Phi(z)$) for $Z$ any random variable (or $z$ any data). 
%%%%%%%%%%%%%%%%%%%%%%
\subsection{Proof of Theorem \ref{the_theorem}}
In this section, we prove the main theorem of the article. 
Let us suppose that \A{1}$\sim$\A{3} and Conditions (\ref{A4bis}) and (\ref{lambda_t_condition}) are all satisfied, and let $\delta>0$. 

Thanks to the well-balanced assumption \A{1}, we can define $W_i := \phi(X_i) - \phi(Y_i)$ for $i \in \{1,\ldots,\nb\}$ such that 
    $$ \widehat{\mu}_X - \widehat{\mu}_Y = \frac{1}{\nb} \underset{i=1}{\overset{\nb}{\sum}} W_i.$$
%%%%%%
\subsubsection*{Reduction to control in the spectral eigendirections}   
Let $x >0$. From (\ref{Hotelling_test_stat}), we get 
    \begin{align}
    \PPO \left(\widehat{D}^2_T > x \right) &= \PPO \left(\frac{\nb}{2} \langle \widehat{\Sigma}_T^{-1} \left(\widehat{\mu}_X - \widehat{\mu}_Y\right), \widehat{\mu}_X - \widehat{\mu}_Y  \rangle_{\Hk}> x \right) \notag \\
    &=  \PPO \left( \frac{\nb}{2} \langle \underset{t=1}{\overset{T}{\sum}} \widehat{\lambda_t}^{-1} \langle \widehat{f}_t , \widehat{\mu}_X - \widehat{\mu}_Y \rangle_{\Hk}  \widehat{f}_t , \widehat{\mu}_X - \widehat{\mu}_Y  \rangle_{\Hk}> x \right) \notag \\    
    &= \PPO \left( \underset{t=1}{\overset{T}{\sum}} \frac{ \frac{\nb}{2} \langle \widehat{f}_t ,\widehat{\mu}_X - \widehat{\mu}_Y \rangle_{\Hk}^2}{\widehat{\lambda_t}}  > x \right). \notag 
    \end{align}
Therefore,
\begin{align}
 \PPO \left(\widehat{D}^2_T > x \right)    &\leq \underset{t=1}{\overset{T}{\sum}} \ \PPO \left( \frac{\frac{\nb}{2}\langle \widehat{f}_t ,\widehat{\mu}_X - \widehat{\mu}_Y \rangle_{\Hk}^2}{\widehat{\lambda_t}}  > \frac{x}{T} \right) \notag \\   
    &= \underset{t=1}{\overset{T}{\sum}} \ \PPO \left( \frac{ \frac{\nb}{2} \left(  \frac{1}{\nb} \underset{i=1}{\overset{\nb}{\sum}} \langle \widehat{f}_t , W_i \rangle_{\Hk} \right)^2}{\widehat{\lambda_t}}  > \frac{x}{T} \right),
    \label{until_Hoeffding}
\end{align}
where the first equality comes from the fact that $\widehat{\Sigma}_T^{-\frac{1}{2}}$ is a self-adjoint operator and the second equality follows from Equation~(\ref{otimes_equality}).
%%%%%%%%%%%
\subsubsection*{Reduction to a sum of self-normalized processes}  
Let us define the following events for all $t \in \{1,\ldots,T\}$ 
\begin{align*}
    E_1 &=  \left\{ \left \|  \widehat{\mu}_X - \widehat{\mu}_Y \right \|_{\Hk} \leq  2 \left( 2\sqrt{\frac{M_k}{\nb}} + \sqrt{\frac{M_k\delta}{\nb}} \right) \right\} \\
    E_2 &= \left\{ \left \| \widehat{\Sigma} - \Sigma \right \|_{\HSk} \leq \frac{6 M_k}{\sqrt{\nb}} \left( 1 + \sqrt{\frac{\delta}{2}}\right) \right\} \\
    E_{3,t} &= \left\{ \frac{1}{\nb} \underset{i=1}{\overset{\nb}{\sum}} \langle f_t, \widehat{\mu}_X - \phi(X_i) \rangle_{\Hk} \langle f_t,\phi(Y_i) - \widehat{\mu}_Y  \rangle_{\Hk} \leq 8 M_k \sqrt{\frac \delta \nb}   \right\} \\
    E_{4,t} &=  \left\{ \frac{1}{\nb} \underset{i=1}{\overset{\nb}{\sum}} \langle f_t, W_i \rangle_{\Hk}^2 \geq - 8 M_k \sqrt{\frac{\delta}{\nb}} + 2 \lambda_t \right\} \\
    E_5 &=  \left\{ \left\| \frac{1}{\nb} \underset{i=1}{\overset{\nb}{\sum}} W_i\right\|_{\Hk} \leq \sqrt{\frac{2 M_k}{\nb}} \left(1+\sqrt{2\delta}\right) \right\}.
\end{align*}
The next computations consist in decomposing the probabilities in~\eqref{until_Hoeffding} over the events $E_1$, $E_2$, $E_{3,t}$, $E_{4,t}$ and $E_5$ and their complementary events, for some fixed $t \in \{1, \dots T\}$. According to Theorem~\ref{Gretton} from Section \ref{existing_results}  and Corollary~\ref{Blanchard_for_SigmaW}, Lemma~\ref{covariance_term}, Lemma~\ref{num_den_term} and Lemma~\ref{lemma_norm_Z_bar} from Subsection \ref{Preliminary_results}, we have that the probabilities of   $E_1$, $E_2$, $E_{3,t}$, $E_{4,t}$ and $E_5$ are all less than $ e^{-\delta}$ or $2e^{-\delta}$ . Thus,
\begin{equation}
    \PPO \left(\widehat{D}^2_T > x \right)  
    \leq  \underset{t=1}{\overset{T}{\sum}} \ \PPO (F_t ) + 7 T e^{-\delta}.
    \label{splitting}  
\end{equation}
where 
$$F_t = \left\{ \frac{ \frac{\nb}{2} \left(  \frac{1}{\nb} \underset{i=1}{\overset{\nb}{\sum}} \langle \widehat{f}_t , W_i \rangle_{\Hk} \right)^2}{\widehat{\lambda_t}}  > \frac{x}{T}  \right\}\bigcap E_1 \bigcap E_2  \bigcap E_{3,t} \bigcap E_{4,t} \bigcap E_5.$$
In the following, note that each of the inequalities holds, since all denominators are strictly positive under assumption (\ref{lambda_t_condition}), which ensures the validity of Lemma \ref{HP_positivity}. 
Moreover, for ease of presentation, we upper bound each term $\frac{\nb-1}{\nb}$ and $\frac{\nb-1}{\nb^{\frac{3}{2}}}$ by $1$ and $\frac{1}{\sqrt{\nb}}$ respectively, but note that to get the sharpest non-asymptotic $x$ formula in (\ref{x_bound}), we keeped all the terms in $\nb$ until the end of the proof. 
Let $s = \sign(\langle f_t, \hat f_t\rangle)$. According to Lemma~\ref{negligeble_terms} together with the events $E_1$ and $E_{3,t}$, and by applying also Proposition~\ref{Zwald_Blanchard_Prop}, we get 
\begin{align*}
    \widehat{\lambda}_t  &\geq \frac{1}{2\nb} \underset{i=1}{\overset{\nb}{\sum}} \langle \widehat{f}_t, W_i\rangle_{\Hk}^2 - \frac{4 \left( 2\sqrt{\frac{M_k}{\nb}} + \sqrt{\frac{M_k\delta}{\nb}} \right)^2}{2}  - 8  M_k \left \| \Pi_{s f_t} - \Pi_{\widehat{f}_t} \right \|_{\HSk} - 8 M_k \sqrt{\frac \delta \nb}.
\end{align*}
Since $ \left \{ \left \| \widehat{\Sigma} - \Sigma \right \|_{\HSk} < \frac{\Delta_t}{2} \right \}$ under the gap condition (\ref{A4bis}), then according to  Corollary~\ref{P_V_theo}  we have 
\begin{align} 
 F_t   &   \subset     \left\{ \frac{ \frac{\nb}{2} \left(  \frac{1}{\nb} \underset{i=1}{\overset{\nb}{\sum}} \langle \widehat{f}_t , W_i \rangle_{\Hk} \right)^2}{\frac{1}{2\nb} \underset{i=1}{\overset{\nb}{\sum}} \langle \widehat{f}_t, W_i\rangle_{\Hk}^2 - 2\frac{M_k}{\nb} \left( 2 + \sqrt{\delta} \right)^2 - 16  M_k  \frac{\left \| \widehat{\Sigma} - \Sigma \right \|_{\HSk} }{\Delta_t}   - 8 M_k  \sqrt{\frac \delta \nb} }  > \frac{x}{T}  \right\}  \cap E_2  \cap E_{4,t}  \cap E_5  \notag   \\
    &  \subset      \left\{ \frac{ \frac{\nb}{2} \left(  \frac{1}{\nb} \underset{i=1}{\overset{\nb}{\sum}} \langle \widehat{f}_t , W_i \rangle_{\Hk} \right)^2}{\frac{1}{2\nb} \underset{i=1}{\overset{\nb}{\sum}} \langle \widehat{f}_t, W_i\rangle_{\Hk}^2 - 2 \frac{M_k}{\nb} \left( 2 + \sqrt{\delta} \right)^2  - 16   M_k \left(\frac{6 M_k \left( 1 + \sqrt{\frac{\delta}{2}}\right)}{\Delta_t \sqrt{\nb}} \right) - 8 M_k \sqrt{\frac \delta \nb} }  > \frac{x}{T}  \right\}    \cap E_2  \cap E_{4,t}  \cap E_5  \notag   \\
    &  \subset  \left\{ \frac{ \frac{\nb}{2} \left(  \frac{1}{\nb} \underset{i=1}{\overset{\nb}{\sum}} \langle \widehat{f}_t , W_i \rangle_{\Hk} \right)^2}{\frac{1}{2\nb} \underset{i=1}{\overset{\nb}{\sum}} \langle \widehat{f}_t, W_i\rangle_{\Hk}^2 - 2\frac{M_k}{\nb} \left( 2 + \sqrt{\delta} \right)^2  - \frac{96}{\Delta_t}  \frac{M_k^2}{\sqrt{\nb}} \left( 1 + \sqrt{\frac{\delta}{2}} \right) -   8 M_k \sqrt{\frac \delta \nb}}   > \frac{x}{T}  \right\}        \cap E_2  \cap E_{4,t}  \cap E_5   \label{until_Bertail_constant} ,
\end{align}
where the second inclusion follows because $E_2$ is satisfied in $F_t$. \\
For ease of presentation, we introduce for each $t \in \{1,\ldots,T\}$ the notation
$$R_{\nb} =  R_{\nb}\left(\delta,M_k,\Delta_t\right) := 
2\frac{M_k}{\nb} \left( 2 + \sqrt{\delta} \right)^2  + \frac{96}{\Delta_t}  \frac{M_k^2}{\sqrt{\nb}} \left( 1 + \sqrt{\frac{\delta}{2}} \right) +   8 M_k \sqrt{\frac \delta \nb}.
$$
Note that on $  E_2    \cap E_5 $ we have 
\begin{align}
\frac{1}{\nb} \underset{i=1}{\overset{\nb}{\sum}}  \langle \widehat{f}_t , W_i \rangle_{\Hk}^2 & = 
\frac{1}{\nb} \underset{i=1}{\overset{\nb}{\sum}} \langle s f_t , W_i \rangle_{\Hk}^2 - \frac{2}{\nb} \underset{i=1}{\overset{\nb}{\sum}} \left( \langle s f_t - \widehat{f}_t, W_i \rangle_{\Hk} \langle s f_t, W_i \rangle_{\Hk} \right)  \notag \\
& \geq \frac{1}{\nb} \underset{i=1}{\overset{\nb}{\sum}} \langle f_t , W_i \rangle_{\Hk}^2 -  2 \left | \langle s f_t - \widehat{f}_t, \frac{1}{\nb} \underset{i=1}{\overset{\nb}{\sum}} W_i \rangle_{\Hk} \right | \left \| s 	f_t \right \|_{\Hk}  \underset{ 1\leq i \leq \nb}{\max} \left \| W_i \right \|_{\Hk} \notag  \\
 & \geq  \frac{1}{\nb} \underset{i=1}{\overset{\nb}{\sum}} \langle f_t , W_i \rangle_{\Hk}^2  -  2 \left \| s f_t - \widehat{f}_t \right \|_{\Hk} \sqrt{\frac{2 M_k}{\nb}} \left(1+\sqrt{2\delta}\right) 2 \sqrt{M_k} \notag  \\
 & \geq \frac{1}{\nb} \underset{i=1}{\overset{\nb}{\sum}} \langle f_t , W_i \rangle_{\Hk}^2 -  \frac{4 \sqrt{2} M_k \left(1 +\sqrt{2\delta} \right)}{\sqrt{\nb}} \left \| \Pi_{s f_t} - \Pi_{\widehat{f}_t} \right \|_{\HSk} \notag  \\
  & \geq \frac{1}{\nb} \underset{i=1}{\overset{\nb}{\sum}} \langle f_t , W_i \rangle_{\Hk}^2 -  \frac{8 \sqrt{2} M_k \left(1 +\sqrt{2\delta} \right)}{\sqrt{\nb} \Delta_t} \left \| \widehat{\Sigma} - \Sigma \right \|_{\HSk},  \label{minorangle}
\end{align}
providing from Lemma \ref{kernel_bound}, Proposition~\ref{Zwald_Blanchard_Prop} and Corollary \ref{P_V_theo}, since the event $E_2$ is included in the event $ \left \{ \left \| \widehat{\Sigma} - \Sigma \right \|_{\HSk} < \frac{\Delta_t}{2} \right \}$ from  the gap condition (\ref{A4bis}). We then derive from \eqref{until_Bertail_constant}, \eqref{minorangle} and $E_2$ that we also have on $F_t$
$$
  \frac{\nb \left(  \frac{1}{\nb} \underset{i=1}{\overset{\nb}{\sum}} \langle \widehat{f}_t , W_i \rangle_{\Hk} \right)^2}{\frac{1}{\nb} \underset{i=1}{\overset{\nb}{\sum}} \langle f_t , W_i \rangle_{\Hk}^2 - \frac{48 \sqrt{2} M_k^2 \left(1+ \sqrt{2\delta}\right) \left(1+ \sqrt{\frac{\delta}{2}}\right)}{\nb \Delta_t} - 2 R_{\nb}}  > \frac{x}{T}  ,
  $$
so
  $$ \left |  \frac{1}{\nb} \underset{i=1}{\overset{\nb}{\sum}} \langle \widehat{f}_t , W_i \rangle_{\Hk} \right |  
  > \sqrt{\frac{x}{\nb T} \left( \frac{1}{\nb} \underset{i=1}{\overset{\nb}{\sum}} \langle f_t , W_i \rangle_{\Hk}^2  - \frac{48 \sqrt{2} M_k^2 \left(1+ \sqrt{2\delta}\right) \left(1+ \sqrt{\frac{\delta}{2}}\right)}{\nb \Delta_t} - 2 R_{\nb} \right)}    ,
  $$
 and 
\begin{multline*}   
\left |  \frac{1}{\nb} \underset{i=1}{\overset{\nb}{\sum}} \langle s f_t , W_i \rangle_{\Hk} \right |  +  \left \| \widehat{f}_t - s f_t \right \|_{\Hk} \left \| \frac{1}{\nb} \underset{i=1}{\overset{\nb}{\sum}} W_i \right \|_{\Hk}   \\
>  \sqrt{\frac{x}{\nb T} \left( \frac{1}{\nb} \underset{i=1}{\overset{\nb}{\sum}} \langle f_t , W_i \rangle_{\Hk}^2  - \frac{48 \sqrt{2} M_k^2 \left(1+ \sqrt{2\delta}\right) \left(1+ \sqrt{\frac{\delta}{2}}\right)}{\nb \Delta_t} - 2 R_{\nb} \right)} .
\end{multline*}  
Using again Proposition~\ref{Zwald_Blanchard_Prop} and Corollary \ref{P_V_theo} and the events $E_2$ and $E_5$, we obtain that on $F_t$ 
\begin{multline*}   
  \left |  \frac{1}{\nb} \underset{i=1}{\overset{\nb}{\sum}} \langle f_t , W_i \rangle_{\Hk} \right |  +  \frac{12 \sqrt{2} M_k^{\frac{3}{2}} \left(1+ \sqrt{2\delta}\right) \left(1+ \sqrt{\frac{\delta}{2}}\right)}{\nb \Delta_t}  \\
    >    \sqrt{\frac{x}{\nb T} \left( \frac{1}{\nb} \underset{i=1}{\overset{\nb}{\sum}} \langle f_t , W_i \rangle_{\Hk}^2  - \frac{48 \sqrt{2} M_k^2 \left(1+ \sqrt{2\delta}\right) \left(1+ \sqrt{\frac{\delta}{2}}\right)}{\nb \Delta_t} - 2 R_{\nb} \right)}.
\end{multline*}  
We can now rewrite this last equation, which is satisfied on $F_t$, in term of an auto-normalized process 
\begin{multline*} 
  \frac{\left |  \frac{1}{\nb} \underset{i=1}{\overset{\nb}{\sum}} \langle f_t , W_i \rangle_{\Hk} \right |}{\sqrt{\frac{1}{\nb} \underset{i=1}{\overset{\nb}{\sum}} \langle f_t , W_i \rangle_{\Hk}^2}}  >  
 \sqrt{\frac{x}{\nb T}} \sqrt{1 - \frac{\frac{48 \sqrt{2} M_k^2 \left(1+ \sqrt{2\delta}\right) \left(1+ \sqrt{\frac{\delta}{2}}\right)}{\nb \Delta_t} + 2 R_{\nb} }{\frac{1}{\nb} \underset{i=1}{\overset{\nb}{\sum}} \langle f_t , W_i \rangle_{\Hk}^2}} -  \frac{12 \sqrt{2} M_k^{\frac{3}{2}} \left(1+ \sqrt{2\delta}\right) \left(1+ \sqrt{\frac{\delta}{2}}\right)}{\nb \Delta_t \sqrt{\frac{1}{\nb} \underset{i=1}{\overset{\nb}{\sum}} \langle f_t , W_i \rangle_{\Hk}^2}}.
\end{multline*} 
Since $E_{4,t}$ is satisfied on $F_t$, we thus have on $F_t$ that
\begin{multline*}   
  \frac{\left |  \frac{1}{\nb} \underset{i=1}{\overset{\nb}{\sum}} \langle f_t , W_i \rangle_{\Hk} \right |}{\sqrt{\frac{1}{\nb} \underset{i=1}{\overset{\nb}{\sum}} \langle f_t , W_i \rangle_{\Hk}^2}}  
>  \sqrt{\frac{x}{\nb T}} \sqrt{1 - \frac{48 \sqrt{2} M_k^2 \left(1+ \sqrt{2\delta}\right) \left(1+ \sqrt{\frac{\delta}{2}}\right) + 2 \nb \Delta_t R_{\nb} }{\nb \Delta_t \left(- 8 M_k \sqrt{\frac{\delta}{\nb}} + 2 \lambda_t\right)}}  \\
-  \frac{12 \sqrt{2} M_k^{\frac{3}{2}} \left(1+ \sqrt{2\delta}\right) \left(1+ \sqrt{\frac{\delta}{2}}\right)}{\nb \Delta_t \sqrt{- 8 M_k \sqrt{\frac{\delta}{\nb}} + 2 \lambda_t}}  .
 \end{multline*}
 As $W_i$ has a symmetric distribution, it gives that
\begin{multline}  
  \PPO (F_t)   \leq  2\PPO \Bigg( \frac{  \frac{1}{\nb} \underset{i=1}{\overset{\nb}{\sum}} \langle f_t , W_i \rangle_{\Hk}}{\sqrt{\frac{1}{\nb} \underset{i=1}{\overset{\nb}{\sum}} \langle f_t , W_i \rangle_{\Hk}^2}}  
     >   \sqrt{\frac{x}{\nb T}} \sqrt{1 - \frac{48 \sqrt{2} M_k^2 \left(1+ \sqrt{2\delta}\right) \left(1+ \sqrt{\frac{\delta}{2}}\right) + 2 \nb \Delta_t R_{\nb} }{\nb \Delta_t \left(- 8 M_k \sqrt{\frac{\delta}{\nb}} + 2 \lambda_t\right)}} \\
     - \frac{12 \sqrt{2} M_k^{\frac{3}{2}} \left(1+ \sqrt{2\delta}\right) \left(1+ \sqrt{\frac{\delta}{2}}\right)}{\nb \Delta_t \sqrt{- 8 M_k \sqrt{\frac{\delta}{\nb}} + 2 \lambda_t}} \Bigg).
    \label{denominator_bound} 
  \end{multline}
%%%%%%%%%%%
\subsubsection*{Hoeffing's inequality for self-normalized processes} 
The left term into the probability in (\ref{denominator_bound}) is now exactly a sum of self-normalized processes and the rest of the proof is inspired by the work of~\citep{bertail2008exponential}. 
Now, we consider $\nb$ Rademacher random variables $\left(\sigma_i\right)_{i \in \{1,\ldots,\nb\}}$ independent from $\left(W_i\right)_{i \in \{1,\ldots,\nb\}}$.
We define $\sigma_{\nb}(W) := \frac{1}{\nb} \underset{i=1}{\overset{\nb}{\sum}} \sigma_{i} W_i$. 
In particular, $\PP(\sigma_i = 1) =  \PP(\sigma_i = -1) = \frac{1}{2}, \ \PP(\sigma_i^2 = 1) = 1, \ \sigma_{i} W_i \overset{\mathcal{L}}{=} W_i$ and  $\sigma_{\nb}(W) \overset{\mathcal{L}}{=} \frac{1}{\nb} \underset{i=1}{\overset{\nb}{\sum}} W_i$ from independence and symmetry of the $W_i$'s and independence with the $\sigma_i$'s.  

Restarting from (\ref{denominator_bound}), we find that 
\begin{align}
 \PPO (F_t) & \leq 
   2 \PPO \Bigg( \frac{ \frac{1}{\nb} \underset{i=1}{\overset{\nb}{\sum}} \langle f_t ,W_i \rangle_{\Hk} \sigma_i}{\sqrt{\frac{1}{\nb} \underset{i=1}{\overset{\nb}{\sum}} \langle f_t , W_i \rangle_{\Hk}^2}}  \notag >   \sqrt{\frac{x}{\nb T}} \sqrt{1 - \frac{48 \sqrt{2} M_k^2 \left(1+ \sqrt{2\delta}\right) \left(1+ \sqrt{\frac{\delta}{2}}\right) + 2 \nb \Delta_t R_{\nb} }{\nb \Delta_t \left(- 8 M_k \sqrt{\frac{\delta}{\nb}} + 2 \lambda_t\right)}} \\
    & \hspace{2 cm}  - \frac{12 \sqrt{2} M_k^{\frac{3}{2}} \left(1+ \sqrt{2\delta}\right) \left(1+ \sqrt{\frac{\delta}{2}}\right)}{\nb \Delta_t \sqrt{- 8 M_k \sqrt{\frac{\delta}{\nb}} + 2 \lambda_t}} \Bigg)  \notag \\
    &   =  2 \mathbb{E}_{W} \Bigg[ \PPO \Bigg(\underset{i=1}{\overset{\nb}{\sum}} \frac{\frac{1}{\nb} \langle f_t ,W_i \rangle_{\Hk}}{\sqrt{\frac{1}{\nb} \underset{i=1}{\overset{\nb}{\sum}} \langle f_t , W_i \rangle_{\Hk}^2}} \sigma_i      >
     \sqrt{\frac{x}{\nb T}} \sqrt{1 - \frac{48 \sqrt{2} M_k^2 \left(1+ \sqrt{2\delta}\right) \left(1+ \sqrt{\frac{\delta}{2}}\right) + 2 \nb \Delta_t R_{\nb}}{\nb \Delta_t \left(- 8 M_k \sqrt{\frac{\delta}{\nb}} + 2 \lambda_t\right)}} \notag \\
    & \hspace{2 cm} 
    - \frac{12 \sqrt{2} M_k^{\frac{3}{2}} \left(1+ \sqrt{2\delta}\right) \left(1+ \sqrt{\frac{\delta}{2}}\right)}{\nb \Delta_t \sqrt{- 8 M_k \sqrt{\frac{\delta}{\nb}} + 2 \lambda_t}} \vert \left(W_i\right)_{i \in \{1,\ldots,n\}} \Bigg) \Bigg].
    \label{just_before_Hoeffding}
\end{align}

The final step of the proof is to apply the Hoeffding's inequality (Theorem \ref{Hoeffding_inequality}) to each unidimensional self-normalized term $S_i = c_i \sigma_i \vert \left(W_i\right)_{i \in \{1,\ldots,\nb\}}$ for all $i \in \{1,\ldots,\nb\}$ where $c_i :=  \text{sign}(\langle f_t , W_i \rangle_{\Hk}) \times \frac{\frac{1}{\nb} \langle f_t , W_i \rangle_{\Hk}}{\sqrt{\frac{1}{\nb} \underset{i=1}{\overset{\nb}{\sum}} \langle f_t , W_i \rangle_{\Hk}^2}}$, a constant conditionally to $\left(W_i\right)_{i \in \{1,\ldots,\nb\}}$. The Hoeffding's inequality is valid here since (i) the $S_i$'s are independent and centered from properties on the $\sigma_i$'s, (ii) $\forall i \in \{1,\ldots,\nb\}, \ a_i \leq S_i \leq b_i$ with $a_i =-c_i$ and $b_i =c_i$ and (iii) the term on the right in the probability in (\ref{just_before_Hoeffding}) is positive from Lemma \ref{condition_Hoeffding}. \\
So, by applying Hoeffding's Inequality on (\ref{just_before_Hoeffding}), it gives that $\PPO (F_t) \leq 2 E_W \left [e^{- (\star)} \right]$ where 
\begin{eqnarray*} 
  (\star)   &=&   \frac{2 \left(\sqrt{\frac{x}{\nb T}} \sqrt{1 - \frac{48 \sqrt{2} M_k^2 \left(1+ \sqrt{2\delta}\right) \left(1+ \sqrt{\frac{\delta}{2}}\right) + 2 \nb \Delta_t R_{\nb}}{\nb \Delta_t \left(- 8 M_k \sqrt{\frac{\delta}{\nb}} + 2 \lambda_t\right)}} - \frac{12 \sqrt{2} M_k^{\frac{3}{2}} \left(1+ \sqrt{2\delta}\right) \left(1+ \sqrt{\frac{\delta}{2}}\right)}{\nb \Delta_t \sqrt{- 8 M_k \sqrt{\frac{\delta}{\nb}} + 2 \lambda_t}} \right)^2}{4 \underset{i=1}{\overset{\nb}{\sum}} \left(\frac{\frac{1}{\nb} \langle f_t , W_i \rangle_{\Hk}}{\sqrt{\frac{1}{\nb} \underset{i=1}{\overset{\nb}{\sum}} \langle f_t , W_i \rangle_{\Hk}^2}}\right)^2} \\
   &= & \frac{\nb}{2} \left(\sqrt{\frac{x}{\nb T}} \sqrt{1 - \frac{48 \sqrt{2} M_k^2 \left(1+ \sqrt{2\delta}\right) \left(1+ \sqrt{\frac{\delta}{2}}\right) + 2 \nb \Delta_t R_{\nb}}{\nb \Delta_t \left(- 8 M_k \sqrt{\frac{\delta}{\nb}} + 2 \lambda_t\right)}} - \frac{12 \sqrt{2} M_k^{\frac{3}{2}} \left(1+ \sqrt{2\delta}\right) \left(1+ \sqrt{\frac{\delta}{2}}\right)}{\nb \Delta_t \sqrt{- 8 M_k \sqrt{\frac{\delta}{\nb}} +2  \lambda_t}} \right)^2 \\
   &= &   \left(\sqrt{\frac{x}{2T}} \sqrt{\frac{2 \lambda_t - \left( 8 M_k \sqrt{\frac{\delta}{\nb}} + \frac{48 \sqrt{2} M_k^2 \left(1+ \sqrt{2\delta}\right) \left(1+ \sqrt{\frac{\delta}{2}}\right)}{\nb \Delta_t} + 2 R_{\nb}  \right)}{2 \lambda_t - 8 M_k \sqrt{\frac{\delta}{\nb}}}} - \frac{\frac{12 M_k^{\frac{3}{2}}}{\sqrt{\nb}} \left(1+ \sqrt{2\delta}\right) \left(1+ \sqrt{\frac{\delta}{2}}\right)}{\Delta_t \sqrt{2 \lambda_t - 8 M_k \sqrt{\frac{\delta}{\nb}}}}\right)^2,
\end{eqnarray*}  
where we recall that the term in the square of the exponential is strictly positive under assumption (\ref{lambda_t_condition}), which ensures the validity of Lemma \ref{HP_positivity}. 
According to \eqref{splitting}, it gives
\begin{align}
    &\PPO \left(\widehat{D}^2_T > x \right) \notag \\
    &\leq \underset{t=1}{\overset{T}{\sum}} \ 2 e^{-\left( \underset{t=1,\ldots,T}{\min} \Bigg\{ \sqrt{\frac{x}{2T}} \sqrt{\frac{2\lambda_t - \left( 8 M_k \sqrt{\frac{\delta}{\nb}} + \frac{48 \sqrt{2} M_k^2 \left(1+ \sqrt{2\delta}\right) \left(1+ \sqrt{\frac{\delta}{2}}\right)}{\nb \Delta_t} + 2 R_{\nb}  \right)}{2\lambda_t - 8 M_k \sqrt{\frac{\delta}{\nb}}}} - \frac{\frac{12 M_k^{\frac{3}{2}}}{\sqrt{\nb}} \left(1+ \sqrt{2\delta}\right) \left(1+ \sqrt{\frac{\delta}{2}}\right)}{\underset{t=\{1,\ldots,T\}}{\min} \left\{\Delta_t \sqrt{2\lambda_t - 8 M_k \sqrt{\frac{\delta}{\nb}}} \right\}} \Bigg\} \right)^2} + 7 T e^{-\delta} \notag \\
    & = \underset{t=1}{\overset{T}{\sum}} \ 2 e^{-\left( \sqrt{\frac{x}{2T}} \underset{t=1,\ldots,T}{\min} \Bigg\{ \sqrt{\frac{2\lambda_t - \left( 8 M_k \sqrt{\frac{\delta}{\nb}} + \frac{48 \sqrt{2} M_k^2 \left(1+ \sqrt{2\delta}\right) \left(1+ \sqrt{\frac{\delta}{2}}\right)}{\nb \Delta_t} + 2 R_{\nb} \right)}{2\lambda_t - 8 M_k \sqrt{\frac{\delta}{\nb}}}}\Bigg\} - \frac{\frac{12 M_k^{\frac{3}{2}}}{\sqrt{\nb}} \left(1+ \sqrt{2\delta}\right) \left(1+ \sqrt{\frac{\delta}{2}}\right)}{\underset{t=\{1,\ldots,T\}}{\min} \left\{\Delta_t \sqrt{2\lambda_t - 8 M_k \sqrt{\frac{\delta}{\nb}}} \right\}} \right)^2} + 7 T e^{-\delta}.\notag 
\end{align}
Take $x$ according to (\ref{x_bound}) as in the Statement of the Theorem
\begin{eqnarray*}
x &= &    \qmaj (\nb, \delta,M_k,\lambda_{1:T},\Delta_{1:T})   \\
         &= &     2 T \underset{t=1,\ldots,T}{\max}  \left(\frac{\lambda_t - 4 M_k \sqrt{\frac{\delta}{\nb}}}{\lambda_t - K_{1,t}\left(\nb,M_k,\delta,\Delta_t \right) }  \right) \left(\sqrt{\delta} + \frac{K_2 \left( \nb,M_k,\delta \right)}{\underset{t=\{1,\ldots,T\}}{\min} \left\{\Delta_t \sqrt{\lambda_t - 4 M_k \sqrt{\frac{\delta}{\nb}}} \right\}} \right)^2,
\end{eqnarray*}
  and then
$$\PPO \left(\widehat{D}^2_T > x \right)    \leq  \underset{t=1}{\overset{T}{\sum}} \ 2 e^{-C_t^2} + 7 T e^{-\delta},$$ where
\begin{align}
    C_t &\leq \sqrt{\frac{2 T \underset{t=1,\ldots,T}{\max}  \left( \frac{2\lambda_t - 8 M_k \sqrt{\frac{\delta}{\nb}}}{2\lambda_t - \left( 8 M_k \sqrt{\frac{\delta}{\nb}} + \frac{48 \sqrt{2} M_k^2 \left(1+ \sqrt{2\delta}\right) \left(1+ \sqrt{\frac{\delta}{2}}\right)}{\nb \Delta_t} + 2 R_{\nb} \right)}\right)  \left(\sqrt{\delta} + \frac{\frac{12 M_k^{\frac{3}{2}}}{\sqrt{\nb}} \left(1+ \sqrt{2\delta}\right) \left(1+ \sqrt{\frac{\delta}{2}}\right)}{ \underset{t=\{1,\ldots,T\}}{\min} \left\{\Delta_t \sqrt{2\lambda_t - 8 M_k \sqrt{\frac{\delta}{\nb}}} \right\}} \right)^2}{2T}} \notag \\
    & \hspace{0.5 cm} \times \underset{t=1,\ldots,T}{\min} \Bigg\{ \sqrt{\frac{2\lambda_t - \left( 8 M_k \sqrt{\frac{\delta}{\nb}} + \frac{48 \sqrt{2} M_k^2 \left(1+ \sqrt{2\delta}\right) \left(1+ \sqrt{\frac{\delta}{2}}\right)}{\nb \Delta_t} + 2 R_{\nb} \right)}{2\lambda_t - 8 M_k \sqrt{\frac{\delta}{\nb}}}}\Bigg\} - \frac{\frac{12 M_k^{\frac{3}{2}}}{\sqrt{\nb}} \left(1+ \sqrt{2\delta}\right) \left(1+ \sqrt{\frac{\delta}{2}}\right)}{\underset{t=\{1,\ldots,T\}}{\min} \left\{\Delta_t \sqrt{2\lambda_t - 8 M_k \sqrt{\frac{\delta}{\nb}}} \right\}} \notag \\
    &= \frac{\sqrt{\delta} + \frac{\frac{12 M_k^{\frac{3}{2}}}{\sqrt{\nb}} \left(1+ \sqrt{2\delta}\right) \left(1+ \sqrt{\frac{\delta}{2}}\right)}{ \underset{t=\{1,\ldots,T\}}{\min} \left\{\Delta_t \sqrt{2\lambda_t - 8 M_k \sqrt{\frac{\delta}{\nb}}} \right\}}
    }{\sqrt{\underset{t=1,\ldots,T}{\min}  \left( \frac{2\lambda_t - \left( 8 M_k \sqrt{\frac{\delta}{\nb}} + \frac{48 \sqrt{2} M_k^2 \left(1+ \sqrt{2\delta}\right) \left(1+ \sqrt{\frac{\delta}{2}}\right)}{\nb \Delta_t} + 2 R_{\nb}  \right)}{2\lambda_t - 8 M_k \sqrt{\frac{\delta}{\nb}}}  \right)}}   \notag \\
    & \hspace{0.5 cm} \times \underset{t=1,\ldots,T}{\min} \Bigg\{ \sqrt{\frac{2\lambda_t - \left( 8 M_k \sqrt{\frac{\delta}{\nb}} + \frac{48 \sqrt{2} M_k^2 \left(1+ \sqrt{2\delta}\right) \left(1+ \sqrt{\frac{\delta}{2}}\right)}{\nb \Delta_t} + 2 R_{\nb} \right)}{2\lambda_t - 8 M_k \sqrt{\frac{\delta}{\nb}}}}\Bigg\} - \frac{\frac{12 M_k^{\frac{3}{2}}}{\sqrt{\nb}} \left(1+ \sqrt{2\delta}\right) \left(1+ \sqrt{\frac{\delta}{2}}\right)}{\underset{t=\{1,\ldots,T\}}{\min} \left\{\Delta_t \sqrt{2\lambda_t - 8 M_k \sqrt{\frac{\delta}{\nb}}} \right\}} \notag \\
    &= \sqrt{\delta}\notag .
\end{align}
Note that as we said before, for ease of presentation, we upper bounded each term $\frac{\nb-1}{\nb}$ and $\frac{\nb-1}{\nb^{\frac{3}{2}}}$ by $1$ and $\frac{1}{\sqrt{\nb}}$ respectively. By keeping all the terms in $\nb$, the inequality just above applied as before with  $x = \qmaj$ is an equality. 
Hence, we finally get $$ \PPO \left(\widehat{D}^2_T >   \qmaj \right) \leq 9 T e^{-\delta},$$
which concludes the proof. 
%%%%%%%%%%%%%%
\subsection{Proof of Corollary \ref{bound_simple}}
\begin{proof}[Proof of Corollary \ref{bound_simple}]
Let us assume \A{1}$\sim$\A{3} and Conditions (\ref{A4bis}), (\ref{lambda_t_condition}) for all $t \in \{1,\ldots,T\}$, let also assume that  (\ref{K_3_simplified}) is satisfied for the constant $c$ introduced on the Corollary. Let $\delta>0$.
From Theorem~\ref{the_theorem}, we have 
\begin{equation*}
      \qmaj \leq   \qmajbis,
\end{equation*}
providing simply that
\begin{align*}
    \PPO \left(\widehat{D}^2_T > \qmajbis \right) &\leq \PPO \left(\widehat{D}^2_T > \qmaj  \right) \\
    &\leq 9 T e^{-\delta},
\end{align*}
which concludes the proof. 
\end{proof}
%%%%%%%%%%%%%%
\subsection{Preliminary results}
\label{Preliminary_results}
%%%%%%%
\subsubsection*{Upper bounds over the elements in $(\Hk,\|\cdot\|_{\Hk})$}
\begin{lemma}
\label{kernel_bound}
Assume \A{2}. For all $Z$ following $\PP$, 
\begin{equation*}
     \left \| \phi(Z) \right \| \leq \sqrt{M_k}.
\end{equation*}
\end{lemma}
\begin{proof}
From assumption \A{2},
\begin{align*}
    \left \| \phi(Z) \right \| = \sqrt{\langle \phi(Z),\phi(Z) \rangle_{\Hk}} = \sqrt{k(Z,Z)} \leq \sqrt{M_k}.
 \end{align*}
\end{proof}
\begin{lemma}
\label{bound_kernel}
Assume \A{2}. For all $(i,j) \in \{1,\ldots,\nb\}$ 
    \begin{equation*}
         \left \| \phi(X_i) - \widehat{\mu}_X \right \| \leq  2 \sqrt{M_k}  \quad \text{and} \quad \left \| \phi(Y_j) - \widehat{\mu}_Y \right \| \leq  2 \sqrt{M_k} .
   \end{equation*}
\end{lemma}
\begin{proof}
Let $i$ be in $\{1,\ldots,\nb\}$. From the triangle inequality and Lemma \ref{bound_kernel}, we get 
\begin{equation*}
\left \| \phi(X_i) - \widehat{\mu}_X \right \| =  \left \| \frac{1}{\nb}  \underset{j \neq i}      {\sum} \left( \phi(X_i) - \phi(X_j) \right) \right \| \leq  \frac{1}{\nb} \underset{j \neq i}
     {\sum} \left( \left \| \phi(X_i) \right \| + \left \| \phi(X_j) \right \| \right) \leq  2 \sqrt{M_k}.
\end{equation*}
The same result is obtained for $\left \|  \phi(Y_i) - \widehat{\mu}_Y \right \|_{\Hk}$. 
\end{proof}
%%%%%%%%%
\subsubsection*{Concentration inequalities of $\Hk-$elements projected to the $\Sigma_T$-eigenfunctions.}
\begin{lemma}
\label{covariance_term}
Assume \A{2}. Let $t \in \{1,\ldots,T\}$. For all $\delta >0$ 
\begin{equation*}
    \PP \left(  \frac{1}{\nb} \underset{i=1}{\overset{\nb}{\sum}} \langle f_t, \widehat{\mu}_X - \phi(X_i) \rangle_{\Hk} \langle f_t,\phi(Y_i) - \widehat{\mu}_Y  \rangle_{\Hk}> 8 M_k \sqrt{\frac \delta \nb} \right) \leq e^{- \delta}.
\end{equation*}
\end{lemma}
\begin{proof}
Let $t \in \{1,\ldots,T\}$. We apply Theorem \ref{McDiarmid's theo} in Section \ref{existing_results} (McDiarmid's inequality). For this purpose, we set $g : [\mathcal{Z}^{\otimes 2\nb} \longrightarrow \R]$ defined for all $ \left(x_1,\ldots,x_{\nb},y_1,\ldots,y_{\nb}\right) \in \mathcal{Z}^{\otimes 2\nb}$ by
$$g\left(x_1,\ldots,x_{\nb},y_1,\ldots,y_{\nb}\right) =  \frac{1}{\nb} \underset{i=1}{\overset{\nb}{\sum}} \langle f_t, \widehat{\nu}_1 - \phi(x_i) \rangle_{\Hk} \langle f_t,\phi(y_i) - \widehat{\nu}_2  \rangle_{\Hk},$$
where $\widehat{\nu}_1 = \frac{1}{\nb} \underset{i=1}{\overset{\nb}{\sum}} \phi(x_i)$ and $\widehat{\nu}_2 = \frac{1}{\nb} \underset{i=1}{\overset{\nb}{\sum}} \phi(y_i)$, and we prove that $g$ satisfies the bounded differences property (see Definition \ref{def_bounded_difference_property} in Section \ref{existing_results}). For $i \in \{1,\ldots,\nb\}$, we define $x_i'$ any copy of $x_i$ and $\widehat{\nu}_1' = \frac{1}{\nb}  \underset{j \neq i}      {\sum} \phi(x_j) + \frac{1}{\nb} \phi(x_i')$. Then, by using successively the triangle inequality, the Cauchy-Schwarz inequality and Lemmas~\ref{kernel_bound} and~\ref{bound_kernel}, we get 
\begin{align*}
     & \left | g\left(x_1,\ldots,x_{i-1},x_i',x_{i+1},\ldots,x_{\nb},y_1,\ldots,y_{\nb}\right) - g\left(x_1,\ldots,x_{\nb},y_1,\ldots,y_{\nb}\right) \right | \\
     &  \leq \frac{1}{\nb}  \underset{j \neq i}{\sum}  \left | \langle f_t,\phi(y_j) - \widehat{\nu}_2  \rangle_{\Hk} \right | \left | \langle f_t, \widehat{\nu}_1' - \widehat{\nu}_1 \rangle_{\Hk}\right | \\
     & \hspace{1 cm} + \frac{1}{\nb} \left | \langle f_t,\phi(y_i) - \widehat{\nu}_2  \rangle_{\Hk} \right | \left( \left | \langle f_t, \widehat{\nu}_1' - \widehat{\nu}_1\rangle_{\Hk} \right | + \left | \langle f_t, \phi(x_i) -  \phi(x_i') \rangle_{\Hk} \right | \right) \\
    & \leq \frac{1}{\nb}  \underset{j \neq i}      {\sum}  \left \| f_t \right \|  \left \| \phi(y_j) - \widehat{\nu}_2  \right \|  \left \| f_t \right \| \frac{1}{\nb}  \left \|  \phi(x_i') - \phi(x_i)  \right \| \\
    & \hspace{1 cm} + \frac{1}{\nb} \left \| f_t \right \|  \left \| \phi(y_i) - \widehat{\nu}_2  \right \|  \left( \left \| f_t \right \| \frac{1}{\nb}  \left \|  \phi(x_i') - \phi(x_i)  \right \| +  \left \| f_t \right \|  \left \|  \phi(x_i) - \phi(x_i') \right \| \right)\\ 
    & \leq \frac{8}{\nb} M_k .
 \end{align*}
The same upper bound is similarly obtained on 
$$\left | g\left(x_1,\ldots,x_{\nb},y_1,\ldots,y_{i-1},y_i',y_{i+1},\ldots,y_{\nb}\right) - g\left(x_1,\ldots,x_{\nb},y_1,\ldots,y_{\nb}\right) \right |$$ where $i \in \{\nb + 1,\ldots,2\nb\}$, $y_i'$ is a copy of $y_i$ and $\widehat{\nu}_2' = \frac{1}{\nb}  \underset{j \neq i}{\sum} \phi(y_j) + \frac{1}{\nb} \phi(y_i')$. \\
Hence, $g$ satisfies the bounded differences property with bounds $c_i=  \frac{8}{\nb} M_k$ for all $i \in \{1,\ldots,2 \nb \}$. \\
As 
$$ \EE[g\left(X_1,\ldots,X_{\nb},Y_1,\ldots,Y_{\nb}\right)] = \langle f_t, \EE[\widehat{\mu}_X - \phi(X_i)] \rangle_{\Hk} \langle f_t, \EE[\phi(Y_i)-\widehat{\mu}_Y] \rangle_{\Hk} = 0,$$
from the  independence between $(\widehat{\mu}_X - \phi(X_i))_i$ and $(\phi(Y_i) - \widehat{\mu}_Y)_i$, we get from
the McDiarmid's inequality (Theorem \ref{McDiarmid's theo}) applied on $g$ that $\forall \varepsilon >0$ 
\begin{equation*}
    \PP \left(  \frac{1}{\nb} \underset{i=1}{\overset{\nb}{\sum}} \langle f_t, \widehat{\mu}_X - \phi(X_i) \rangle_{\Hk} \langle f_t,\phi(Y_i) - \widehat{\mu}_Y  \rangle_{\Hk}> \varepsilon  \right) \leq e^{- \frac{\nb \varepsilon^2}{8^2   M_k^2}},
\end{equation*}
\ti{e.g.} for all $\delta >0$ 
\begin{equation*}
    \PP \left(  \frac{1}{\nb} \underset{i=1}{\overset{\nb}{\sum}} \langle f_t, \widehat{\mu}_X - \phi(X_i) \rangle_{\Hk} \langle f_t,\phi(Y_i) - \widehat{\mu}_Y  \rangle_{\Hk}> 8 M_k \sqrt{\frac \delta \nb}   \right) \leq e^{- \delta}.
\end{equation*}
\end{proof}
\begin{lemma}
\label{num_den_term}
Assume \A{2}. Let $t \in \{1,\ldots,T\}$. For all $\delta>0$ 
\begin{equation*}
    \PPO \left(  \frac{1}{\nb} \underset{i=1}{\overset{\nb}{\sum}} \langle f_t, \phi(X_i) - \phi(Y_i) \rangle_{\Hk}^2 - 2 \lambda_t < - 8 M_k \sqrt{\frac{\delta}{\nb}}  \right) \leq e^{- \delta}.
\end{equation*}
\end{lemma}
\begin{proof}
Let $t \in \{1,\ldots,T\}$. We apply Theorem \ref{McDiarmid's theo} in Section \ref{existing_results} (McDiarmid's inequality). For this purpose, we set $h : [\mathcal{Z}^{\otimes 2\nb} \longrightarrow \R]$ defined for all $ \left(x_1,\ldots,x_{\nb},y_1,\ldots,y_{\nb}\right) \in \mathcal{Z}^{\otimes 2\nb}$ by 
$$  h\left(x_1,\ldots,x_{\nb},y_1,\ldots,y_{\nb}\right) =  \frac{1}{\nb} \underset{i=1}{\overset{\nb}{\sum}} \langle f_t, \phi(x_i) - \phi(y_i)  \rangle_{\Hk}^2,$$
and we prove that $h$ satisfies the bounded differences property (see Definition \ref{def_bounded_difference_property} in Section \ref{existing_results}). 
For $i \in \{1,\ldots,\nb\}$, we define $x_i'$ any copy of $x_i$. Then, by using successively the triangle inequality, the Cauchy-Schwarz inequality and Lemmas~\ref{kernel_bound} and~\ref{bound_kernel}, we get 
\begin{align*}
    & \left | h\left(x_1,\ldots,x_{i-1},x_i',x_{i+1},\ldots,x_{\nb},y_1,\ldots,y_{\nb}\right) - h\left(x_1,\ldots,x_{\nb},y_1,\ldots,y_{\nb}\right) \right | \\
    &\leq \frac{1}{\nb} \left | \langle f_t, \phi(x_i') - \phi(y_i)  \rangle_{\Hk} \right |^2 + \left | \langle f_t, \phi(x_i) - \phi(y_i)  \rangle_{\Hk} \right |^2 \\
    &\leq \frac{1}{\nb} \left \| f_t \right \|^2 \left \| \phi(x_i') - \phi(y_i)  \right \|^2 + \left \| f_t \right \|^2 \left \| \phi(x_i) - \phi(y_i)  \right \|^2 \\ 
    & \leq \frac{8}{\nb} M_k. \\
\end{align*}
The same upper bound is similarly obtained on $$\left | h\left(x_1,\ldots,x_{\nb},y_1,\ldots,y_{i-1},y_i',y_{i+1},\ldots,y_{\nb}\right) - h\left(x_1,\ldots,x_{\nb},y_1,\ldots,y_{\nb}\right) \right |$$ where $i \in \{\nb + 1,\ldots,2\nb\}$ and $y_i'$ is a copy of $y_i$. \\
Hence, $h$ satisfies the bounded differences property with bounds $c_i=\frac{8}{\nb} M_k$ for all $i \in \{1,\ldots,2 \nb \}$. \\
Moreover, from (\ref{otimes_equality}), equality $\mu_X = \mu_Y$ under $H_0$ and the independence between $\phi(X)$ and $\phi(Y)$, we get 
\begin{align*}
    \lambda_t
    &= \langle f_t, \Sigma (f_t) \rangle_{\Hk} \\
    & = \langle f_t, \frac 12 \left(\EE_{\PP_X} \left [ \left(\phi(X) - \mu_X \right)^{\otimes^2_{\Hk}} \right] +\EE_{\PP_Y} \left [ \left(\phi(Y) - \mu_Y \right)^{\otimes^2_{\Hk}} \right] \right)(f_t) \rangle_{\Hk} \\
    & = \frac 12 \EE_{\PP_{X,Y}} \left [ \langle f_t,  \langle f_t, \phi(X) - \mu_X  \rangle_{\Hk} \left( \phi(X) - \mu_X  \right) \rangle_{\Hk} + \langle f_t,  \langle f_t, \phi(Y) - \mu_Y  \rangle_{\Hk} \left( \phi(Y) - \mu_Y \right) \rangle_{\Hk} \right ]  \\
    & = \frac 12 \EE_{\PP_{X,Y}} \left [ \langle f_t, \phi(X) - \mu_X  \rangle_{\Hk}^2 + \langle f_t, \phi(Y) - \mu_Y  \rangle_{\Hk}^2 \right ]  \\
    &= \frac 12 \EE_{\PP_{X,Y}} \left [ \langle f_t, \phi(X) - \mu_X - \left( \phi(Y) - \mu_Y \right) \rangle_{\Hk}^2 +   \langle f_t, \phi(X) - \mu_X \rangle_{\Hk} \langle f_t, \phi(Y) - \mu_Y \rangle_{\Hk}  \right ] \\
    & = \frac 12 \EE_{\PP_{X,Y}} \left [ \langle f_t, \phi(X) - \phi(Y) \rangle_{\Hk}^2  \right ] +   \langle f_t,\EE_{\PP_{X,Y}} \left [ \phi(X) - \mu_X  \right ] \rangle_{\Hk} \langle f_t, \EE_{\PP_{X,Y}} \left [ \phi(Y) - \mu_Y \right ]  \rangle_{\Hk} \\
    & = \frac 12 \EE_{\PP_{X,Y}} \left [ \langle f_t, \phi(X) - \phi(Y) \rangle_{\Hk}^2  \right ],
\end{align*}
where $p = \frac{1}{2}$ in formula (\ref{within_covariance_operator}) under assumption \A{1}. 
From the McDiarmid's inequality (Theorem \ref{McDiarmid's theo}) applied on $h$, we get $\forall \varepsilon >0$ 
\begin{equation*}
    \PPO \left(  \frac{1}{\nb} \underset{i=1}{\overset{\nb}{\sum}} \langle f_t, \phi(X_i) - \phi(Y_i)  \rangle_{\Hk}^2 - 2 \lambda_t < -\varepsilon  \right) \leq e^{- \frac{\nb \varepsilon^2}{64 M_k^2}},
\end{equation*}
\ti{e.g.} for all $\delta >0$ 
\begin{equation*}
    \PPO \left(  \frac{1}{\nb} \underset{i=1}{\overset{\nb}{\sum}} \langle f_t, \phi(X_i) - \phi(Y_i) \rangle_{\Hk}^2 - 2 \lambda_t < - 8 M_k \sqrt{\frac{\delta}{\nb}}  \right) \leq e^{- \delta}.
\end{equation*}
\end{proof}
\begin{lemma}
\label{lemma_norm_Z_bar}
Assume \A{2}. 
For all $\delta>0$ 
\begin{equation*}
    \PPO \left(  \left\| \frac{1}{\nb} \underset{i=1}{\overset{\nb}{\sum}} \left(\phi(X_i) - \phi(Y_i)\right) \right\|_{\Hk} > \sqrt{\frac{2 M_k}{\nb}} \left(1+\sqrt{2\delta}\right) \right) \leq e^{-\delta}.
\end{equation*}
\end{lemma}
\begin{proof}
We apply Theorem \ref{McDiarmid's theo} in Section \ref{existing_results} (McDiarmid's inequality). For this purpose, we set $v : [\mathcal{Z}^{\otimes 2\nb} \longrightarrow \R]$ defined for all $ \left(x_1,\ldots,x_{\nb},y_1,\ldots,y_{\nb}\right) \in \mathcal{Z}^{\otimes 2\nb}$ by 
$$   v\left(x_1,\ldots,x_{\nb},y_1,\ldots,y_{\nb}\right) =  \left\| \frac{1}{\nb} \underset{i=1}{\overset{\nb}{\sum}} \left(\phi(x_i) - \phi(y_i)\right) \right\|_{\Hk},$$
and we prove that $v$ satisfies the bounded differences property (see Definition \ref{def_bounded_difference_property} in Section \ref{existing_results}). 
For $i \in \{1,\ldots,\nb\}$, we define $x_i'$ any copy of $x_i$. Then, by using successively the triangle inequality and Lemma~\ref{kernel_bound}, we get 
\begin{align*}
    | v(x_1,&\ldots,x_{i-1},x_i',x_{i+1},\ldots,x_{\nb},y_1,\ldots,y_{\nb}) - v\left(x_1,\ldots,x_{\nb},y_1,\ldots,y_{\nb}\right) | \\
     & \leq \frac{1}{\nb} \left\|  \underset{j \neq i}{\underset{j=1}{\overset{\nb}{\sum}}} \left(\phi(x_j) - \phi(y_j)\right) + \left(\phi(x_i') - \phi(y_i)\right) - \left( \underset{j \neq i}{\underset{j=1}{\overset{\nb}{\sum}}} \left(\phi(x_j) - \phi(y_j)\right) + \left(\phi(x_i) - \phi(y_j)\right) \right) \right\|_{\Hk} \\
     & \leq \frac{1}{\nb} \left( \left\| \phi(x_i')\right\|_{\Hk} + \left\| \phi(x_i) \right\|_{\Hk} \right) \leq \frac{2}{n} \sqrt{M_k}.
\end{align*}
The same upper bound is similarly obtained on $$\left | v\left(x_1,\ldots,x_{\nb},y_1,\ldots,y_{i-1},y_i',y_{i+1},\ldots,y_{\nb}\right) - v\left(x_1,\ldots,x_{\nb},y_1,\ldots,y_{\nb}\right) \right |$$  where $i \in \{\nb + 1,\ldots,2\nb\}$ and $y_i'$ is a copy of $y_i$. \\ 
Hence, $v$ satisfies the bounded differences property with bounds $c_i= \frac{2}{n} \sqrt{M_k}$ for all $i \in \{1,\ldots,2 \nb \}$. \\
Moreover, from Jensen's inequality, the triangle inequality, Lemma~\ref{kernel_bound} and the independence between variables $\left(\phi(X_i) - \phi(Y_i)\right)_{\{1 \leq i \leq n\}}$, we get
\begin{align*}
    \EE[v\left(X_1,\ldots,X_{\nb},Y_1,\ldots,Y_{\nb}\right)] &\leq \sqrt{\EE\left[\left\| \frac{1}{\nb} \underset{i=1}{\overset{\nb}{\sum}} \left(\phi(X_i) - \phi(Y_i)\right) \right\|_{\Hk}^2\right]} \\
     &\leq \sqrt{\frac{1}{\nb^2} \underset{i=1}{\overset{\nb}{\sum}} \EE\left[ \left(\left\|\phi(X_i)\right\|_{\Hk}^2 + \left\|\phi(Y_i)\right\|_{\Hk}^2\right)\right]} \\
    &\leq \sqrt{\frac{2 M_k}{\nb}}.
    \label{temp_proof}
\end{align*}
From the McDiarmid's inequality (Theorem \ref{McDiarmid's theo}) on $v$, we get $\forall \varepsilon>0$ 
\begin{align*}
    \PPO &\left(  \left\| \frac{1}{\nb} \underset{i=1}{\overset{\nb}{\sum}} \left(\phi(X_i) - \phi(Y_i)\right) \right\|_{\Hk} - \sqrt{\frac{2 M_k}{\nb}} > \varepsilon  \right) \\
    &\leq \PPO \left(  \left\| \frac{1}{\nb} \underset{i=1}{\overset{\nb}{\sum}} \left(\phi(X_i) - \phi(Y_i)\right) \right\|_{\Hk} - \EE[v\left(X_1,\ldots,X_{\nb},Y_1,\ldots,Y_{\nb}\right)]  > \varepsilon  \right) \leq e^{-\frac{\varepsilon^2 \nb}{4 M_k}},
\end{align*}
\ti{e.g.} for all $\delta >0$ 
\begin{equation*}
    \PPO \left(  \left\| \frac{1}{\nb} \underset{i=1}{\overset{\nb}{\sum}} \left(\phi(X_i) - \phi(Y_i)\right) \right\|_{\Hk} > \sqrt{\frac{2 M_k}{\nb}} \left(1+\sqrt{2\delta}\right) \right) \leq e^{-\delta}.
\end{equation*}
\end{proof}
%%%%%%%%%%%% 
\subsubsection*{Some controls on the eigenvalue $\widehat{\lambda}_t$.}
\begin{lemma}
\label{lambda_t_chap_decompo}
For all $t \in \{1,\ldots,T\}$,
\begin{equation*}
   \widehat{\lambda}_t = \frac{1}{2\nb} \underset{i=1}{\overset{\nb}{\sum}} \langle \widehat{f}_t, \phi(X_i) - \phi(Y_i)\rangle_{\Hk}^2 -\frac{1}{2} \langle \widehat{f}_t, \widehat{\mu}_X - \widehat{\mu}_Y \rangle_{\Hk}^2 + \frac{1}{\nb} \underset{i=1}{\overset{\nb}{\sum}} \langle \widehat{f}_t, \phi(X_i) - \widehat{\mu}_X  \rangle_{\Hk} \langle \widehat{f}_t,\phi(Y_i) - \widehat{\mu}_Y  \rangle_{\Hk}.
\end{equation*}
\end{lemma}
\begin{proof}
Let $t \in \{1,\ldots,T\}$. From (\ref{otimes_equality}) and (\ref{emprical_SW}), we get 
\begin{align*}
    \widehat{\lambda}_t 
    &= \langle \widehat{f}_t, \widehat{\Sigma} (\widehat{f}_t) \rangle_{\Hk}  \\
    & = \langle \widehat{f}_t, \frac{1}{2\nb} \underset{i=1}{\overset{\nb}{\sum}} \left( \left( \phi(X_i) - \widehat{\mu}_X \right)^{\otimes_{\Hk}^2} + \left( \phi(Y_i) - \widehat{\mu}_Y \right)^{\otimes_{\Hk}^2} \right) \left(\widehat{f}_t \right) \rangle_{\Hk}  \\
    & = \frac{1}{2\nb} \underset{i=1}{\overset{\nb}{\sum}} \langle \widehat{f}_t, \langle \widehat{f}_t, \phi(X_i) - \widehat{\mu}_X \rangle_{\Hk} \left( \phi(X_i) - \widehat{\mu}_X \right) +  \langle \widehat{f}_t, \phi(Y_i) - \widehat{\mu}_Y \rangle_{\Hk} \left( \phi(Y_i) - \widehat{\mu}_Y \right) \rangle_{\Hk}  \\
    & = \frac{1}{2\nb} \underset{i=1}{\overset{\nb}{\sum}} \left(\langle \widehat{f}_t, \phi(X_i) - \widehat{\mu}_X \rangle_{\Hk}^2 +  \langle \widehat{f}_t, \phi(Y_i) - \widehat{\mu}_Y \rangle_{\Hk}^2 \right)  \\
    & = \frac{1}{2\nb} \underset{i=1}{\overset{\nb}{\sum}} \left( \langle \widehat{f}_t, \left(\phi(X_i) - \widehat{\mu}_X \right) - \left(\phi(Y_i) - \widehat{\mu}_Y \right) \rangle_{\Hk}^2 +  2 \langle \widehat{f}_t, \phi(X_i) - \widehat{\mu}_X  \rangle_{\Hk} \langle \widehat{f}_t,\phi(Y_i) - \widehat{\mu}_Y  \rangle_{\Hk} \right)  \\
    &= \frac{1}{2\nb} \underset{i=1}{\overset{\nb}{\sum}} \langle \widehat{f}_t, \phi(X_i) - \phi(Y_i)\rangle_{\Hk}^2 -\frac{1}{2} \langle \widehat{f}_t, \widehat{\mu}_X - \widehat{\mu}_Y \rangle_{\Hk}^2 + \frac{1}{\nb} \underset{i=1}{\overset{\nb}{\sum}} \langle \widehat{f}_t, \phi(X_i) - \widehat{\mu}_X  \rangle_{\Hk} \langle \widehat{f}_t,\phi(Y_i) - \widehat{\mu}_Y  \rangle_{\Hk}.
\end{align*}
\end{proof}
\begin{lemma}
\label{negligeble_terms}
Assume \A{2}. For all $t \in \{1,\ldots,T\}$,
\begin{align*} 
    \widehat{\lambda}_t &\geq \frac{1}{2\nb} \underset{i=1}{\overset{\nb}{\sum}} \langle \widehat{f}_t, \phi(X_i) - \phi(Y_i)\rangle_{\Hk}^2 - \frac{1}{2} \left \|  \widehat{\mu}_X - \widehat{\mu}_Y \right \|_{\Hk}^2 - 8   M_k \left \|  f_t - \widehat{f}_t \right \|_{\Hk} \\
    & \hspace{1 cm} - \frac{1}{\nb} \underset{i=1}{\overset{\nb}{\sum}} \langle f_t, \widehat{\mu}_X - \phi(X_i) \rangle_{\Hk} \langle f_t,\phi(Y_i) - \widehat{\mu}_Y  \rangle_{\Hk}. 
\end{align*}
Moreover, the same inequality is satisfied with $f_t$ replaced by $s f_t$ with $s = \sign( \langle f_t, \widehat{f}_t \rangle)$.
\end{lemma}
\begin{proof}
Let $t \in \{1,\ldots,T\}$. \\
From Cauchy-Schwarz inequality and Lemma \ref{bound_kernel}, we get successively for all $i \in \{1,\ldots,\nb\}$ 
\begin{align*}
    \langle \widehat{f}_t, \widehat{\mu}_X - \widehat{\mu}_Y \rangle_{\Hk}^2 &\leq \left \|  \widehat{\mu}_X - \widehat{\mu}_Y \right \|_{\Hk}^2, \\
    \langle \widehat{f}_t, \phi(X_i) - \widehat{\mu}_X  \rangle_{\Hk} \langle \widehat{f}_t,\phi(Y_i) - \widehat{\mu}_Y  \rangle_{\Hk} &= - \langle f_t - \widehat{f}_t, \phi(X_i) - \widehat{\mu}_X  \rangle_{\Hk} \langle \widehat{f}_t,\phi(Y_i) - \widehat{\mu}_Y  \rangle_{\Hk} \\
    & \hspace{0.5 cm} - \langle f_t, \phi(X_i) - \widehat{\mu}_X  \rangle_{\Hk} \langle f_t - \widehat{f}_t,\phi(Y_i) - \widehat{\mu}_Y  \rangle_{\Hk} \\
    & \hspace{0.5 cm} - \langle f_t, \widehat{\mu}_X - \phi(X_i) \rangle_{\Hk} \langle f_t,\phi(Y_i) - \widehat{\mu}_Y  \rangle_{\Hk}, \\
    \langle f_t - \widehat{f}_t, \phi(X_i) - \widehat{\mu}_X  \rangle_{\Hk} \langle \widehat{f}_t,\phi(Y_i) - \widehat{\mu}_Y  \rangle_{\Hk} &\leq  \left | \langle f_t - \widehat{f}_t, \phi(X_i) - \widehat{\mu}_X  \rangle_{\Hk}  \right | \left | \langle \widehat{f}_t,\phi(Y_i) - \widehat{\mu}_Y  \rangle_{\Hk} \right | \\
    & \leq \left \|  f_t - \widehat{f}_t \right \|_{\Hk} \left \|   \phi(X_i) - \widehat{\mu}_X \right \|_{\Hk} \left \|  \widehat{f}_t \right \|_{\Hk} \left \|  \phi(Y_i) - \widehat{\mu}_Y \right \|_{\Hk} \\
    & \leq 4  M_k \left \|  f_t - \widehat{f}_t \right \|_{\Hk}, \\
    \text{and} \ \ \langle f_t, \phi(X_i) - \widehat{\mu}_X  \rangle_{\Hk} \langle f_t - \widehat{f}_t,\phi(Y_i) - \widehat{\mu}_Y  \rangle_{\Hk} &\leq  \left | \langle f_t, \phi(X_i) - \widehat{\mu}_X  \rangle_{\Hk} \right | \left | \langle f_t - \widehat{f}_t,\phi(Y_i) - \widehat{\mu}_Y \rangle_{\Hk} \right | \\
    & \leq \left \|  f_t \right \|_{\Hk} \left \|   \phi(X_i) - \widehat{\mu}_X \right \|_{\Hk} \left \|  f_t - \widehat{f}_t \right \|_{\Hk} \left \|  \phi(Y_i) - \widehat{\mu}_Y \right \|_{\Hk}  \\
    & \leq 4  M_k \left \|  f_t - \widehat{f}_t \right \|_{\Hk}.
\end{align*}
We deduce 
\begin{align*}
    \frac{1}{\nb} \underset{i=1}{\overset{\nb}{\sum}} \langle \widehat{f}_t, \phi(X_i) - \widehat{\mu}_X  \rangle_{\Hk} \langle \widehat{f}_t,\phi(Y_i) - \widehat{\mu}_Y  \rangle_{\Hk} &\geq - 8   M_k \left \|  f_t - \widehat{f}_t \right \|_{\Hk} \\
    & \hspace{1 cm} - \frac{1}{\nb} \underset{i=1}{\overset{\nb}{\sum}} \langle f_t, \widehat{\mu}_X - \phi(X_i) \rangle_{\Hk} \langle f_t,\phi(Y_i) - \widehat{\mu}_Y  \rangle_{\Hk},
\end{align*}
and from Lemma ~\ref{lambda_t_chap_decompo} 
\begin{align*} 
    \widehat{\lambda}_t &\geq \frac{1}{2\nb} \underset{i=1}{\overset{\nb}{\sum}} \langle \widehat{f}_t, \phi(X_i) - \phi(Y_i)\rangle_{\Hk}^2 - \frac{1}{2} \left \|  \widehat{\mu}_X - \widehat{\mu}_Y \right \|_{\Hk}^2 - 8  M_k \left \|  f_t - \widehat{f}_t \right \|_{\Hk} \\
    & \hspace{1 cm} - \frac{1}{\nb} \underset{i=1}{\overset{\nb}{\sum}} \langle f_t, \widehat{\mu}_X - \phi(X_i) \rangle_{\Hk} \langle f_t,\phi(Y_i) - \widehat{\mu}_Y  \rangle_{\Hk}. 
\end{align*}
The same inequality is clearly satisfied with $f_t$ replaced by $s f_t$ with $s = \sign(\langle f_t, \hat f_t \rangle )$. 
\end{proof}
%%%%%%%%%%
\subsubsection*{Exponential bounds on centered covariance and homogeneous within-group covariance operators.}
The following result about the centered version of the empirical covariance operator is briefly mentioned in \citep{zwald2005convergence} in a comment and used in the work of \citep{blanchard2007statistical} (Section 3.5.) to control reconstruction error. 
\begin{lemma}
	\bf{Centered version of \bf{Lemma 1. of \cite{zwald2005convergence}}\rm} 
	\normalfont
	(see Theorem \ref{Blanchard_non_centered_C})]
Let $n$ be an integer and $Z_1,\ldots,Z_n$ be $n$ independent random variables taking theirs values in a measurable space $\mathcal{Z}$ and following the distribution $\PP$. Let $\phi$ be the feature mapping function from $\mathcal{Z}$ to a reproducing kernel Hilbert space $\Hk$ associated to the kernel function $k$ such that $\phi(z) = k(z,\cdot)$ for all $z \in \mathcal{Z}$. 
Let us define the centered covariance operator $\overline{C} \in \HSk$ and the centered empirical covariance operator $\overline{C_n} \in \HSk$ respectively by 
\begin{align*}
    \overline{C} &= \EE[\left(\phi(Z)-\mu_Z\right)\otimes_{\Hk} \left(\phi(Z)-\mu_Z\right)] \\
    \overline{C_n} &= \frac{1}{n} \underset{i=1}{\overset{n}{\sum}} \left(\phi(Z_i)-\widehat{\mu}_Z\right)\otimes_{\Hk} \left(\phi(Z_i)-\widehat{\mu}_Z\right),
\end{align*}
where   
\begin{equation*}
    \mu_Z = \EE[\phi(Z)] \quad \text{and} \quad \widehat{\mu}_Z = \frac{1}{n} \underset{i=1}{\overset{n}{\sum}} \phi(Z_i). 
\end{equation*}
Assume \A{2} and the simplicity of the eigenvalues of $\overline{C}$. Then, for all $\delta>0$ 
\begin{equation*}
     \PPO \left( \left \| \overline{C_n} - \overline{C} \right \|_{\HSk} \leq \frac{6 M_k}{\sqrt{n}} \left(1+ \sqrt{\frac{\delta}{2}}\right) \right) \geq 1 - e^{-\delta}.
\end{equation*}
\label{Blanchard_centered_C}
\end{lemma}
\begin{proof}
    We have 
    \begin{align*}
        \overline{C} &= \EE\left[\phi(Z) \otimes_{\Hk} \phi(Z) -  \EE[\phi(Z)] \otimes_{\Hk} \phi(Z) -  \phi(Z) \otimes_{\Hk} \EE[\phi(Z)] + \EE[\phi(Z)] \otimes_{\Hk} \EE[\phi(Z)] \right] \\
        &= \EE[\phi(Z) \otimes_{\Hk} \phi(Z)] - \EE[\phi(Z)] \otimes_{\Hk}  \EE[\phi(Z)] - \EE[\phi(Z)] \otimes_{\Hk}  \EE[\phi(Z)] + \EE[\phi(Z)] \otimes_{\Hk} \EE[\phi(Z)] \\
        &= C - \mu_Z \otimes_{\Hk} \mu_Z, \\
        \text{and} \ \ \overline{C_n} &= C_n - \widehat{\mu}_Z  \otimes_{\Hk} \widehat{\mu}_Z.
    \end{align*}
So, by using the triangle inequality and Theorem \ref{Blanchard_non_centered_C} applied on $C$, we get 
\begin{align}
    \PPO &\left( \left \| \overline{C_n} - \overline{C} \right \|_{\HSk} > \frac{6 M_k}{\sqrt{n}} \left(1+ \sqrt{\frac{\delta}{2}}\right) \right) \notag \\
    & \leq \PPO \left( \left \| C_n - C \right \|_{\HSk} +  \left \| \mu_Z \otimes_{\Hk} \mu_Z - \widehat{\mu}_Z  \otimes_{\Hk} \widehat{\mu}_Z \right \|_{\HSk} > \frac{6 M_k}{\sqrt{n}} \left(1+ \sqrt{\frac{\delta}{2}}\right) \right)  \notag \\
    &  \leq  \PPO \left( \left \| C_n - C \right \|_{\HSk} > \frac{2 M_k}{\sqrt{n}} \left(1+ \sqrt{\frac{\delta}{2}}\right) \right) \notag  \\
    & \hspace{1 cm} + \PPO \left( \left \| \mu_Z \otimes_{\Hk} \mu_Z - \widehat{\mu}_Z  \otimes_{\Hk} \widehat{\mu}_Z \right \|_{\HSk} > \frac{4 M_k}{\sqrt{n}} \left(1+ \sqrt{\frac{\delta}{2}}\right) \right)  \notag \\
    &  \leq  e^{-\delta} + \PPO \left( \left \| \mu_Z \otimes_{\Hk} \mu_Z - \widehat{\mu}_Z  \otimes_{\Hk} \widehat{\mu}_Z \right \|_{\HSk} > \frac{4 M_k}{\sqrt{n}} \left(1+ \sqrt{\frac{\delta}{2}}\right) \right).
    \label{first_comput}
\end{align}
Moreover, by applying successively the triangle inequality, formula (\ref{norm_inequality}), Jensen's inequality and Lemma \ref{kernel_bound}, we get 
\begin{align}
    \left \| \mu_Z \otimes_{\Hk} \mu_Z - \widehat{\mu}_Z  \otimes_{\Hk} \widehat{\mu}_Z \right \|_{\HSk} & \leq   \left \| \left(\mu_Z -  \widehat{\mu}_Z  \right) \otimes_{\Hk} \mu_Z \right \|_{\HSk} + \left \| \widehat{\mu}_Z \otimes_{\Hk} \left(\mu_Z -  \widehat{\mu}_Z  \right) \right \|_{\HSk} \notag \\
    & \leq  \left \| \mu_Z -  \widehat{\mu}_Z \right \|_{\Hk}  \left \| \mu_Z \right \|_{\Hk} +  \left \| \widehat{\mu}_Z \right \|_{\Hk}  \left \| \mu_Z -  \widehat{\mu}_Z \right \|_{\Hk} \notag  \\
    & \leq   \left \| \mu_Z -  \widehat{\mu}_Z \right \|_{\Hk} \left(  \EE[\left \|  \phi(Z)\right \|_{\Hk} ] + \frac{1}{n} \underset{i=1}{\overset{n}{\sum}} \left \|  \phi(Z_i)\right \|_{\Hk} \right) \notag \\
    & \leq   2 \sqrt{M_k} \left \| \mu_Z -  \widehat{\mu}_Z \right \|_{\Hk}. 
    \label{2_sqrt_Mk}
\end{align}
We apply Theorem \ref{McDiarmid's theo} in Section \ref{existing_results} (McDiarmid's inequality). For this purpose, we set $\ell : [\mathcal{Z}^{\otimes \nb} \longrightarrow \R]$ defined for all $\left(z_1,\ldots,z_{n}\right) \in \mathcal{Z}^{\otimes n}$ by $$    \ell\left(z_1,\ldots,z_{n}\right) = \left \| \widehat{\mu}_z - \mu_Z \right \|_{\Hk} = \left \| \frac{1}{n} \underset{i=1}{\overset{n}{\sum}}
     \phi(z_i) - \EE[\phi(Z)] \right \|_{\Hk},$$
and we prove that $\ell$ satisfies the bounded differences property (see Definition \ref{def_bounded_difference_property} in Section \ref{existing_results}). 
For $i \in \{1,\ldots,\nb\}$ we define $z_i'$ any copy of $z_i$. Then, by using the triangle inequality and Lemma \ref{kernel_bound}, we get 
\begin{align*}
    & \left | \ell\left(z_1,\ldots,z_{i-1},z_i',z_{i+1},\ldots,z_{n}\right) - \ell\left(z_1,\ldots,z_{n}\right) \right | \\
     &\leq \left |  \left \| \frac{1}{n} \underset{i=1}{\overset{n}{\sum}}
     \phi(z_i) - \EE[\phi(Z)] \right \|_{\Hk} +   \frac{1}{n} \left \| \left( \phi(z_{i'}) - \phi(z_i) \right)\right \|_{\Hk} - \left \| \frac{1}{n} \underset{i=1}{\overset{n}{\sum}}
     \phi(z_i) - \EE[\phi(Z)] \right \|_{\Hk} \right | \\
     & \leq \frac{1}{n} \left | \left \| \phi(z_{i'}) \right \|_{\Hk} + \left \| \phi(z_i) \right \|_{\Hk} \right | \\
     & \leq \frac{2}{n} \sqrt{M_k}.
\end{align*}
Hence, $\ell$ satisfies the bounded differences property with bounds $c_i=\frac{2}{n} \sqrt{M_k}$ for all $i \in \{1,\ldots,n\}$. \\
Moreover, from Jensen's inequality, the independence and identically distributed $Z_i$'s, the triangle inequality and  Lemma \ref{kernel_bound}, we get 
\begin{align*}
    \EE\left[ \left \| \mu_Z -  \widehat{\mu}_Z \right \|_{\Hk} \right] & \leq \sqrt{ \EE\left[ \left \| \mu_Z -  \widehat{\mu}_Z \right \|_{\Hk}^2 \right] } \notag \\
    & = \sqrt{\frac{1}{n} \EE\left[ \left \| \EE[\phi(Z)] - \phi(Z) \right\|_{\Hk}^2\right]} \notag \\
    & \leq 2 \sqrt{\frac{M_k}{n}}.
\end{align*}
From the McDiarmid’s inequality (Theorem \ref{McDiarmid's theo}) applied on $\ell$, we get $\forall \varepsilon >0$ 
\begin{align*}
    \PP \left( \left \|  \mu_Z - \widehat{\mu}_Z \right \|_{\Hk} -  2\sqrt{\frac{M_k}{n}} > \varepsilon  \right) &\leq \PP \left( \left \|  \mu_Z - \widehat{\mu}_Z  \right \|_{\Hk} - \EE\left[ \left \| \mu_Z -  \widehat{\mu}_Z \right \|_{\Hk} \right] > \varepsilon  \right) \\
     & \leq e^{-\frac{2 \varepsilon^2}{n \left(\frac{2 \sqrt{M_k}}{n}\right)^2}}.
\end{align*}
\ti{e.g.} for all $\delta >0$ 
\begin{equation}
      \PP \left( \left \| \mu_Z - \widehat{\mu}_Z \right \|_{\Hk} > 2\sqrt{\frac{M_k}{n}} \left( 1 + \sqrt{\frac{\delta}{2}}\right)  \right) \leq e^{-\delta}. 
      \label{almost_the_end}
\end{equation}
Finally, from (\ref{first_comput}), (\ref{2_sqrt_Mk}) and (\ref{almost_the_end}), we get for all $\delta>0$
\begin{equation*}
    \PPO \left( \left \| \overline{C_n} - \overline{C} \right \|_{\HSk} > \frac{6 M_k}{\sqrt{n}} \left(1+ \sqrt{\frac{\delta}{2}}\right) \right) \leq 2 e^{-\delta}.
\end{equation*}
\end{proof}
\begin{corollary}[Adaptation of Lemma \ref{Blanchard_centered_C} to $\Sigma_W$]
Assume \A{2} and \A{3}. For all $\delta>0$ 
\begin{equation*}
     \PPO \left( \left \| \hS - \SW \right \|_{\HSk} \leq 6 M_k \left( 1 + \sqrt{\frac{\delta}{2}}\right) \frac{\sqrt{n_X} + \sqrt{n_Y}}{n_X + n_Y} \right) \geq 1 - 2e^{-\delta}.
\end{equation*}
\label{Blanchard_for_SigmaW}
\end{corollary}
\begin{proof}
    From the homoscedastic assumption and the triangle inequality, we get 
    \begin{align*}
        \PPO  &\left( \left \| \widehat{\Sigma}_W - \Sigma_W \right \|_{\HSk} > 6 M_k \left( 1 + \sqrt{\frac{\delta}{2}}\right) \frac{\sqrt{n_X} + \sqrt{n_Y}}{n_X + n_Y} \right) \\
        &\leq  \PPO \left( \frac{n_X}{n_X + n_Y} \left \|  \left( \widehat{\Sigma}_{X} - \Sigma_{X} \right) \right \|_{\HSk} > \frac{6 M_k \sqrt{n_X}}{n_X + n_Y} \left( 1 + \sqrt{\frac{\delta}{2}}\right) \right) + \\
        & \hspace{1 cm} \PPO \left( \frac{n_Y}{n_X + n_Y} \left \| \left( \widehat{\Sigma}_{Y} - \Sigma_{Y} \right) \right \|_{\HSk} > \frac{6 M_k \sqrt{n_Y}}{n_X + n_Y} \left( 1 + \sqrt{\frac{\delta}{2}}\right) \right) \\
        &  \leq 2 e^{-\delta},
    \end{align*}
    where the last inequality provides form Lemma \ref{Blanchard_centered_C} applied on the centered covariance operators $\Sigma_{X}$ and $\Sigma_{Y}$ of $X$ and $Y$ with a $n_X-$ and a $n_Y-$sample respectively. 
\end{proof}
\begin{corollary}
	\bf{Adaptation of Theorem 2. of \cite{zwald2005convergence} and Lemma 5.2. of \cite{koltchinskii2000random} to $\SW$ (see Theorem \ref{Zwald_Blanchard_Theo})\rm}
	\normalfont
Assume \A{2} and \A{3}. Let $t \in \{1,\ldots,T\}$. \\
Then, on the event $\left \{\left \| \hS - \SW \right \|_{\HSk} < \frac{\Delta_t}{2} \right\}$ 
$$ \left \| \Pi_{f_t} - \Pi_{\widehat{f}_t} \right \|_{\HSk} \leq \frac{2 \left \| \hS - \SW \right \|_{\HSk}}{\Delta_t} ,$$ 
\label{P_V_theo}
and the same inequality is valid with $f_t$ replaced by $s f_t$ where $s = \sign( \langle f_t, \widehat{f}_t \rangle)$.
\end{corollary}
\begin{proof}
    Let $t \in \{1,\ldots,T\}$. From assumption \A{3}, $\SW$ is a symmetric positive Hilbert-Schmidt operator of $\Hk$ with simple positive eigenvalues. 
    So, from the Hoffmann-Wielandt inequality \citep{bhatia1994hoffman} (see Proposition \ref{Hoffmann_inequality} in Section \ref{existing_results}), we get on the event $\left \{\left \| \hS - \SW \right \|_{\HSk} < \frac{\Delta_t}{2} \right\}$ 
    $$ | \widehat{\lambda}_t - \lambda_t | \leq \| \widehat{\Sigma}_W - \Sigma_W \|_{\Hk} < \frac{\Delta_t}{2},$$ and so 
    \begin{align*}
        \widehat{\lambda}_t & \geq \lambda_t - |\widehat{\lambda}_t - \lambda_t| > 2 \Delta_t - \Delta_t = \Delta_t > 0 \\
        \text{and} \quad \widehat{\lambda}_t - \widehat{\lambda}_{t+1} & \geq - | \widehat{\lambda}_t - \lambda_t | + 2 \Delta_t - | \lambda_{t+1} - \widehat{\lambda}_{t+1} | > - \Delta_t + 2 \Delta_t- \Delta_t = 0.
    \end{align*}
    Hence, $\hS$ is also a symmetric positive Hilbert-Schmidt operator of $\Hk$ with simple positive eigenvalues. 
    The final result is a direct application of Theorem 2 in \citep{zwald2005convergence} with $A = \SW$ and $A + B=\hS$.
    We easily check that the inequality is satisfied with  $f_t$ replaced by $s f_t$ where $s = \sign( \langle f_t, \widehat{f}_t \rangle)$, as $s f_t$ is also an eigenfunction of $\SW$.
\end{proof}
%%%%%%%%%%%%%%%%
\subsubsection*{About expression of $x$}
\begin{lemma}
\label{HP_positivity}
Assume (\ref{lambda_t_condition}) and \A{3}. For all $\delta >0$ and for all $t \in \{1,\ldots,T\}$, we have:
$$  2 \lambda_t - 8 M_k \sqrt{\frac{\delta}{\nb}}> 0,\quad \Delta_t \sqrt{2 \lambda_t - 8 M_k \sqrt{\frac{\delta}{\nb}}} > 0,\quad x > 0$$
and
$$\frac{2\lambda_t - K_{1,t}\left(\nb,M_k,\delta,\Delta_t \right) }{2\lambda_t - 8 M_k \sqrt{\frac{\delta}{\nb}}}>0.$$
%\begin{align*}
%    2 \lambda_t - 8 M_k \sqrt{\frac{\delta}{\nb}} &> 0, \\
%        \Delta_t \sqrt{2 \lambda_t - 8 M_k \sqrt{\frac{\delta}{\nb}}} &> 0, \\
%        \frac{2\lambda_t - K_{1,t}\left(\nb,M_k,\delta,\Delta_t \right) }{2\lambda_t - 8 M_k \sqrt{\frac{\delta}{\nb}}} &> 0 \\
%        \text{and} \quad x &> 0.
%\end{align*}
\end{lemma}
\begin{proof}
Assume (\ref{lambda_t_condition}) and \A{3}. Let $\delta>0$ and $t \in \{1,\ldots,T\}$. It suffices to note that 
    \begin{align*}
        K_{1,t}\left(\nb,M_k,\delta,\Delta_t \right)  &= 8 M_k \sqrt{\frac{\delta}{\nb}} + 16 M_k \frac{\nb-1}{\nb} \sqrt{\frac{\delta}{\nb}} \\
        &+ \left(\frac{48 \sqrt{2} M_k^2 \left(1 +\sqrt{2\delta} \right)}{\nb \Delta_t} +  \frac{192 M_k^2}{\Delta_t \sqrt{\nb}} \left(\frac{\nb-1}{\nb}\right)^2 \right) \left( 1+ \sqrt{\frac{\delta}{2}} \right)  + \frac{4 M_k}{\nb} \left(2 + \sqrt{\delta} \right)^2,
    \end{align*}
    which simply implies positivity of all the terms.
\end{proof}
\begin{lemma}
\label{condition_Hoeffding}
Assume (\ref{lambda_t_condition}) and \A{3}. For all $\delta >0$ and for all $t \in \{1,\ldots,T\}$,
\begin{equation*}
    \sqrt{\frac{x}{2T}} \sqrt{\frac{2\lambda_t - K_{1,t}\left(\nb,M_k,\delta,\Delta_t \right) }{2\lambda_t - 8 M_k \sqrt{\frac{\delta}{\nb}}}} - \frac{K_2 \left( \nb,M_k,\delta \right)}{\Delta_t \sqrt{2\lambda_t - 8 M_k \sqrt{\frac{\delta}{\nb}}}} > 0.
\end{equation*}
\end{lemma}
\begin{proof}
Assume (\ref{lambda_t_condition}) and \A{3} implying Lemma \ref{HP_positivity} and positivity of the numerator and denominator terms below.
Let $\delta>0$ and $t \in \{1,\ldots,T\}$. 
For ease of presentation, we simplify the notations $K_{1,t}\left(\nb,M_k,\delta,\Delta_t \right)$ by $K_{1,t}$ and $K_2 \left( \nb,M_k,\delta \right)$ by $K_2$. 
\begin{align*}
    \sqrt{\delta} &= \sqrt{\delta} + \frac{K_2}{\Delta_t \sqrt{2\lambda_t - 8 M_k \sqrt{\frac{\delta}{\nb}}}} - \frac{K_2}{\Delta_t \sqrt{2\lambda_t - 8 M_k \sqrt{\frac{\delta}{\nb}}}} \\
    &\leq \sqrt{\delta} + \frac{K_2}{\underset{t=\{1,\ldots,T\}}{\min} \left\{\Delta_t \sqrt{2\lambda_t - 8 M_k \sqrt{\frac{\delta}{\nb}}} \right\}} - \frac{K_2}{\Delta_t \sqrt{2\lambda_t - 8 M_k \sqrt{\frac{\delta}{\nb}}}} \\
    & = \sqrt{\frac{2\lambda_t - 8 M_k \sqrt{\frac{\delta}{\nb}}}{2\lambda_t - K_{1,t} }} \left(\sqrt{\delta} + \frac{K_2}{\underset{t=\{1,\ldots,T\}}{\min} \left\{\Delta_t \sqrt{2\lambda_t - 8 M_k \sqrt{\frac{\delta}{\nb}}} \right\}} \right) \sqrt{\frac{2\lambda_t - K_{1,t} }{2\lambda_t - 8 M_k \sqrt{\frac{\delta}{\nb}}}}  - \frac{K_2}{\Delta_t \sqrt{2\lambda_t - 8 M_k \sqrt{\frac{\delta}{\nb}}}}.
    \end{align*}
    Therefore,
   \begin{align*}
    \sqrt{\delta}&\leq \underset{t=1,\ldots,T}{\max}  \left(\sqrt{\frac{2\lambda_t - 8 M_k \sqrt{\frac{\delta}{\nb}}}{2\lambda_t - K_{1,t} }}  \right) \left(\sqrt{\delta} + \frac{K_2}{\underset{t=\{1,\ldots,T\}}{\min} \left\{\Delta_t \sqrt{2\lambda_t - 8 M_k \sqrt{\frac{\delta}{\nb}}} \right\}} \right) \sqrt{\frac{2\lambda_t - K_{1,t} }{2\lambda_t - 8 M_k \sqrt{\frac{\delta}{\nb}}}} \\
    & \hspace{10 cm} - \frac{K_2}{\Delta_t \sqrt{2\lambda_t - 8 M_k \sqrt{\frac{\delta}{\nb}}}} \\
    & = \sqrt{\frac{x}{2T}} \sqrt{\frac{2\lambda_t - K_{1,t} }{2\lambda_t - 8 M_k \sqrt{\frac{\delta}{\nb}}}} - \frac{K_2}{\Delta_t \sqrt{2\lambda_t - 8 M_k \sqrt{\frac{\delta}{\nb}}}}.
\end{align*}
\normalsize
\end{proof}
%%%%%%%%%%%%%%%%%%%%%%
%%%%%%%%%%%%%%%%%%%%%%
\section{Proofs of the asymptotic results}\label{la_section_proof_asympt}
This section is devoted to prove asymptotic results of Sections~\ref{sec:Discussion:asym} and~\ref{sec:simplified:asym}. 
\subsection{Notations}
In this section, we use the following notations. For two  positive functions  $A(n,M_k,\delta,\Delta_t)$ and $B(n,M_k,\delta,\Delta_t)$, depending on $n$, $M_k$, $\delta$ and $\Delta_t$, we denote
$$A(n,M_k,\delta,\Delta_t)\ll B(n,M_k,\delta,\Delta_t) \quad \mbox{for}\quad \limsup_{n\to+\infty}\frac{A(n,M_k,\delta,\Delta_t)}{B(n,M_k,\delta,\Delta_t)}=0,$$
$$A(n,M_k,\delta,\Delta_t)\lesssim B(n,M_k,\delta,\Delta_t)  \quad \mbox{for}\quad\limsup_{n\to+\infty}\frac{A(n,M_k,\delta,\Delta_t)}{B(n,M_k,\delta,\Delta_t)}\leq C_1,$$
with $C_1$ not depending on $M_k$  and $\Delta_t$.
 We also denote
$A(n,M_k,\delta,\Delta_t)\gtrsim B(n,M_k,\delta,\Delta_t)$ if $A(n,M_k,\delta,\Delta_t)^{-1}\lesssim B(n,M_k,\delta,\Delta_t)^{-1}$, $A(n,M_k,\delta,\Delta_t)\gg B(n,M_k,\delta,\Delta_t)$ if $A(n,M_k,\delta,\Delta_t)^{-1}\ll B(n,M_k,\delta,\Delta_t)^{-1}$ and 
$$A(n,M_k,\delta,\Delta_t)\approx B(n,M_k,\delta,\Delta_t)$$ for $$A(n,M_k,\delta,\Delta_t)\lesssim B(n,M_k,\delta,\Delta_t) \mbox{ and } B(n,M_k,\delta,\Delta_t)\lesssim A(n,M_k,\delta,\Delta_t).$$
We simplify the notation $\Tilde{x}\left(\nb,\delta,M_k,\lambda_{1:T},\Delta_{1:T}\right)$ by $\Tilde{x}$.
 %%%%%%%%%%%%%%%
 \subsection{Proof of Proposition~\ref{asymptotic_regime_with_delta_condition}}\label{sec:proofProp1}
Remember that $\delta$ depends on $n$ with $\delta\equiv \delta(n)\to+\infty$ and $\delta(n)/n\to 0$ when $n\to+\infty$. We also recall that $M_k$ does not depend on $n$ but $\Delta_t$ satisfies Condition (\ref{lambda_t_condition}) so that it may depend on $n$.
Assume that \A{1}$\sim$\A{3} and the gap condition (\ref{A4bis}) are satisfied. We recall (see Theorem~\ref{the_theorem}) that for all $t \in \{1,\ldots,T\}$
    \begin{align*}
        K_{1,t}\left(\nb,M_k,\delta,\Delta_t \right) 
         &= 4 M_k \left(1+ 2 \frac{\nb-1}{\nb} \right) \sqrt{\frac{\delta}{\nb}} \ \notag \\
        & \hspace{-0.5 cm} + \frac{24 M_k^2}{\Delta_t} \left(\frac{\sqrt{2} \left(1 +\sqrt{2\delta} \right)}{\nb } +  \frac{4}{\sqrt{\nb}} \left(\frac{\nb-1}{\nb}\right)^2 \right) \left( 1+ \sqrt{\frac{\delta}{2}} \right)  + \frac{2 M_k}{\nb} \left(2 + \sqrt{\delta} \right)^2  \notag \\
         K_2 \left( \nb,M_k,\delta \right)  &= \frac{12 M_k^{\frac{3}{2}}}{\sqrt{2 \nb}} \left(1 + \sqrt{2\delta}\right) \left(1 + \sqrt{\frac{\delta}{2}} \right).
    \end{align*}
Therefore, since $\delta\gtrsim 1$,
\begin{align*}
K_{1,t}\left(\nb,M_k,\delta,\Delta_t \right) 
         &\approx M_k\sqrt{\frac{\delta}{\nb}} + \frac{M_k^2\sqrt{\delta}}{\Delta_t} \left(\frac{\sqrt{\delta}}{\nb } +  \frac{1}{\sqrt{\nb}}  \right)  + \frac{M_k\delta}{\nb}
\end{align*}
and
\begin{align}\label{K2}
K_2 \left( \nb,M_k,\delta \right)  &\approx  \frac{M_k^{\frac{3}{2}}\delta}{\sqrt{n}}.
\end{align}
Now, using $\delta\ll n$, $1\lesssim M_k$ and $\Delta_t\lesssim 1$, we have
\begin{align}\label{K1}
K_{1,t}\left(\nb,M_k,\delta,\Delta_t \right) 
         &\approx \frac{M_k^2\sqrt{\delta}}{\Delta_t} \left(\frac{\sqrt{\delta}}{\nb } +  \frac{1}{\sqrt{\nb}}  \right)\approx  \frac{M_k^2}{\Delta_t} \sqrt{\frac{\delta}{\nb}}.
\end{align}
Hence, condition (\ref{lambda_t_condition}), namely
$$ \lambda_t > K_{1,t}\left(\nb,M_k,\delta,\Delta_t \right),$$
becomes
\begin{equation}\label{mino1}
\lambda_t \Delta_t \gtrsim M_k^2 \sqrt{\frac{\delta}{\nb}}.
\end{equation}
Furthermore, we get
$$ \lambda_{t-1} = \sqrt{\lambda_{t-1}^2} \geq \sqrt{\lambda_t \lambda_{t-1}} \geq \sqrt{\lambda_t \Delta_t} \gtrsim M_k \left( \frac{\delta}{\nb} \right)^{\frac{1}{4}}.$$
It implies 
$$ \lambda_{t}\gg M_k \left( \frac{\delta}{\nb} \right)^{\frac{1}{2}},$$
and, for any $t\in\{1,\ldots,T\},$
$$\Delta_t \sqrt{\lambda_t - 4 M_k \sqrt{\frac{\delta}{\nb}}}= \sqrt{\Delta_t^2 \left(\lambda_t - 4 M_k \sqrt{\frac{\delta}{\nb}} \right)}\gtrsim \sqrt{\Delta_t^2 \lambda_t}.$$
Now, we have
$$\Delta_t \sqrt{\lambda_t - 4 M_k \sqrt{\frac{\delta}{\nb}}} \gtrsim \sqrt{\Delta_t^2 \frac{M_k^2}{\Delta_t} \sqrt{\frac{\delta}{\nb}}} \gtrsim M_k^2\sqrt{\frac{\delta}{\nb}} \approx \frac{M_k^{1/2}K_2 \left( \nb,M_k,\delta \right)}{\sqrt{\delta}},$$
where we have used \eqref{mino1}  two times to get inequalities and \eqref{K2} to get the approximation.
Since $M_k\gtrsim 1$, we then have that
$$\Delta_t \sqrt{\lambda_t - 4 M_k \sqrt{\frac{\delta}{\nb}}} \geq c \frac{K_2 \left( \nb,M_k,\delta \right)}{\sqrt{\delta}},$$
for $c$ a positive constant not depending on $M_k$, $\nb$, $\lambda_{1:T}$ and $\Delta_{1:T}$. 
This proves that condition (\ref{lambda_t_condition}) implies condition (\ref{K_3_simplified}). \\
Lastly, Corollary~\ref{bound_simple} gives an approximation of the quantile to be taken as
$$ \qmajbis = 2\left(1+c^{-1}\right)^2 T \delta  \underset{t=1,\ldots,T}{\max}  \left(\frac{\lambda_t - 4 M_k \sqrt{\frac{\delta}{\nb}}}{\lambda_t - K_{1,t}\left(\nb,M_k,\delta,\Delta_t \right) }  \right). $$
Using \eqref{K1}, we obtain
$$ \qmajbis= 2\left(1+c^{-1}\right)^2 T \delta  \underset{t=1,\ldots,T}{\max}  \left(\frac{\lambda_t - 4 M_k \sqrt{\frac{\delta}{\nb}}}{\lambda_t - \frac{\kappa M_k^2}{\Delta_t}\sqrt{\frac{\delta}{\nb}}} \right),$$ for a bounded constant $\kappa >0$ independent of $M_k$, $\nb$, $\lambda_{1:T}$ and $\Delta_{1:T}$, which concludes the proof of Proposition~\ref{asymptotic_regime_with_delta_condition}.

\section{Additional figures}
\label{sec:additional_figures}

\begin{figure}
	\begin{center}
		\includegraphics[scale=0.7]{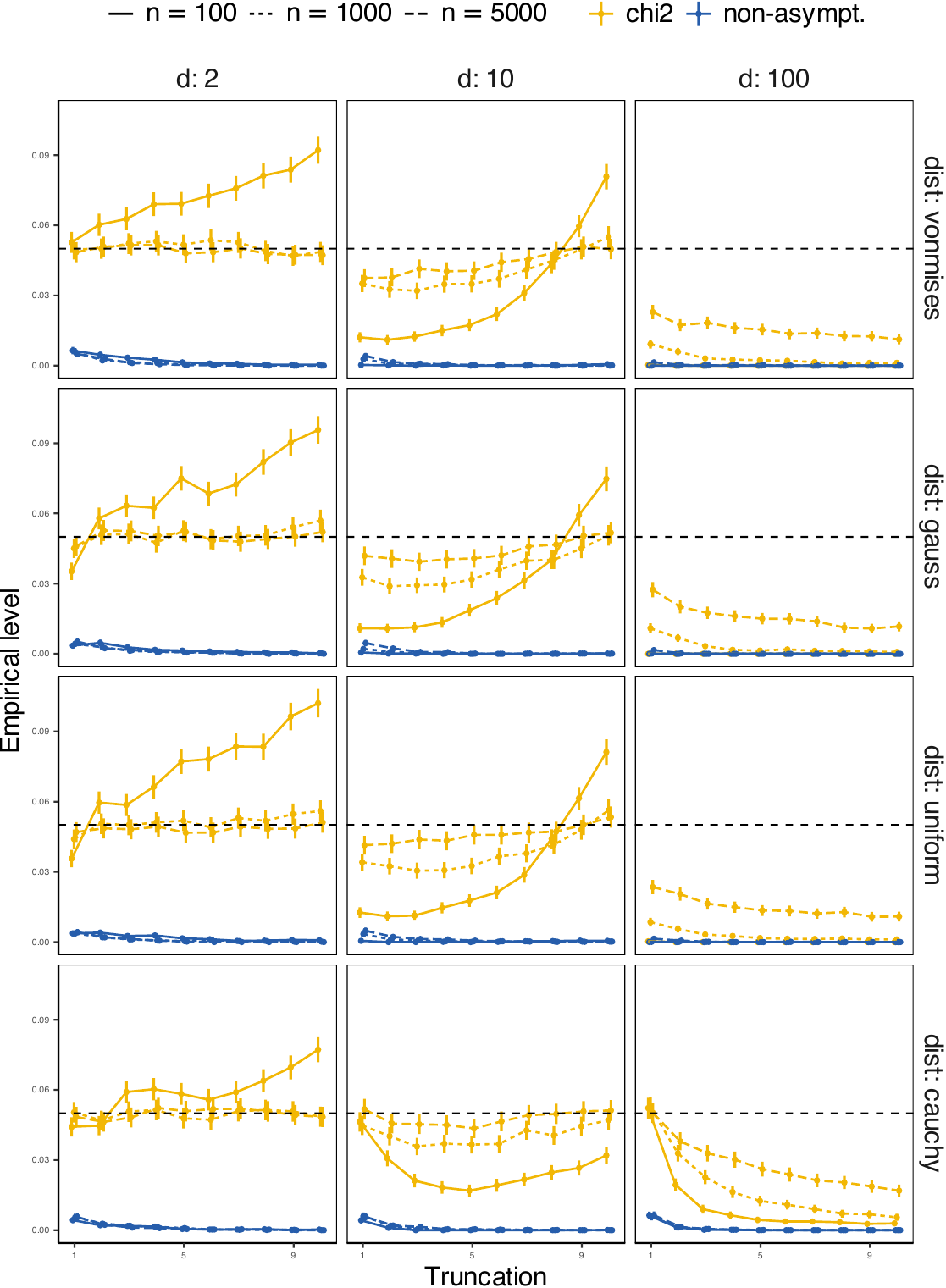}
		\caption{Average empirical level (and 95\% confidence interval) of the st-nMMD test, for varying truncations $T$, for the test based on the asymptotic $\chi^2$ approximation and our non-asymptotic bound, using the $\max$ to compute the quantile (instead of the mean). The test is performed at a nominal level $\alpha=0.05$ (black dashed horizontal line), for 4 different distributions and 3 different dimensions $d \in \{2, 10,100\}$. \label{Fig:level_simulations_T_max_005}}
	\end{center}
\end{figure}

\begin{figure}
	\begin{center}
		\includegraphics[scale=0.7]{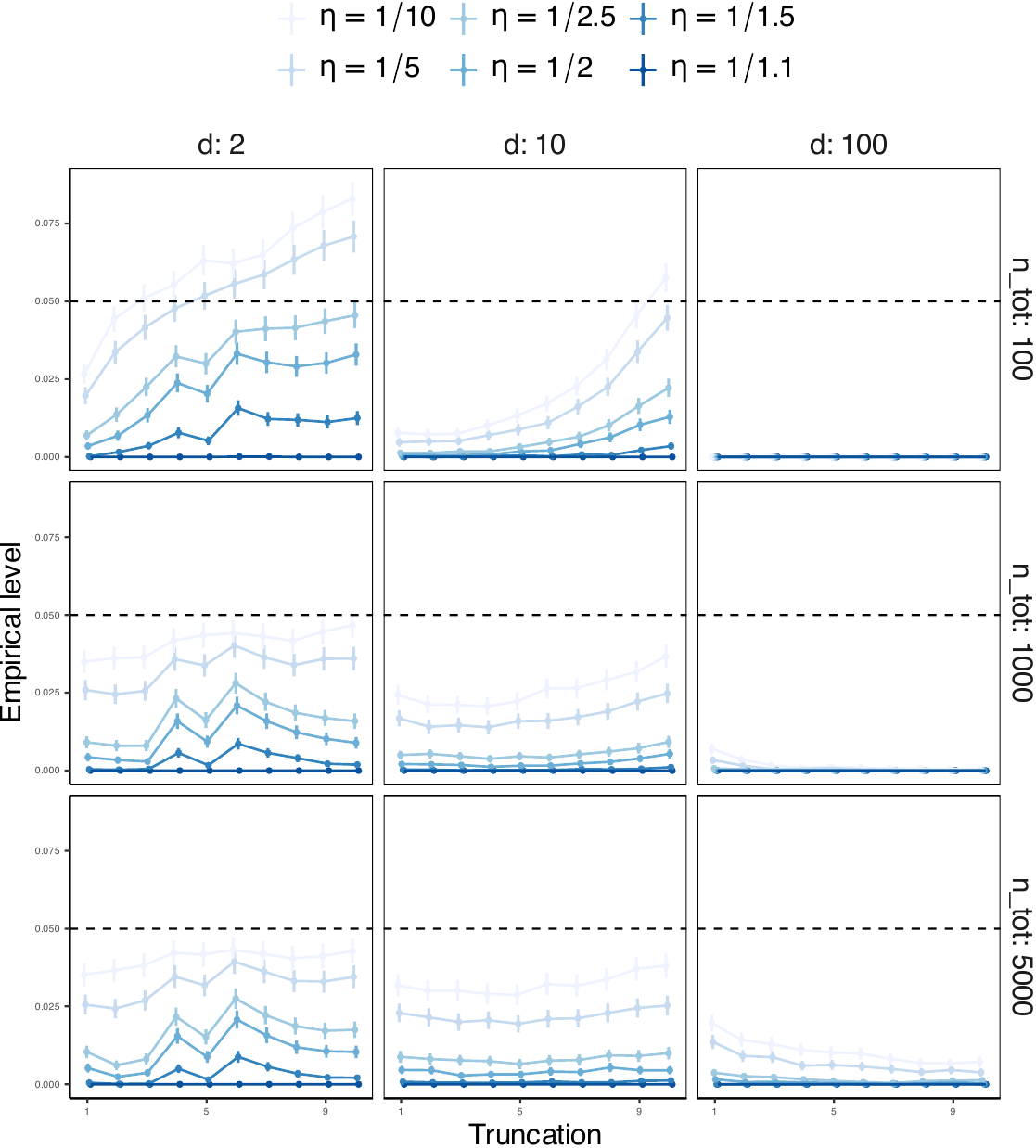}
	\end{center}
	\caption{Average empirical level (and 95\% confidence interval) of the st-nMMD test, for varying truncations $T$, for the test based on the our non-asymptotic bound, with varying hyperparameter $\eta$. The test is performed at a nominal level $\alpha=0.05$ (black dashed horizontal line), for the Gaussian distribution with 3 different dimensions $d \in \{2, 10,100\}$. \label{Fig:level_simulations_T_eta_005}}
\end{figure}
%%%%%%%%%%%%%%%%%%%%%
%%%%%%%%%%%%%%%%%%%%%

\newpage
\section{Existing results}\label{existing_results}
%%%%%%%%%%%
\subsubsection*{Hoeffding and McDiarmid's concentration inequalities}
\begin{theorem}[\bf{Hoeffding's inequality} \cite{hoeffding1994probability}]
\label{Hoeffding_inequality}
Let $n$ be an integer and $X_1,\ldots,X_n$ be a sequence of independent random variables such that $\P\left( a_i \leq X_i \leq b_i \right) =1$ for $\left(a_i\right)_{i \in \{1,\ldots,n\}}$ and $\left(b_i\right)_{i \in \{1,\ldots,n\}}$ two sequences of real values such that $a_i < b_i$. Let us define $S_n = X_1 + \ldots + X_n$. For all $t>0$ 
\begin{equation*}
    \P \left( S_n - \mathbb{E}[S_n] > t \right) \leq e^{-\frac{2 t^2}{\underset{i=1}{\overset{n}{\sum}} \left(b_i - a_i \right)^2}}.
\end{equation*}
\end{theorem}
\begin{definition}[\bf{Bounded differences property}]
\label{def_bounded_difference_property}
Let $n$ be an integer and $\mathcal{Z}_1,\ldots,\mathcal{Z}_n$ be some measurable spaces. A function $q: [\mathcal{Z}_1 \times \ldots \times \mathcal{Z}_n \longrightarrow \R]$ satisfies the bounded differences property if there exist positive constants $c_1,\ldots,c_n$ such that for all $i \in \{1,\ldots,n\}$ and for all $\left(z_1,\ldots,z_n\right) \in \left(\mathcal{Z}_1, \ldots, \mathcal{Z}_n\right)$ 
\begin{equation*}
    \underset{x_i' \in \mathcal{Z}_i}{\sup} \left | q\left(z_1,\ldots,z_{i-1},z_i',z_{i+1},\ldots,z_n\right) - q\left(z_1,\ldots,z_n\right) \right | \leq c_i.
\end{equation*}
\end{definition}
\begin{theorem}[\bf{McDiarmid's inequality} \cite{mcdiarmid1989method}]
    Let $n$ be an integer and $\mathcal{Z}_1,\ldots,\mathcal{Z}_n$ be some measurable spaces. Let $q : [\mathcal{Z}_1 \times \ldots \times \mathcal{Z}_n \longrightarrow \R]$ be a function satisfying the bounded differences property with bounds $c_1,\ldots,c_n$ and let $Z_1, \ldots, Z_n$ be independent random variables where $Z_i \in \mathcal{Z}_i$ for all $i \in \{1,\ldots,n\}$. \\
    Then, $\forall \varepsilon >0$ 
    \begin{equation*}
        \PP \left( q\left(Z_1,\ldots,Z_n\right) - \EE \left[q\left(Z_1,\ldots,Z_n\right) \right] > \varepsilon \right) \leq e^{- \frac{2 \varepsilon^2}{\underset{i=1}{\overset{n}{\sum}} c_i^2}}.
    \end{equation*}
    and
    \begin{equation*}
        \PP \left( q\left(Z_1,\ldots,Z_n\right) - \EE \left[q\left(Z_1,\ldots,Z_n\right) \right] < -\varepsilon \right) \leq e^{- \frac{2 \varepsilon^2}{\underset{i=1}{\overset{n}{\sum}} c_i^2}}.
    \end{equation*}
    \label{McDiarmid's theo}
\end{theorem}
%%%%%%%%%%%%%%
\subsubsection*{Maximum Mean Discrepancy's control} 
Here is the classical exponential bound of the maximum mean discrepancy metric defined by $\left \| \widehat{\mu}_X - \widehat{\mu}_Y \right \|_{\Hk}$ and that we used in our proof. It is valid whatever the population sizes $n_X$ and $n_Y$ (without assumption \A{1}). 
\begin{theorem}[\bf{Theorem 7. of \citep{gretton2012kernel}}]
Assume \A{2}. Then, for all $\delta>0$, 
\begin{equation*}
     \PPO \left( \left \| \widehat{\mu}_X - \widehat{\mu}_Y \right \|_{\Hk} \leq 2 \left( \sqrt{\frac{M_k}{n_X}} + \sqrt{\frac{M_k}{n_Y}} + \sqrt{\frac{M_k (n_X+n_Y) \delta}{2 n_X n_Y}} \right) \right) \geq 1 - 2 e^{-\delta}.
\end{equation*}
\label{Gretton}
\end{theorem}
%%%%%%%%%%%%%%%%
\subsubsection*{Operator perturbation theory} 
Perturbation theory is one of the main tool to derive our exponential bound. In particular, we were strongly inspired by \citep{koltchinskii2000random,blanchard2007statistical,zwald2005convergence}'s works developing the perturbation theory on the variance-covariance operator for kernel principal components analysis. 
In our proof, we adapted the following results to $\ST$ and $\hST$.

For the following results, we suppose that $n$ is an integer and $Z_1, \ldots, Z_n$ are $n$ independent random variables taking theirs values in a measurable space $\mathcal{Z}$ and following the distribution $\mathbb{P}$. Let $\phi$ be the feature mapping function from $\mathcal{Z}$ to a reproducing kernel Hilbert space $\Hk$ associated to the kernel function $k$ such that $\phi(z) = k(z,\cdot)$ for all $z \in \mathcal{Z}$. 
%\begin{theorem}[\bf{Lemma 1. of \citep{zwald2005convergence} and Corollary 5. of \citep{shawe2003estimating}}]
\begin{theorem}
	\bf{(Lemma 1. of \cite{zwald2005convergence} and Corollary 5. of \cite{shawe2003estimating}).\rm} 
	\normalfont
	Let us define the non-centered covariance operator $C \in \HSk$ and the non-centered empirical covariance operator $C_n \in \HSk$ respectively by 
\begin{align*}
    C &= \EE[\phi(Z_1) \otimes_{\Hk} \phi(Z_1)], \\
    C_n &= \frac{1}{n} \underset{i=1}{\overset{n}{\sum}} \phi(Z_i) \otimes_{\Hk} \phi(Z_i).
\end{align*}
Assume \A{2} and the simplicity of the eigenvalues of $C$. 
Then, for all $\delta>0$ 
\begin{equation*}
     \PPO \left( \left \| C_n - C \right \|_{\HSk} \leq \frac{2 M_k}{\sqrt{n}} \left( 1 + \sqrt{\frac{\delta}{2}}\right) \right) \geq 1 - e^{-\delta}.
\end{equation*}
\label{Blanchard_non_centered_C}
\end{theorem}
\begin{remark}[\bf{Comment about the Centered Case in \cite{zwald2005convergence}}]
    Corollary 5. of \citep{shawe2003estimating} can be generalized to the centered covariance operator and a similar result holds up to some additional constant factors. In our proof, we need to specify these constants (see Lemma \ref{Blanchard_centered_C}). 
\end{remark}
%\begin{theorem}[\bf{Theorem 2. of \citep{zwald2005convergence} and Lemma 5.2. of \citep{koltchinskii2000random}}]
\begin{theorem}
	\bf{(Theorem 2. of \cite{zwald2005convergence} and Lemma 5.2. of \cite{koltchinskii2000random}). \rm} 
	\normalfont
Let $A \in \HSk$ be a symmetric positive operator with simple positive eigenvalues $\alpha_1 > \alpha_2 > \ldots$. For an integer $r$ such that $\alpha_r >0$, let $\delta_r = \frac{1}{2}\min(\alpha_r - \alpha_{r-1}, \alpha_{r-1} - \alpha_{r-2})$.
Let $B \in \HSk$ be another symmetric operator such that $\|B\|_{\HSk} < \frac{\delta_r}{2}$ and $(A+B)$ is still a positive operator with simple nonzero eigenvalues. 
Then, the orthogonal projector $\Pi_{g_r}$ onto the one-dimensional subspace of $\Hk$ spanned by the $r^{th}$ eigenfunction $g_r$ of $A$ satisfies 
$$ \|\Pi_{g_r}(A) - \Pi_{g_r}(A+B) \|_{\HSk} \leq \frac{2 \|B\|_{\HSk}}{\delta_r}.$$
\label{Zwald_Blanchard_Theo}
\end{theorem}
\begin{proposition}
	\bf{Remark about the Approximation Error of the Eigenvectors of \cite{zwald2005convergence} \rm}
	\normalfont
Under assumptions of Theorem \ref{Zwald_Blanchard_Theo}, if $g_r$ and $h_r$ denotes the $r-$th eigenvectors of $A$ and $(A+B)$ respectively and if $\langle g_r, h_r \rangle_{\Hk} >0$, then 
$$ \left \| g_r - h_r \right \|_{\Hk} \leq \left \| \Pi_{g_r} - \Pi_{h_r} \right \|_{\HSk}. $$
\label{Zwald_Blanchard_Prop}
\end{proposition}
\begin{proposition}
	\bf{The Hoffmann-Wielandt inequality \cite{bhatia1994hoffman} in infinite dimensional space.\rm}
	\normalfont
	(see the proof of Theorem 3. of \cite{zwald2005convergence}). 
\label{Hoffmann_inequality}
Under assumptions of Theorem \ref{Zwald_Blanchard_Theo}, if  $\beta_1 > \beta_2 > \ldots$ denotes the eigenvalues of $(A+B)$, then for each $i >0$ 
\begin{equation*}
	|\alpha_i-\beta_i| \leq  \|B\|_{\HSk}.
\end{equation*}
\end{proposition}

\end{document}